\documentclass[12pt,reqno]{amsart}
\usepackage[paper=letterpaper,margin=1in]{geometry}

\usepackage{amsmath,amssymb,mathrsfs,amscd,amsthm}
\usepackage{braket}
\usepackage{mathtools} 
\usepackage[inline]{enumitem} 
\usepackage{hyperref}
\hypersetup{%
	colorlinks=true, linkcolor=blue, 
	citecolor=green
}

\newtheorem{theorem}{Theorem}[section]
\newtheorem{lemma}[theorem]{Lemma}
\newtheorem{proposition}[theorem]{Proposition}
\newtheorem{corollary}[theorem]{Corollary}

\theoremstyle{definition}
\newtheorem{remark}[theorem]{Remark}

\numberwithin{equation}{section}
\allowdisplaybreaks

\usepackage{acronym}
\acrodef{SDE}[SDE]{Stochastic Differential Equation}
\acrodef{DBM}[DBM]{Dyson's Brownian Motion}

\newcommand{\BK}[1]{ {\left( #1 \right)} }
\newcommand{\sqBK}[1]{ {\left[ #1 \right]} }
\newcommand{\curBK}[1]{ {\left\{ #1 \right\}} }
\newcommand{\absBK}[1]{ {\left| #1 \right|} }

\newcommand{\ceilBK}[1]{ {\lceil #1 \rceil} }
\newcommand{\floorBK}[1]{ {\lfloor #1 \rfloor} }

\newcommand{\bdy}{\lambda} 
\newcommand{\I}{\mathcal{L}} 

\DeclareMathOperator{\sign}{sign}
\newcommand{\sine}{\mathcal{S}_{\text{sine}}}

\newcommand{\ind}{\mathbf{1}}
\newcommand{\av}{\mathrlap{\rotatebox[origin=c]{20}{\bf--}}{\Sigma}}

\newcommand{\norm}[1]{|#1|_{\alpha,\rho}} 

\newcommand{\Xsp}{\mathcal{X}}
\newcommand{\XTsp}{\mathcal{X}_{\mathcal{T}}}
\newcommand{\XRsp}{\mathcal{X}^\text{rg}}
\newcommand{\XRTsp}{\mathcal{X}_{\mathcal{T}}^\text{rg}}

\newcommand{\Ysp}{\mathcal{Y}}
\newcommand{\Yspl}{\underline{\mathcal{Y}}}

\newcommand{\YTsp}{\mathcal{Y}_{\mathcal{T}}}
\newcommand{\YTspl}{\underline{\mathcal{Y}}_{\mathcal{T}}}
\newcommand{\YTspll}{\underline{\mathcal{Y}}_{\mathcal{T}}'}

\newcommand{\Rsp}{\mathcal{R}}
\newcommand{\RTsp}{\mathcal{R_T}}

\newcommand{\KTsp}{\mathcal{KT}} 

\newcommand{\Weyl}{\mathcal{W}} 
\newcommand{\WeylT}{\mathcal{W_T}} 

\newcommand{\A}{\mathcal{A}} 
\newcommand{\AL}{\mathcal{A}^{\hZ}} 
\newcommand{\tilA}{\widetilde{\mathcal{A}}} 
\newcommand{\hZ}{\mathbb{L}} 

\usepackage{graphicx}
\newcommand*{\Cdot}{{\raisebox{-0.5ex}{\scalebox{1.8}{$\cdot$}}}}

\newcommand{\f}{ \psi }
\newcommand{\fS}{ \eta }
\newcommand{\undffp}{ \widetilde{\eta}^{\text{lw},+} }
\newcommand{\undffm}{ \widetilde{\eta}^{\text{lw},-} }
\newcommand{\undffpm}{ \widetilde{\eta}^{\text{lw},\pm} }

\newcommand{\undffs}{ \widetilde{\eta}^{\text{lw},\sigma} }
\newcommand{\undf}{ \eta^{\text{lw}} }
\newcommand{\undfp}{ \eta^{\text{lw},+} }
\newcommand{\undfm}{ \eta^{\text{lw},-} }
\newcommand{\undfpm}{ \eta^{\text{lw},\pm} }
\newcommand{\undfmp}{ \eta^{\text{lw},\mp} }
\newcommand{\barf}{ \eta^{\text{up}} }
\newcommand{\barE}[1]{ \overline{E_{#1}} } 

\newcommand{\bfYgn}{ \mathbf{Y}^{\vee\gamma_n} }
\newcommand{\bfYg}{ \mathbf{Y}^{\vee\gamma} }

\newcommand{\YR}[2]{Y^{^\star #1}_{#2}}
\newcommand{\barYR}[2]{\overline{Y^{^\star #1}_{#2}}}

\newcommand{\bfYR}[1]{\mathbf{Y}^{^\star #1}}

\newcommand{\bfG}{\mathbf{G}}

\newcommand{\fup}{ f^{\text{up}} }
\newcommand{\flw}{ f^{\text{lw}} }

\newcommand{\yup}{ y^{\text{up}} }
\newcommand{\Yup}{ Y^{\text{up}} }
\newcommand{\bfYup}{ \mathbf{Y}^{\text{up}} }
\newcommand{\ylw}{ y^{\text{lw}} }
\newcommand{\Ylw}{ Y^{\text{lw}} }

\newcommand{\undYlw}{ \underline{Y}^{\text{lw}} }
\newcommand{\bfYlw}{ \mathbf{Y}^{\text{lw}} }

\newcommand{\bfundYlw}{ \underline{\mathbf{Y}}^\text{lw} }

\newcommand{\Zup}{ Z^{\text{up}} }
\newcommand{\bfZup}{ \mathbf{Z}^{\text{up}} }
\newcommand{\Zlw}{ Z^{\text{lw}} }
\newcommand{\bfZlw}{ \mathbf{Z}^{\text{lw}} }

\newcommand{\hatB}{ \widehat{B} }

\newcommand{\tilm}{\widetilde{m}}
\newcommand{\tilu}{\widetilde{u}}

\newcommand{\tilphi}{\widetilde{\phi}}
\newcommand{\tilL}{\widetilde{L}}
\newcommand{\tilB}{\widetilde{B}}
\newcommand{\tilQ}{\widetilde{Q}}
\newcommand{\tilT}{\widetilde{T}}

\newcommand{\tilY}{\widetilde{Y}}

\newcommand{\barY}{ \overline{Y} }

\newcommand{\undy}{ \underline{y} }

\newcommand{\undY}{ \underline{Y} }

\newcommand{\bftilY}{ \widetilde{\mathbf{Y}} }

\newcommand{\bfbarY}{ \overline{\mathbf{Y}} }

\newcommand{\bfundY}{ \underline{\mathbf{Y}} }
\newcommand{\bfundZ}{ \underline{\mathbf{Z}} }

\newcommand{\bfgamma}{ \boldsymbol{\gamma} }

\newcommand{\bfx}{ \mathbf{x} }
\newcommand{\bfy}{ \mathbf{y} }
\newcommand{\bfz}{ \mathbf{z} }
\newcommand{\bfB}{ \mathbf{B} }
\newcommand{\bfQ}{ \mathbf{Q} }

\newcommand{\bfW}{ \mathbf{W} }
\newcommand{\bfX}{ \mathbf{X} }
\newcommand{\bfY}{ \mathbf{Y} }
\newcommand{\bfZ}{ \mathbf{Z} }

\newcommand{\bbE}{ \mathbb{E} }
\newcommand{\bbK}{ \mathbb{K} }
\newcommand{\bbP}{ \mathbb{P} }
\newcommand{\bbQ}{ \mathbb{Q} }
\newcommand{\bbR}{ \mathbb{R} }
\newcommand{\bbZ}{ \mathbb{Z} }

\newcommand{\calA}{ \mathcal{A} }

\newcommand{\calI}{ \mathcal{I} }

\newcommand{\calK}{ \mathcal{K} }
\newcommand{\calL}{ \mathcal{L} }

\newcommand{\calO}{ \mathcal{O} }

\newcommand{\calS}{ \mathcal{S} }
\newcommand{\calT}{ \mathcal{T} }
\newcommand{\calU}{ \mathcal{U} }

\newcommand{\fkA}{ \mathfrak{A} }
\newcommand{\fkB}{ \mathfrak{B} }

\newcommand{\fkG}{ \mathfrak{G} }

\newcommand{\fkI}{ \mathfrak{I} }
\newcommand{\fkJ}{ \mathfrak{J} }

\newcommand{\fkL}{ \mathfrak{L} }

\newcommand{\fkN}{ \mathfrak{N} }

\newcommand{\ic}{\text{in}}

\newcommand{\Wfil}{ \mathscr{F}^\mathbf{W} }
\newcommand{\Bfil}{ \mathscr{F}^\mathbf{B} }
\newcommand{\scrG}{ \mathscr{G} }

\begin{document}

\title[Infinite Dimensional SDE for Dyson's Model]
{Infinite Dimensional Stochastic Differential Equations for Dyson's Model}

\address{L.-C.\ Tsai, Department of Mathematics, Stanford University, California 94305}
\email{lctsai@stanford.edu}
\author[L.-C.\ Tsai]{Li-Cheng Tsai}

\date{\today}

\thanks{}

\begin{abstract}
In this paper we show the strong existence and the pathwise uniqueness
of an infinite-dimensional \ac{SDE} corresponding to the bulk limit of \ac{DBM},
for all $\beta\geq 1$.
Our construction applies to an explicit and general class of initial conditions,
including the lattice configuration $\{x_i\}=\bbZ$ and the sine process.
We further show the convergence of the finite to infinite-dimensional \ac{SDE}.
This convergence concludes the determinantal formula of \cite{katori10}
for the solution of this \ac{SDE} at $\beta=2$.
\end{abstract}

\maketitle

\section{Introduction}
\label{sect:intro}

In this paper we study the well-posedness of the infinite-dimensional \ac{SDE},
\begin{align}\label{eq:DBM}
	X_i(t) = X_i(0) + B_i(t) 
	+ 
	\beta \int_0^t 
	\phi_i(\bfX(s)) ds,
	\ i\in\bbZ,
\end{align}
where $\bfX(s)=( \ldots<X_0(s)<X_1(s)<\ldots )$ describes ordered particles on $ \bbR $,
 $ B_i(t) $, $ i\in\bbZ $, denote independent standard Brownian motions,
and the interaction $ \phi_i(\bfx) $ takes the form
\begin{align}\label{eq:phi}
	\phi_i(\bfx) := \frac12 \lim_{k\to\infty}\sum_{j:|j-i|\leq k} \frac{1}{x_i-x_j},
\end{align}
with $ \beta\geq 1 $ measuring its strength.
The interest of such \ac{SDE} arises from random matrix theory.
Equation \eqref{eq:DBM} represents the bulk limit of \ac{DBM}, which describes the evolution of the eigenvalues 
of the symmetric and Hermitian random matrices
with independent Brownian entries, for $ \beta=1,2 $, respectively, see \cite{dyson62,mckean69}.

The difficulty of establishing the well-posedness of \eqref{eq:DBM}
lies in the long-range and singular nature of $ \phi_i $. 
Indeed, for a particle configuration $ \bfx $ with a roughly uniform density,
we have  
\begin{align}\label{eq:diverge}
	\sum_{j:j\neq i} \frac{1}{|x_i-x_j|} = \infty,
\end{align}
so the only way \eqref{eq:phi} converges is by canceling
two divergent series from $ j<i $ and $ j>i $.
Further, as we argue in Remark~\ref{rmk:singular} in the following,
unlike the case of finite dimensions, the Bessel-type repulsion of $ \phi_i $ alone
does \emph{not} prevent finite time collisions, i.e.\ $ X_i(t)=X_{i+1}(t) $.
Alternatively, under the framework of \cite{lang77,lang77a},
equation \eqref{eq:DBM} formally has the logarithmic potential $ -\beta \sum_{i<j} \log|x_i-x_j| $.
However, due the logarithmic growth as $ |x_i-x_j| \to \infty $,
such a potential is still \emph{ill-defined} even under a limiting procedure as in \eqref{eq:phi},
suggesting a considerable challenge for establishing the well-posedness of \eqref{eq:DBM}.

At $ \beta=1,2,4 $, this challenge has been largely overcome thanks to the integrable structure of \ac{DBM}.
This starts with \cite{spohn87} constructing the equilibrium process as an $ L^2 $ Markovian semigroup.
Combining the theory of Dirichlet form and the theory of determinantal or Pfaffian point processes,
\cite{osada12,osada13} obtain the weak existence for near-equilibrium configurations.
The recent work of \cite{osada14a} further shows the strong existence and pathwise uniqueness at equilibrium.
In a different direction,
\cite{katori10} constructs infinite-dimensional \ac{DBM} as a determinantal (in spacetime) point process,
for general, out-of-equilibrium, configurations at $ \beta=2 $.
This construction, as a point process, is not directly related to solutions of the \ac{SDE} \eqref{eq:DBM}.

In this paper, we attack the problem, for all $ \beta\geq 1 $, \emph{without} referring to the integrable structure,
whereby establishing the strong existence and pathwise uniqueness of \eqref{eq:DBM} (see Theorem~\ref{thm:DBM}).
As our techniques do not refer to a specific equilibrium measure,
Theorem~\ref{thm:DBM} holds for an \emph{explicit}, \emph{out-of-equilibrium} configuration space $\XRsp(\alpha,\rho,p)$,
which, loosely speaking, consists of particle configurations with a roughly uniform density $ \rho^{-1}>0 $.
In particular, the space includes the lattice configuration $\{x_i\}=\bbZ$
and the sine process (see Lemma~\ref{lem:sine}).
For infinite-dimensional interacting diffusions with $ C^3_0 $ potentials,
an out-of-equilibrium result is first established in \cite{fritz87}.
With the logarithmic potential,
Theorem~\ref{thm:DBM} is the first out-of-equilibrium result on well-posedness.

Further, by establishing a finite-to-infinite-dimensional convergence,
in Corollary~\ref{cor:det} we show that the determinantal point process
constructed in \cite{katori10} coincides with the unique strong solution given by Theorem~\ref{thm:DBM},
for a class of out-of-equilibrium configurations.
This has also been obtained in the recent work of \cite[Theorem 2.2]{osada14b}
for the equilibrium process at $ \beta=2 $.

The main idea here is to use the \emph{monotonicity} of the gap process
$ \{Y_a(t)\}_{a\in\hZ} $, where $ Y_a(t):=X_{a+1/2}(t)-X_{a-1/2}(t) $ and $ \hZ:=\frac12+\bbZ $,
based on a certain simple observation of $ \phi_i $.
Such monotonicity allows us to conveniently identify
the long-range and singular effect of $ \phi_i $ on $ \bfX $.
Although the techniques employed in this paper are standard,
they are applicable only in a careful setup that captures the monotonicity.
This monotonicity of $ \{Y_a(t)\}_{a\in\hZ} $ is new, 
and in particular differs from that of \cite[Lemma~4.3.6]{anderson10}.

\begin{remark}
For $ \beta\geq 1 $,
we shows that particles stay strictly ordered, $ X_i(t)<X_{i+1}(t) $,
for all time, almost surely.
For $ \beta\in(0,1) $, however, one expects finite time collisions to occur.
Due to this fact,
proving well-posedness, even in finite dimensions, requires extra effort (see \cite{cepa1995}).
We do not pursuit the case $ \beta\in(0,1) $ here. 
\end{remark}

Besides the bulk limit of \ac{DBM} \eqref{eq:DBM} considered here,
the edge limit is also a related subject of interest.
The interest lies in random matrix theory and
the Kardar--Parisi--Zhang universality class (see \cite{corwin12}).
Based on the aforementioned theory of Dirichlet form and determinantal point processes,
\cite{osada13a,osada14} obtain well-posedness results of the corresponding \ac{SDE},
and \cite{katori09} constructs the corresponding determinantal point process.
A multilayer generalization of \ac{DBM}, the corner process, 
is studied in \cite{gorin14a,gorin14} and the references therein.
In \cite{corwin14,corwin15},
the notion of Brownian--Gibbs property is introduced to characterize the edge limit as a line ensemble,
and is further generalized to the corresponding property for the Kardar--Parisi--Zhang equation.

\subsection{Definitions and Statement of the Results}
We begin by defining the spaces $\Xsp(\alpha,\rho)$ and $\XRsp(\alpha,\rho,p)$.
This is done by considering their corresponding \emph{gap configurations}.
More explicitly,
let $\Weyl:=\{\bfx\in\bbR^\bbZ: x_i<x_{i+1},\forall i\in\bbZ\}$
denote the Weyl chamber (of particle configurations),
and let $u$ denote the map into gap configurations:
\begin{align}\label{eq:u}
	& u: \Weyl \longrightarrow (0,\infty)^\hZ,
	\quad
	\hZ:=\tfrac12+\bbZ,&
	&
	u(\bfx) := (x_{a+1/2}-x_{a-1/2})_{a\in\hZ},
\end{align}
which is made bijective by augmenting the zeroth particle coordinate, as 
\begin{align}\label{eq:tilu}
	\tilu: \ \Weyl
	\xrightarrow{\text{bijective}} \bbR \times (0,\infty)^\hZ,
	\quad
	\tilu(\bfx) := (x_0,u(\bfx)).
\end{align}
For $ \alpha\in(0,1) $ and $ \rho>0 $, 
we consider the following space of gap configurations 
\begin{align}
	&
	\label{eq:Ysp}
	\Ysp(\alpha,\rho) := \curBK{ \bfy\in(0,\infty)^\hZ : \norm{\bfy} < \infty },
\\
	&
	\label{eq:Ynorm}
	\norm{\bfy}
	:=
	\sup_{m\in\bbZ\setminus\{0\}} 
	\big\{
		\absBK{ \av_{(0,m)}(\bfy)-\rho} \ |m|^{\alpha}
	\big\},
\end{align}
where $ \av_\calI(\bfy) $ denotes the average over a generic finite set $\calI$:
\begin{align}\label{eq:av}
	\av^p_\calI(\bfy) := |\calI|^{-1} \sum_{a\in\calI} (y_a)^p,
	\quad
	\av_\calI(\bfy) := \av^1_\calI(\bfy),
\end{align}
with the convention $(i,j]=[j,i)$ (and similarly for $(i,j)$, $[i,j]$, etc)
and $ \av^p_{\emptyset}(\bfy):=0 $.
We \emph{define} $ \Xsp(\alpha,\rho) := u^{-1}(\Ysp(\alpha,\rho)) $.
That is, $ \Xsp(\alpha,\rho) $ consists of particle configurations
whose corresponding gap processes satisfy \eqref{eq:Ysp}.
Similarly, for $ p>1 $, 
we define $ \XRsp(\alpha,\rho,p) :=  u^{-1}(\Ysp(\alpha,\rho)\cap\Rsp(p)) $, where
\begin{align}
	&
	\label{eq:Rsp}
	\Rsp(p) 
	:=
	\Big\{ 
		\bfy\in(0,\infty)^\hZ:
		\sup_{m\in\bbZ} \av^p_{(0,m)}(\bfy) <\infty
	\Big\}.
\end{align}

We proceed to defining the process-valued analogs of $\Xsp(\alpha,\rho)$ and $\XRsp(\alpha,\rho,p)$.
To simply notations,
we often use $ \bfx $ and $ \bfy $, instead of $ \bfx(\Cdot) $ and $ \bfy(\Cdot) $, 
to denote processes.
Let $\WeylT:=\{\bfx\in C([0,\infty))^\bbZ : \bfx(t)\in\Weyl,\forall t\geq 0\}$
denote the process-valued analog of $ \Weyl $.
By abuse of notation, we let $u$ and $ \tilu $ act on $\WeylT$ 
by $u(\bfx)(t) := u(\bfx(t))$ and by $\tilu(\bfx)(t) := \tilu(\bfx(t))$.
With $ \YTsp(\alpha,\rho) $ and $ \RTsp(p) $ denoting the analogs of
$ \Ysp(\alpha,\rho) $ and $ \Rsp(p) $ as follows
\begin{align}
	&
	\label{eq:YTsp}
	\YTsp(\alpha,\rho) 
	:= 
	\Big\{ 
		\bfy \in C_+([0,\infty))^\hZ 
		: 
		\sup_{s\in[0,t]}\norm{\bfy(s)} < \infty, \ \forall t\geq 0 
	\Big\},
\\
	&
	\label{eq:RTsp}
	\RTsp(p)
	:=
	\Big\{
		\bfy \in C_+([0,\infty))^\hZ :
		\sup_{s\in[0,t],m\in\bbZ} \av^p_{(0,m)}(\bfy)<\infty,
		\forall t \geq 0
	\Big\},
\\
	&
	\label{eq:Csp}
	\quad \text{ where } C_+([0,\infty)) := \{ y\in C([0,\infty)): y(t)>0,\forall t\geq 0 \},
\end{align}
we define
$
	\XTsp(\alpha,\rho) := \tilu^{-1} (C([0,\infty))\times\YTsp(\alpha,\rho))
$
and
$
	\XRTsp(\alpha,\rho,p) := \tilu^{-1} (C([0,\infty))\times(\YTsp(\alpha,\rho)\cap\Rsp(p))).
$

Recall from \cite[Definition 5.2.1, 5.3.2]{karatzas91} 
the notions of strong solutions and pathwise uniqueness of \ac{SDE},
which are readily generalized to infinite dimensions here.
Let $ \bfB(t) := (B_i(t))_{i\in\bbZ} $ denote the driving Brownian motion,
with the canonical filtration $\Bfil_t:=\sigma(\bfB(s):s\in[0,t])$.
\emph{Hereafter, we fix $ \beta \geq 1 $, $ \alpha\in(0,1) $, $ \rho>0 $ and $ p>1 $ unless otherwise stated.}
The following is our main result.
\begin{theorem}\label{thm:DBM}
Given any $\bfx^\ic\in\Xsp(\alpha,\rho)$,
there exists an $\XTsp(\alpha,\rho)$-valued, $\Bfil$-adapted solution $\bfX$ of \eqref{eq:DBM}
starting from $\bfx^\ic$.
If, in addition, $\bfx^\ic\in\XRsp(\alpha,\rho,p)$, 
this solution $\bfX$ takes value in $\XRTsp(\alpha,\rho,p)$,
and is the unique $\XRTsp(\alpha,\rho,p)$-valued solution in the pathwise sense.
\end{theorem}

\begin{remark}\label{rmk:Xsp}
For any $ \bfx\in\XTsp(\alpha,\rho) $, one easily verifies that
$
	(\lim_{k\to\infty} \sum_{j:|i-j|\leq k} \frac{1}{x_i(t)-x_j(t)})
$
converges uniformly in $ t\in[0,t'] $,
for any fixed $ i\in\bbZ $ and $ t'<\infty $.
Further, the limit $ \phi_i(\bfx(t)) $ 
takes values in $ L^\infty_\text{loc}([0,\infty)) $,
so in particular the r.h.s.\ of \eqref{eq:DBM} is well-defined
for $ \XTsp(\alpha,\rho) $-valued processes.
\end{remark}

Proceeding to the result on finite-to-infinite-dimensional convergence,
we consider the finite-dimensional version of \eqref{eq:DBM}:
\begin{align}\label{eq:DBMf}
	X_i(t) = X_i(0) + B_i(t) 
	+ 
	\beta \int_0^t 
	\phi_i(\bfX(s)) ds,
	\ i\in[i_1,i_2]\cap\bbZ.
\end{align}
Let $ \Weyl^{[i_1,i_2]}:=\{\bfx\in\bbR^{[i_1,i_2]\cap\bbZ} : x_i<x_{i+1}, i\in[i_1,i_2]\} $ denote
the finite-dimensional Weyl camber.
Recall from \cite[Lemma 4.3.3]{anderson10} that, for any given $ \bfx^\ic\in\Weyl^{[i_1,i_2]} $,
there exists a $ C([0,\infty))^{[i_1,i_2]\cap\bbZ} $-valued strong solution $\bfX$ of \eqref{eq:DBMf}
with $\bbP(\bfX(t)\in\Weyl^{[i_1,i_2]},\forall t\geq 0)=1$, which is unique in the pathwise sense.
In Section~\ref{sect:conv} we show

\begin{theorem}\label{thm:conv}
Fixing $\bfx^\ic \in \XRsp(\alpha,\rho,p)$, 
we let $\bfX$ be the $\XRsp(\alpha,\rho,p)$-valued solution
of \eqref{eq:DBM} starting from $\bfx^\ic$,
and for
\begin{align}\label{eq:I}
	i_n^+ := \max\curBK{ i: x^\ic_i <n },
	\quad
	i_n^- := \min\curBK{ i: x^\ic_i > -n },
	\quad
	\I_n := [i^-_n(\bfx),i^+_n(\bfx)],
\end{align}
we let $\bfX^{n}$ be the $C([0,\infty))^{\I_n\cap\bbZ}$-valued
solution of \eqref{eq:DBMf} starting from $(x^\ic_i)_{i\in\I_n}$.
We have the following finite-to-infinite-dimensional convergences:
\begin{enumerate}[label=(\alph*)]
	\item	\label{enu:asCnvg}	
		For any fixed $ t\geq 0 $ and $ i\in\bbZ $,
		\begin{align}\label{eq:asCnvg}
			\sup_{s\in[0,t]}  |X^n_i(s)-X_i(s)| \to 0,	\text{ almost surely, as } n \to \infty;
		\end{align}
	\item	\label{enu:LpCnvg}	
		For any fixed $ t\geq 0 $, $ i\in\bbZ $ and $ p' \geq 1 $,
		\begin{align}\label{eq:LpCnvg}
			\bbE \Big( \sup_{s\in[0,t]} |X^n_i(s)-X_i(s)| \Big)^{p'} \to 0,
			\text{ as } n \to \infty;
		\end{align}
	\item	\label{enu:crlCnvg}	
		For any open $ \calO_1,\ldots,\calO_{j_*}\subset\bbR $ and $ s_1,\ldots,s_{j_*}\in[0,\infty) $,
		\begin{align*} 
			\bbE\BK{ \prod_{j=1}^{j_*} \Big| \calO_j\cap \{ X^{n}_{i}(s_j) \}_{i\in\I_n} \Big| }
			\to	
			\bbE\BK{ \prod_{j=1}^{j_*} \Big| \calO_j\cap \{ X_{i}(s_j) \}_{i\in\bbZ} \Big| }
			<
			\infty,
			\text{ as } n\to\infty.
		\end{align*}
\end{enumerate}
\end{theorem}

\begin{remark}
Hereafter, the limit $ n\to\infty $ as in Theorem~\ref{thm:conv}\ref{enu:asCnvg}--\ref{enu:LpCnvg},
are understood to be for all $ n $ large enough such that $ \calI_n \ni i $.
The sequence $ \{n:n\in\bbZ_{>0}\} $ can in fact be replaced by any sequence tending to infinity,
but we focus on the former to simply notations.
\end{remark}

As mentioned in the preceding,
for $\beta=2$,
\cite{katori10} shows that 
$ \{ X^n_i(s): i\in\I_n,s\in(0,\infty) \} $ is determinantal with an explicit kernel function,
and that, for $\bfx^\ic\in\KTsp := \KTsp_1 \cap \KTsp_2 \cap \KTsp_3$,
\begin{align*}
	\KTsp_1 &:= \Big\{ 
		\bfx\in\Weyl : 
		  \sup_{r>0 } \Big|\sum_{i\in\I_r} \frac{1}{x_i} \ind\curBK{x_i\neq 0} \Big| <\infty 
		\Big\},
\\
	\KTsp_2 &:= 
	\bigcup_{\alpha\in(1,2)}
	\Big\{ 
			\bfx\in\Weyl : 
			  \sum_{i\in\bbZ} \frac{1}{|x_i|^\alpha} \ind\curBK{x_i\neq 0} <\infty 
		\Big\},
\\
	\KTsp_3 &:= \bigcup_{\alpha>0}
	\Big\{ 
		\bfx\in\Weyl:
		\sup_{i\in\bbZ} \Big\{ (|x_i|\vee 1)^{\alpha} \sum_{j:j\neq i} \frac{1}{|(x_j)^2-(x_i)^2|} \Big\} <\infty
	\Big\},
\end{align*}
as $n\to\infty$ the kernel function converges to
$\bbK^{\bfx^\ic}(\Cdot,\Cdot;\Cdot,\Cdot)$
given as in \cite[(2.3)]{katori10}.
Indeed, since $ \XRsp(\alpha,\rho,p) \subset \KTsp_1\cap\KTsp_2 $,
combining this result of \cite{katori10} and Theorem~\ref{thm:conv}\ref{enu:crlCnvg}
we immediately obtain

\begin{corollary}\label{cor:det}
Fixing $ \beta=2 $, we let $\bfX$ be the $\XRTsp(\alpha,\rho,p)$-valued solution starting from $ \bfx^\ic\in\XRsp(\alpha,\rho,p)\cap\KTsp_3 $.
We have that $\{X_i(s):i\in\bbZ,s\in(0,\infty)\}$ is determinantal with the kernel function $\bbK^{\bfx^\ic}(\Cdot,\Cdot;\Cdot,\Cdot)$.
\end{corollary}

The rest of this paper is outlined as follows.
In Section~\ref{sect:outline} we present a proof of Theorem~\ref{thm:DBM}, 
which is detailed in Section~\ref{sect:comp}--\ref{sect:exist}.
Among these, Section~\ref{sect:comp} settles the monotonicity \eqref{eq:Ydec} and well-posedness of certain finite-dimensional \ac{SDE},
and Section~\ref{sect:unique}--\ref{sect:exist} handle the relevant propositions as indicated in their titles.
Section~\ref{sect:conv} consists of the proof of Theorem~\ref{thm:conv}. 
In Section~\ref{sect:nearEq}, 
we prove that near-equilibrium solutions (defined therein)
are $ \XRTsp(\alpha,\rho,p) $-valued, to unify the construction of \cite{osada12} with ours.

\subsection*{Acknowledgment.} 
LCT thanks Alexei Borodin for suggesting this direction of research,
Amir Dembo for many fruitful discussions,
and the anonymous reviewers for improving the presentation of this paper.
LCT is partially supported by the NSF through DMS-0709248.

\section{Proof of Theorem~\ref{thm:DBM}}
\label{sect:outline}

Throughout this paper we use lower-case English and Greek letters
such as $ x,y,\alpha,\gamma,u $ to denote deterministic variables or functions,
among which $ i,j,k,\ell,m,n $ denote integers, and $ a,b $ denote half integers.
We use upper-case English letters such as $ X,Y,I,J $ to denote random variables,
use the calligraphic font (e.g.\ $ \calA,\calI $) to denote deterministic sets,
and use the Fraktur font (e.g.\ $\fkA,\fkI$) to denote random sets.
We let $ c=c(t,k,\ldots) $ denote a generic deterministic positive finite constant 
that depends only on the designated variables.

The first step is to reduce the equation of particles, \eqref{eq:DBM}, 
to the equation of the \emph{gaps}.
To this end, we consider the interaction of the gaps
\begin{align}\label{eq:phifS}
	\fS_a(\bfy) 
	&:= 
	\fS_a(u(\bfx))
	:=
	\phi_{a+1/2}(\bfx) - \phi_{a-1/2}(\bfx),
\\
	\label{eq:fS}
	&= \tfrac{1}{y_a} - \f_a(y_a,\bfy),
\end{align}
consisting of the (Bessel-type) repulsion terms $ 1/y_a $ and the compression terms $ \f_a $
defined as
\begin{align}
	\label{eq:psi}
	\f_a: [0,\infty)\times (0,\infty)^\hZ \to [0,\infty),
	\ \
	\f_a(y,\bfz)
	:=
	\left\{\begin{array}{l@{,}l}
		\displaystyle
		\frac12 \sum_{ i:|i-a|>1 }
		\frac{ y }{ z_{(a,i)} (y+z_{(a,i)}) }
		&\text{ for } y>0,
	\\
		0	&\text{ for } y=0,
	\end{array}\right.
\end{align}
where $ z_{\calI} := \sum_{a\in\calI} z_a $ and $ (a,i):=(i,a) $ (as mentioned before).
We have the following equation for $(X_0,\bfY):=u(\bfX)$:
\begin{align}
	\label{eq:DBM0}
	X_0(t) &= X_0(0) + B_0(t) 
	+ 
	\beta \int_0^t \phi_0(\bfY(s)) ds,
\\
	\label{eq:DBMg}
	Y_a(t) &= Y_a(0) + W_a(t) + \beta \int_0^t \fS_a\BK{\bfY(s)} ds, \quad a\in\hZ,
\end{align}
where $\bfW(t):=u(\bfB(t))$,
and, by abuse of notation,
\begin{align*}
	\phi_0(\bfy) := \phi_0(\tilu^{-1}(0,\bfy))
	=
	\sum_{i=1}^\infty 
	\BK{ \frac{1}{2y_{(-i,0)}} - \frac{1}{2y_{(0,i)}} }.
\end{align*}
Clearly, \eqref{eq:DBM} is equivalent to \eqref{eq:DBM0}--\eqref{eq:DBMg} through the bijection $\tilu$, 
and one easily obtains the following
\begin{proposition}\label{prop:conversion}
\begin{enumerate}[label=(\alph*)]
	\item[]
	\item
	If $\bfY$ is an $\YTsp(\alpha,\rho)$-valued solution of \eqref{eq:DBMg}, 
	defining $ X_0\in C([0,\infty)) $ by \eqref{eq:DBM0},
	we have that $\tilu^{-1}(X_0,\bfY)$ is a $\XTsp(\alpha,\rho)$-valued solution
	of \eqref{eq:DBM}. 
	Further, if $ \bfY $ is $ (\YTsp(\alpha,\rho)\cap\RTsp(p)) $-valued,
	then $\tilu^{-1}(X_0,\bfY)$ is $ \XRTsp(\alpha,\rho,p) $-valued;
	if $ \bfY $ is $\Wfil$-adapted, then $\tilu^{-1}(X_0,\bfY)$ is $ \Bfil $-adapted.
	
	\item
	Conversely, if $\bfX$ is an $\XTsp(\alpha,\rho)$-valued
	solution of \eqref{eq:DBM},
	then $ u(\bfX) $ is a $ \YTsp(\alpha,\rho) $-valued solution of \eqref{eq:DBMg}.
	Further, if $ \bfX $ is $ \XRTsp(\alpha,\rho,p) $-valued,
	then $ u(\bfX) $ is $ (\YTsp(\alpha,\rho)\cap\RTsp(p)) $-valued;
	if $ \bfX $ is $ \Bfil $-adapted, then so is $ u(\bfX) $.
\end{enumerate}
\end{proposition}

\noindent
With this proposition, it now suffices to prove

\begin{proposition}\label{prop:DBMg}
For any given $ \bfy^\ic\in\YTsp(\alpha,\rho) $,
there exists a $ \YTsp(\alpha,\rho) $-valued, $ \Wfil $-adapted solution $ \bfY $ of \eqref{eq:DBMg}.
Moreover, if $ \bfy^\ic\in\Rsp(p) $, then $ \bfY\in\RTsp(p) $,
and $ \bfY $ is the unique $ (\YTsp(\alpha,\rho)\cap\RTsp(p)) $-valued solution in the pathwise sense.
\end{proposition}

We establish Proposition~\ref{prop:DBMg} in two steps:
the existence, as in Proposition~\ref{prop:exist},
and
the uniqueness, as in Proposition~\ref{prop:unique}.
Defining the partial orders
\begin{align}
	&
	\label{eq:order}
	\bfy \leq \bfy' \in [0,\infty]^\hZ
	\text{ if and only if }
	 y_a \leq y'_a,
	 \ \forall a\in\hZ,
\\
	&
	\label{eq:orderT}
	\bfy(\Cdot) \leq \bfy'(\Cdot)
	\text{ if and only if }
	\bfy(t) \leq \bfy'(t),
	 \ \forall t\geq 0,
\end{align}
we call $\bfY$ the greatest $\calS$-valued solution of \eqref{eq:DBMg} if,
for any $\calS$-valued weak solution $\bfY'$ defined on a common probability space with $\bfY'(0) \leq \bfy^\ic$,
we have $\bfY'(\Cdot) \leq \bfY(\Cdot)$ almost surely.
\begin{proposition}[existence] \label{prop:exist}
For any $\bfy^\ic\in\Ysp(\alpha,\rho)$,
there exists a $\YTsp(\alpha,\rho)$-valued, 
$\Wfil$-adapted solution $\bfY$ of \eqref{eq:DBMg} starting form $\bfy^\ic$,
which is the greatest $\YTsp(\alpha,\rho)$-valued solution.
Further, if $ \bfy^\ic\in\Rsp(p) $, then $ \bfY\in\RTsp(p) $.
\end{proposition}
\begin{proposition}[uniqueness] \label{prop:unique}
Let $\bfYup$ and $\bfYlw$ be $(\YTsp(\alpha,\rho)\cap\RTsp(p))$-valued weak solutions of \eqref{eq:DBMg}
defined on a common probability space,
starting from a common initial condition $\bfy^\ic$.
If $\bfYlw(\Cdot)\leq\bfYup(\Cdot)$ almost surely,
we have $\bfYlw(\Cdot)=\bfYup(\Cdot)$ almost surely.
\end{proposition}
\noindent
Indeed, Proposition~\ref{prop:DBMg} follows
by combining Proposition~\ref{prop:exist}--\ref{prop:unique}.
In particular, the pathwise uniqueness follows by applying Proposition~\ref{prop:unique}
for $\bfYup=\bfY$ and $\bfYlw=\bfY'$, 
where $ \bfY $ is the greatest solution as in Proposition~\ref{prop:exist},
and $\bfY'$ is an arbitrary weak solution with $\bfY'(0)=\bfY(0)$.

Proposition~\ref{prop:exist}
is established in two steps:
by first considering the special case $ \bfy^\ic\in [\gamma,\infty)^\hZ $, $ \gamma>0 $,
and then the general case $ \bfy^\ic\in\Ysp(\alpha,\rho) $.
For the former case,
we construct the solution of \eqref{eq:DBMg} by the following iteration scheme,
\begin{subequations}\label{eq:iter:}
\begin{align}
	&
	\label{eq:iter0}
	Y^{(0)}_a(t) = y^\ic_a + W_a(t) + \beta \int_0^t \frac{1}{Y^{(0)}_a(s)} ds, \quad a\in\hZ,
\\
	&
	\label{eq:iter}
	Y^{(n)}_a(t) = y^\ic_a + W_a(t) 
\\
	&
	\notag	
	+ \beta \int_0^t \BK{ \frac{1}{Y^{(n)}_a(s)} - \f_a(Y^{(n)}_a(s),\bfY^{(n-1)}(s)) } ds,
	\quad a\in\hZ, \ n\in\bbZ_{>0}.
\end{align}
\end{subequations}
That is, we let $Y_a^{(0)}$ be the Bessel process (driven by $W_a$),
and for $n\geq 1$, we let $ Y_a^{(n)} $ be the solution of 
the following \emph{one-dimensional} \ac{SDE}
\begin{align}\label{eq:oneD}
	Y(t) = Y(0) + W_a(t) + \beta \int_0^t \BK{ \frac{1}{Y(s)}  - \f_a(Y(s),\bfZ(s)) } ds,
\end{align}
for \emph{given} $ \bfZ=\bfY^{(n-1)} $.
Letting
\begin{align}
	&
	\label{eq:Yspl}
	\Yspl(\gamma)
	:=
	\Big\{
		\bfy\in(0,\infty)^\hZ:
		\liminf_{|m|\to\infty} \av_{(0,m)} (\bfy) \geq \gamma
	\Big\},
\\
	&
	\label{eq:YTspl}
	\YTspl(\gamma)
	:=
	\Big\{
		\bfy(\Cdot)\in C_+([0,\infty))^\hZ:
		\liminf_{|m|\to\infty}
		\inf_{s\in[0,t]} \av_{(0,m)}(\bfy(s)) \geq \gamma,
		\forall t\geq 0
	\Big\},
\\
	&
	\label{eq:Yspgg}
	\Yspl:=\cup_{\gamma>0}\Yspl(\gamma),
	\quad
	\YTspl:=\cup_{\gamma>0}\YTspl(\gamma),
\end{align}
in Section~\ref{sect:exist:special} we prove

\begin{proposition}\label{prop:existg}
Fix $ \gamma>0 $. 
For any given $\bfy^\ic\in[\gamma,\infty)^\hZ$,
there exists a $ \YTsp(\gamma) $-valued, $ \Wfil $-adapted sequence 
$ \{\bfY^{(n)}\}_{n\in\bbZ_{\geq 0}} $ satisfying \eqref{eq:iter:}.
Further, such a sequence is \emph{decreasing}, i.e.\
\begin{align}\label{eq:Ydec}
	\bfY^{(0)}(\Cdot) \geq \bfY^{(1)}(\Cdot) \geq \bfY^{(2)}(\Cdot) \geq \ldots,
\end{align}
almost surely.
Defining the $\Wfil$-adapted process $ Y^{(\infty)}_a(t):= \lim_{n\to\infty} Y^{(n)}_a(t) $,
we have that $ \bfY^{(\infty)} $ is the greatest $ \YTspl $-valued solution of \eqref{eq:DBMg}.
If $ \bfy^\ic\in\Rsp(p) $, then $ \bfY^{(\infty)}\in\RTsp(p) $.
\end{proposition}
\noindent
For the general case $\bfy^\ic\in\Ysp(\alpha,\rho)$,
we consider the truncated initial condition $ (\bfy^\ic\vee\bfgamma):=(y^\ic_a\vee\gamma)_{a\in\hZ} $,
$ \gamma>0 $, and let $\bfYg$ be the $\YTspl$-valued solution starting from $ (\bfy^\ic\vee\bfgamma) $
given by Proposition~\ref{prop:existg}.
As $ \bfYg $ is the greatest solution, for any decreasing $ \{\gamma_1>\gamma_2>\ldots \} $,
the sequence $ \{\bfY^{\vee \gamma_k}\}_{k} $ is decreasing.
In Section~\ref{sect:exist}, we prove
\begin{proposition}\label{prop:existGener}
Let $ \bfy^\ic\in\Ysp(\alpha,\rho) $ and $ \bfYg\in\YTspl(\gamma)$ be as in the preceding.
Fix an arbitrary decreasing sequence $ 1\geq \gamma_1>\gamma_2>\ldots \to 0 $.
Defining the $ \Wfil $-adapted process $ Y_a(t) := \lim_{n\to\infty} Y^{\vee\gamma_n}_a(t) $,
we have that $ \bfY $ is the greatest $ \XTsp(\alpha,\rho) $-valued solution of \eqref{eq:DBMg}.
\end{proposition}

As for Proposition~\ref{prop:unique},
letting 
\begin{align}\label{eq:E}
	E_{(i_1,i_2)}(t) := \sum_{a\in(i_1,i_2)} (\Yup_a(t)-\Ylw_a(t)),
\end{align}
with $ \bfYlw(\Cdot) \leq \bfYup(\Cdot) $, we have
$
	\absBK{\Yup_{a}(t)-\Ylw_{a}(t)}
	\leq
	E_{(i_1,i_2)]}(t)
	\leq
	E_{(-\infty,\infty)}(t),
$
$ \forall a\in (i_1,i_2) $.
With this, in Section~\ref{sect:unique} we prove
\begin{proposition}\label{prop:unique2}
For any $ t>0 $,
$
	\sup_{s\in[0,t]} E_{(-\infty,\infty)}(s) 
	=0,
$
almost surely,
\end{proposition}
\noindent
from which Proposition~\ref{prop:unique} follows immediately.

\subsection{Outline of the Proof of Proposition~\ref{prop:existg}--\ref{prop:unique2}}
The key step of proving Proposition~\ref{prop:existg} 
is to establish the monotonicity
\eqref{eq:Ydec} of $ \{\bfY^{(n)}\}_{n} $.
This, as well as many other monotonicity results 
(e.g.\ that $ \bfY^{(\infty)} $ as in Proposition~\ref{prop:existg} is the greatest solution),
are consequences of the following simple observation:
\begin{align}
	&
	\label{eq:psi:mono:y}
	\f_a(y,\bfz) \leq \f_a(y',\bfz), \text{ if } y \leq y',
\\
	&
	\label{eq:psi:mono:z}
	\f_a(y,\bfz) \geq \f_a(y,\bfz'), \text{ if } \bfz \leq \bfz',
\end{align}
which is clear from \eqref{eq:psi}.
A basic tool we use to leverage \eqref{eq:psi:mono:y}--\eqref{eq:psi:mono:z}
into the monotonicity of $ \{\bfY^{(n)}\}_{n} $
is the following comparison principle for \emph{deterministic}, one-dimensional integral equations.
Let 
\begin{align*} 
	\bfy \leq_{[t',t'']} \bfy' \text{ if and only if } \bfy(t) \leq \bfy(t), \forall t\in[t',t'']
\end{align*}
denote the restriction of \eqref{eq:orderT} onto $ [t',t''] $.
\begin{lemma}\label{lem:comp}
Fixing $t' \leq t'' \in [0,\infty)$, 
we let $w\in C([t',t''])$,
and let $\fup,\flw\in C((0,\infty)\times[t',t''])$
be locally Lipschitz functions in the first variable.
That is, given any compact $\calK\subset(0,\infty)$,
there exists $c(\calK)>0$ such that
\begin{align*}
	\absBK{ \fup (y,t)-\fup (y',t) },
	\
	\absBK{ \flw (y,t)-\flw (y',t) } 
	\leq 
	c(\calK) \absBK{y-y'},
\end{align*}
for all $y,y'\in\calK$ and $t\in[t',t'']$.
If $\yup,\ylw\in C_+([t',t''])$ solve the follows integral equations
\begin{align}
	\label{eq:prop:comp:y}
	\yup(t) &= \yup(t') + (w(t)-w(t')) + \int_{t'}^t \fup(\yup(s),s) ds, \quad \forall \ t\in[t',t''],
\\
	\label{eq:prop:comp:z}
	\ylw(t) &= \ylw(t') + (w(t)-w(t')) + \int_{t'}^t \flw(\ylw(s),s) ds, \quad \forall \ t\in[t',t''],
\end{align}
and if $\flw(y,\Cdot) \leq_{[t',t'']} \fup(y,\Cdot)$, $\forall y\in(0,\infty)$,
and $\ylw(t') \leq \yup(t')$,
then 
\begin{align*}
	\ylw \leq_{[t',t'']} \yup.
\end{align*}
\end{lemma}
\noindent
With $ \fup(\Cdot,s) $ and $ \flw(\Cdot,s) $ being locally Lipschitz,
Lemma~\ref{lem:comp} is proven by standard ODE arguments
using Gronwall's inequality. We omit the proof.
Equipped with the monotonicity of $ \{\bfY^{(n)}\}_n $,
the next step is to take the limit $ n\to\infty $ in \eqref{eq:iter},
and show that the r.h.s.\ converges to the appropriate limit.
The major challenge here is to control $ \int_0^t \tfrac{\beta}{Y^{(\infty)}_a(s)} ds $,
which we achieve by showing
\begin{align}\label{eq:ftcolide}
	\inf_{s\in[0,t]} Y^{(\infty)}_a(s) >0,
	\quad
	\text{ almost surely, for all } t\geq 0.
\end{align}
\begin{remark}\label{rmk:singular}
For any $ \beta\geq 1 $,
the non-existence of finite time collisions, \eqref{eq:ftcolide},
cannot be achieved solely by the local Bessel-type repulsion $ (\beta/y_a) $. 
To see this, 
rewrite the interaction $ \beta \fS_a(\bfy) $ (as in \eqref{eq:phifS}) as
\begin{align*}
	\beta\fS_a(\bfy)
	=
	\beta\BK{ \tfrac{1}{y_a} - \tfrac{1}{2y_{a+1}} - \tfrac{1}{2y_{a-1}} + (\text{terms involving multiple gaps}) }.
\end{align*}
Estimating the strength of the first three terms
(which dominate when particles come close together)
by their coefficients,
we find that the term $ (\beta/y_a) $ 
comes just enough to balance $\beta 2^{-1}[(y_{a+1})^{-1}+(y_{a-1})^{-1}]$.
One may continue this estimation to higher orders.
By grouping terms according to the number of gaps involved,
one finds that the strength of positive and negative terms always balance.
This differs from the finite-dimensional case,
where a residual term contributes positively when summing over all gaps.
\end{remark}

The idea of proving \eqref{eq:ftcolide}
is to utilize the \emph{global} property of conservation of average spacing
$
	\kappa(\bfy(t)) := \lim_{(i_1,i_2)\to(-\infty,\infty)} \av_{(i_1,i_2)} (\bfy(t)),
$
assuming such a limit exists.
To see the intuition of such a quantity being conserved,
note that for a generic solution $ \bfX $ of \eqref{eq:DBM} we have
\begin{align}
	\notag
	\av_{(i_1,i_2)} (\bfY(s)) \Big|_{s=0}^{s=t}
	&= 
	\frac{ X_{i_2}(s)-X_{i_1}(s) }{ i_2-i_1 } \Big|^{s=t}_{s=0}
\\
	\label{eq:global}
	&=
	\frac{ B_{i_2}(t)-B_{i_1}(t) }{ i_2-i_1 }
	+
	\beta \int_0^t \BK{ \frac{\phi_{i_2}(\bfX(s))}{i_2-i_1} - \frac{\phi_{i_1}(\bfX(s))}{i_2-i_1} } ds,
\end{align}
where $ \bfY(t):= u(\bfX(t)) $.
Letting $ (i_1,i_2)\to(-\infty,\infty) $,
assuming $ |\phi_{i_1}(\bfX(s))|, |\phi_{i_2}(\bfX(s))| \ll |i_2-i_1| $,
we find that $ \kappa(\bfY(t))=\kappa(\bfY(0)) $, $ \forall t> 0 $, i.e.\ the average spacing is conserved.
With $ \bfY^{(n)} $ satisfying \eqref{eq:iter},
following the preceding type of argument, we show that 
$ \kappa(\bfY^{(\infty)}(t)) \geq \kappa(\bfY^{(\infty)}(0)) = \rho$, $ \forall t>0 $ (see Lemma~\ref{lem:YZlw}--\ref{lem:Yiter}),
which roughly speaking implies
$ \av_{(m,m')} (\bfY^{(\infty)}(t)) \geq |m-m'|/c $, $ c<\infty $,
outsides of large windows.
This then allows to control the strength of $ \f_a(\bfY^{(\infty)}(t)) $ outsides of a certain large window,
whereby reducing the problem to finite dimensions.

The main step of proving Proposition~\ref{prop:existGener} is to show $ \bfY\in\YTsp(\alpha,\rho) $.
To this end, in Section~\ref{sect:exist},
we partition $ \hZ $ into certain mesoscopic intervals $ \A_{b,k} $, $ b\in\hZ $, (see \eqref{eq:A})
and simultaneously estimate $ \av_{\A_{b,k}}(\bfYgn(s)) $, $ \forall n\in\bbZ_{>0}, b\in\hZ $.
This yields that the mesoscopic average of $ \bfY(s) $ over $ \A_{b,k} $ is at least $ \tfrac{\rho}{2} $ (see Proposition~\ref{prop:dens}).
Using this as a `seed',
we estimate the global density $ \av_{(0,m)}(\bfY(s)) $, $ |m|\gg 1 $, via \eqref{eq:global}
to obtain $ \bfY\in\YTsp(\alpha,\rho) $.

To prove Proposition~\ref{prop:unique2}, in Section~\ref{sect:unique},
we derive the following equation 
\begin{align}\label{eq:uniEq:est}
	E_{(i_1,i_2)}(t) 
	= 
	E_{(i_1,i_2)}(t') 
	+
	\beta \int_{t'}^{t} 
	\BK{ 
		 L^+_{i_2}(s) - L^-_{i_2}(s) 
		- L^+_{i_1}(s) + L^-_{i_1}(s)
		}ds,
	\
	\forall t\geq t',
\end{align}
that describes $ E_{(i_1,i_2)}(t) $ in terms of 
certain boundary interactions $ L^\pm_i(s) $, defined as
\begin{align}\label{eq:L}
	 L_i^\pm(s):=
	\frac12 \sum_{j\in(i,\pm\infty)} 
	\frac{ \Yup_{(i,j)}(s)-\Ylw_{(i,j)}(s) }{ \Yup_{(i,j)}(s) \Ylw_{(i,j)}(s)}.
\end{align}

With $ E_{(i_1,i_2)}(0)=0 $,
equipped with \eqref{eq:uniEq:est},
in Section~\ref{sect:unique},
we prove Proposition~\ref{prop:unique2} by showing $ \int_0^t L^\pm_{i_k}(s) ds \to 0 $,
along some suitable subsequence $ i_k \to \pm\infty $.
To see the intuition of this,
note that for each $j$, the denominator of the $j$-th term in \eqref{eq:L} 
is approximately $\rho |j -i|^2$.
As for the numerator, with $\bfYlw, \bfYup\in\Ysp(\alpha,\rho)$, we have
\begin{align*}
	\Big| \av_{(0,m)} \BK{ \bfYup(s) - \bfYlw(s) } \Big| |m|^{\alpha'} 
	\xrightarrow{|m|\to\infty} 0,
\end{align*}
for all $\alpha'<\alpha$,
suggesting that the numerator is at most $\sum_{a\in(i,j)} |a|^{-\alpha'}$.
Combining these bounds yields $L^\pm_{i}(s) \to 0$ as $|i|\to\pm\infty$.

\section{Comparison and Monotonicity}
\label{sect:comp}

We begin by establishing the monotonicity \eqref{eq:Ydec}.
Recall the definition of $ \Ysp(\gamma) $ and $ \YTsp(\gamma) $ from \eqref{eq:Yspl}--\eqref{eq:YTspl}.
\begin{proposition}\label{prop:iterDecr}
Fixing $ \bfy^\ic, \bfz^\ic \in \Ysp(\gamma) $, $ \gamma>0 $,
we let $ \{\bfY^{(i)}\}_{i=0}^{n} $ and $ \{\bfZ^{(i)}\}_{i=0}^{n} $ be $ \YTspl(\gamma) $-valued sequences
satisfying \eqref{eq:iter:},
with $ \bfY^{(i)}(0)=\bfy^\ic $ and $ \bfZ^{(i)}(0)=\bfz^\ic $, $ i=0,\ldots, n $.
\begin{enumerate}[label=(\alph*)]
	\item The sequence $ \{\bfY^{(i)}\}_{i=0}^{n} $ is decreasing, 
		$ \bfY^{(0)}(\Cdot) \geq \ldots \geq \bfY^{(n)}(\Cdot) $.
	\item If $ \bfy^\ic \geq \bfz^\ic $,
		we have $ \bfY^{(i)}(\Cdot) \geq \bfZ^{(i)}(\Cdot) $, for $ i=0,\ldots,n $.
\end{enumerate}
\end{proposition}
\begin{proof}
We begin by showing that
\begin{align}\label{eq:psiLip}
	y \mapsto \f_a(y,\bfz(s)) \text{ is uniform Lipschitz over } [0,\infty)\times[0,t],
	\quad
	\forall t \geq 0,\
	\forall \bfz \in \YTsp. 
\end{align}
To this end, with $ \psi_a $ defined as in \eqref{eq:psi}, we estimate the expression
\begin{align}\label{eq:psiLip:}
	\f_a(y,\bfz(s)) - \f_a(y',\bfz(s))
	=
	\frac12 \sum_{i:|i-a|>1} 
	\frac{ y'-y }{ (y+z_{(a,i)}(s)) (y'+z_{(a,i)}(s)) }.
\end{align}
With $ y,y'\geq 0 $,
we bound the r.h.s.\ by $ 2^{-1}|y-y'|\sum_{i:|i-a|>1}  ( z_{(a,i)}(s) )^{-2} $,
which converges uniformly over $ [0,t] $, for any $ \bfz\in\YTspl(\gamma) $. 
Hence \eqref{eq:psiLip} follows.

We now prove $ \bfY^{(i-1)}(\Cdot) \geq \bfY^{(i)}(\Cdot) $ by induction on $ i $.
For $ i=1 $, by \eqref{eq:psiLip},
we have that $ -\psi_a(\Cdot,\bfY^{(0)}(t)) $ is uniformly Lipschitz.
With $ -\psi_a(y_a,\bfY^{(0)}(t) ) \leq 0 $,
and $ Y^{(0)}_a$ and $ Y^{(1)}_a $ solving the respective equations \eqref{eq:iter0} and \eqref{eq:iter},
applying Lemma~\ref{lem:comp} for $ \yup=Y^{(0)}_a $ and $ \ylw=Y^{(1)}_a $,
we conclude $ \bfY^{(0)}(\Cdot) \geq \bfY^{(1)}(\Cdot) $.
Assuming $ \bfY^{(i-1)}(\Cdot) \geq \bfY^{(i)}(\Cdot) $, $ i>1 $,
by \eqref{eq:psi:mono:z} we have
$ -\psi_a(y,\bfY^{(i-1)}(t)) \geq -\psi_a(y,\bfY^{(i)}(t)) $.
With $ Y^{(i)}_a $ and $ Y^{(i+1)}_a $ solving \eqref{eq:iter},
applying Lemma~\ref{lem:comp} for $ \yup=Y^{(i)}_a $ and $ \ylw=Y^{(i+1)}_a $
we conclude $ \bfY^{(i)}(\Cdot) \geq \bfY^{(i+1)}(\Cdot) $.
This completes the proof of (a).

As for (b), 
the case $ i=0 $ follows directly by applying Lemma~\ref{lem:comp}.
For $ i>0 $, by \eqref{eq:psi:mono:z},
we have that
\begin{align*}
	\bfZ^{(i)}(\Cdot) \leq \bfY^{(i)}(\Cdot) 
	\text{ implies }
	- \f_a(y,\bfZ^{(i)}(s)) \leq - \f_a(y,\bfY^{(i)}(s)),
\end{align*}
so, by induction, the case $ i>0 $ 
follows by the preceding comparison argument.
\end{proof}

Next, we establish a backward lower-semicontinuouity for
a generic process of the form \eqref{eq:BesComp:Y}.
To this end,
we consider $\bfQ^{t_1}:= (Q^{t_1}_a)_{a\in\hZ}\in (C([t_1,\infty))\cap C_+(t_1,\infty) )^\hZ $,
\begin{align}\label{eq:Q}
	Q^{t_1}_a(t) = W_a(t) - W_a(t_1) + \int_{t_1}^t \frac{\beta}{Q^{t_1}_a(s)} ds, \quad t\geq t_1,
\end{align}
the Bessel process starting from $0$ at $t_1$,
and let
$
	Q_a^{t_1,t_2} := \sup_{t\in[t_1,t_2]} Q^{t_1}_a(t).
$
Indeed, for
$ \hZ_1 := \tfrac12 + 2\hZ$ and $\hZ_2 := 1+\hZ_1 $,
\begin{align}\label{eq:Qiid}
	\{Q^{t_1}_a(\Cdot)\}_{a\in\hZ_i}, i=1,2,
	\text{ are i.i.d.\ collections of processes.}
\end{align}
Hence, by the Law of Large Numbers, we have
\begin{align}\label{eq:QLLN}
	\lim_{|m|\to\infty } \av^p_{(0,m)} \BK{ \bfQ^{t_1,t_2}  } 
	= 
	\bbE\big( (Q^{t_1,t_2}_{1/2})^p \big) := q(t_2-t_1,p) < \infty.
\end{align}
Hereafter, for generic processes $ Y(\Cdot) $ and $ \bfY(\Cdot) $,
we adopt the notations
\begin{align*}
	&
	\barY(t',t''):=\sup_{s\in[t',t'']} Y(s),&
	&
	\undY(t',t''):=\inf_{s\in[t',t'']} Y(s),
\end{align*}
$\bfbarY(t',t'') := (\barY_a(t',t''))_{a\in\hZ}$
and $\bfundY(t',t'') := (\undY_a(t',t''))_{a\in\hZ}$.

\begin{lemma}\label{lem:BesComp}
Let $a\in\hZ$, $Y^*\in C_+([0,\infty))$, $F \in L^1_{\text{loc}}([0,\infty))$,
$\{\scrG_t\}_{t\geq 0}$ be a filtration such that
$Y^*$, $F$ and $\bfW$ are $\scrG$-adapted and 
that $\bfW$ is a Brownian motion with respect to $\scrG$.
If $F \geq 0$ and if $Y^*$ solves the equation
\begin{align}\label{eq:BesComp:Y}
	Y^*(t) = Y^*(0) + W_a(t) + \beta \int_0^t \BK{ \frac{1}{Y(s)} - F(s) } ds,
	\quad a\in\hZ,
\end{align}
then, for all $t'\leq t''\in[0,\infty)$, we have
\begin{align}
	\label{eq:BesComp}
	&
	Y^*(t'') - \undY^*(t',t'') = \sup_{s\in[t',t'']} (Y^*(t'') - Y^*(s)) 
	\leq Q^{t'}_a(t''),
\\
	\label{eq:BesComp:sup}
	&
	\sup_{s<t\in[t',t'']} \BK{ Y^*(t) - Y^*(s) }
	\leq 	Q_a^{t',t''} := \sup_{t\in[t',t'']} Q^{t'}_a(t),
\end{align}
almost surely.
\end{lemma}

\begin{proof}
To the end of showing \eqref{eq:BesComp}, 
fixing $s_1\in (t',t'') $,
we consider the process $Y^{s_1} \in C_+([s_1,\infty)) $ defined as
\begin{align} \label{eq:BesComp:Ys}
	Y^{s_1}(t) = Y^*(s_1) + W_a(t) - W_a(s_1) 
	+ \beta \int_{s_1}^t \frac{1}{Y^{s_1}(s)} ds, \quad t\geq s_1,
\end{align}
which is a Bessel process starting from $Y^*(s_1) $ at time $ s_1 $. 
With $ Y^* $ and $ Y^{s_1} $ satisfying \eqref{eq:BesComp:Y} and \eqref{eq:BesComp:Ys},
applying Lemma~\ref{lem:comp}
(for $ [t',t'']=[s_1,t''] $, $ \yup = Y^{s_1} $, $ \ylw=Y^* $,
$ \fup(y,s)=\beta/y $ and $ \flw(y,s) = \beta(1/y-F(s)) $),
we obtain $ Y^*(\Cdot) \leq_{[s_1,t'']} Y^{s_1}(\Cdot) $,
and therefore, with $ Y^*(s_1)=Y^{s_1}(s_1) $,
\begin{align}\label{eq:BesComp:1}
	Y^*(t'') - Y^*(s_1) \leq Y^{s_1}(t'') - Y^{s_1}(s_1).
\end{align}
We next compare $ Y^{s_1} $ and $ Q^{s_1}_a $.
They solve the same equation, \eqref{eq:Q} and \eqref{eq:BesComp:Ys},
with different initial conditions $ Y^{s_1}(s_1) > 0 = Q^{s_1}_a(s_1) $.
Hence, applying Lemma~\ref{lem:comp}
for $ (t',t'')=(s_1+\varepsilon,t'') $, $ \varepsilon>0 $ 
(so that $ Y^{s_1}, Q^{s_1}_a \in C_+([s_1+\varepsilon,t'']) $),
conditioned on  $ \{Y^{s_1}(s_1+\varepsilon) \geq Q^{s_1}_a(s_1+\varepsilon)\} $,
and then sending $ \varepsilon\to 0 $,
we obtain
$
	Q^{s_1}_a \leq_{[s_1,t'']} Y^{s_1}
$
almost surely,
thereby $ \int_{s_1}^{t''} \tfrac{\beta}{Y^{s_1}(s)} ds \leq \int_{s_1}^{t''} \tfrac{\beta}{Q^{s_1}_a(s)} ds $.
Plugging this in \eqref{eq:Q} and \eqref{eq:BesComp:Ys},
we obtain
\begin{align}\label{eq:BesComp:2}
	Y^{s_1}(t'') - Y^{s_1}(s_1) \leq Q^{s_1}_a(t'') - Q^{s_1}_a(s_1) = Q^{s_1}_a(t'').
\end{align}
Next, as $ Q^{s_1}_a $ and $ Q^{t'}_a $ solve the same equation on $ [s_1,t''] $
with the initial conditions $ Q^{s_1}_a(s_1)=0 < Q^{t'}_a(s_1) $,
by the preceding comparison argument we obtain $ Q^{s_1}_a(t'') \leq Q^{t'}_a(t'') $.
Combining this with \eqref{eq:BesComp:1}--\eqref{eq:BesComp:2},
we arrive at
$
	Y^*(t'') - Y^*(s_1) \leq Q^{t'}_a(t'').
$
As this holds almost surely for each $ s_1\in (t',t'') $,
taking the infimum over $ s_1\in (t',t'') \cap\bbQ $,
using the continuity of $ Y^*(\Cdot) $, we conclude \eqref{eq:BesComp}.

As for \eqref{eq:BesComp:sup}, 
taking the supremum over  $ t'' \in[t',\widetilde{t}'']\cap\bbQ $ in \eqref{eq:BesComp},
using the continuity of $Y^*(\Cdot)$,
we obtain
\begin{align*}
	\sup_{t''\in[t',\widetilde{t}'']} \BK{ Y^*(t'') - \inf_{s\in[t',\widetilde{t}'']} Y^*(s) }
	=
	\sup_{s<t\in[t',\widetilde{t}'']} \BK{ Y^*(t) - Y^*(s) }
	\leq \sup_{t''\in[t',\widetilde{t}'']} Q^{t'}_a(t'') = Q^{t',\widetilde{t}''}_a.
\end{align*}
\end{proof}

For the rest of this section,
we establish the the well-posedness of certain finite-dimensional \ac{SDE}.
We begin with the one-dimensional equation \eqref{eq:oneD'} in the following,
which is a generalization of \eqref{eq:oneD}.
\begin{lemma}\label{lem:oneD}
Let $t'\geq 0$
and $F:[0,\infty)\times[t',\infty)\to\bbR$ be random, such that 
\begin{enumerate}
\item[]		$s\mapsto F(y,s)$ is $C([t',\infty),\bbR)$-valued and $\Wfil$-adapted for all $y\in[0,\infty)$,	
\item[]		$y\mapsto F(y,s)$ is Lipschitz, uniformly over $(y,s)\in[0,\infty)\times[t',t]$, for all $t\geq t'$,	
\item[]		$F(0,t)= 0$, for all $ t\geq t' $.
\end{enumerate}
Given any $(0,\infty)$-valued, $\Wfil_{t'}$-measurable $Y^\ic$, 
the equation 
\begin{align}\label{eq:oneD'}
	Y(t) = Y^\ic + \BK{W_a(t) - W_a(t')} + \int_{t'}^t \BK{ \frac{\beta}{Y(s)} + F(Y(s),s) } ds
\end{align}
has a $ C_+([t',\infty)) $-valued, 
$\Wfil$-adapted solution starting from $ Y^\ic $ at $ t' $, 
which is the unique $ C_+([t',\infty)) $-valued solution in the pathwise sense.
\end{lemma}
\begin{remark}
Equation \eqref{eq:oneD'} for $F=0$ describes the Bessel process of dimension $(\beta+1)$.
At the critical dimension $\beta+1=2$,
it seems that any $F<0$, even if uniformly bounded,
may be strong enough to drive the solution $ Y $ to $0$ within a finite time.
However, for the type of $ F $ we consider here,
with $ F(0,s)=0 $ and $ F(\Cdot,s) $ being uniformly Lipschitz, 
we have
\begin{align}\label{eq:1D:Lips}
	\sup_{s\in[t',t]} \sup_{y\geq 0} |y^{-1} F(y,s)| 
	= 
	\sup_{s\in[t',t]} \sup_{y\geq 0} \absBK{ \frac{F(y,s)-F(0,s)}{y-0} } < \infty,
	\
	\forall t>0.
\end{align}
This in particular implies that $ |F(y,s)| \to 0 $ linearly (in $ y $) as $ y\to 0 $,
which suffices for $ Y \in C_+([t',\infty)) $.
The same applies for \eqref{eq:DBMfc} in the following.
\end{remark}
\begin{proof}
To show the uniqueness,
with $ F(\Cdot,s) $ being uniformly Lipschitz, 
the only problem is $ \beta/y $ not being Lipschitz at $ y=0 $.
This problem is solved by the standard localization argument: 
by first considering the localized process $ Y(t\wedge T^\delta) $, $ T^\delta:= \inf\{t:Y(t)>\delta\} $,
proving the uniqueness for $ t\in[0,T^\delta] $
(by Gronwall's inequality),
and letting $ \delta\to 0 $, (whence $ T^\delta\to\infty $ by the assumption $ Y\in C_+([t',\infty)) $).

As for existence, 
following the standard argument (c.f.\ \cite[Lemma 4.3.3]{anderson10}),
we construct a solution $ Y^\delta $ up to the first hitting time
$ S^\delta $ of any given level $ \delta>0 $.
With pathwise uniqueness, $ Y^\delta $, $ \delta>0 $, are consistent for different values of $ \delta $,
so it suffices to show $ S^\delta \to \infty $ as $ \delta\to 0 $.
If $ F \geq 0 $, this is easily achieved by comparing 
$ Y^\delta $ and $ Q^{t'} $ on $ [t',S^{\delta}] $.
Indeed, if $ F\geq 0 $, 
by Lemma~\ref{lem:comp} we have $ Q^{t'} \leq_{[t',S^\delta]} Y^\delta $.
With $ Q^{t'} \in C_+((t',\infty)) $, this implies $ S^\delta\to\infty $.

For the general case, $ F \not\geq 0 $,
we show $ S^\delta\to\infty $
by the method of Lyapunov function (c.f.\ \cite[Lemma 4.3.3]{anderson10}).
Applying Ito's formula to the semimartingale $ \log (Y^\delta(\,\Cdot\wedge S^{\delta})) $,
we obtain
\begin{align}
	\label{eq:LyapBes}
	\log \BK{ Y^\delta (t\wedge S^\delta) }
	&= 
	\log y^\ic + \text{MG}
	+ \int_0^{t\wedge S^{\delta}} 
	\BK{
		\frac{\beta-1}{\BK{Y^\delta(s)}^2} 
		+ (Y^\delta(s))^{-1} F(Y^\delta(s),s)
		}ds,
\end{align}
where MG is a martingale with zero mean.
Further localizing \eqref{eq:LyapBes} w.r.t.\ $ T^r_1\wedge T^r_2 $,
where
$
	T_1^r := \inf \{ t\geq 0: Y^\delta(t) > r \}
$
and
$
	T_2^r := \inf \{ t\geq 0: \sup_{y\geq 0} \{|y^{-1}F(y,t)|\} > r \},
$
and taking expectation of both sides,
with $ \beta\geq 1 $, we arrive at
\begin{align*}
	\log r + (\log\delta) \bbP(S^\delta< t\wedge T^r_1\wedge T^r_2)
	\geq
	\log y^\ic - tr.
\end{align*}
From this, $ \bbP(\lim_{\delta\to 0} S^\delta=\infty) =1 $
follows by letting $ \delta\to 0 $, $ t\to\infty $ and $ r\to\infty $ in order,
\emph{provided}
$
	\bbP\BK{ \lim_{r\to\infty} T^r_1=\infty } = 1
$
and	
$	
	\bbP\BK{ \lim_{r\to\infty} T^r_2=\infty } = 1.
$
The latter follows immediately from \eqref{eq:1D:Lips}.
To show the former,
we apply the preceding construction for the case $ F\geq 0 $
to obtain the $ C_+([t',\infty)) $-valued process $ Y' $ such that
\begin{align*}
	Y'(t) = Y^\ic + W_a(t)-W_a(t') + \int_{t'}^t \BK{ \frac{\beta}{Y'(s)} + F(Y'(s),s)_{\pmb{+}} } ds,
\end{align*}
where $ F(y,s)_{\pmb{+}} $ denote the positive part of $ F(y,s) $.
Note that $ F(y,s)_{\pmb{+}} $ indeed meets the prescribed conditions of this Lemma,
and is in particular uniformly Lipschitz because
\begin{align*}
	\sup_{y\neq y'\geq 0} \frac{|F(y,s)_{\pmb{+}}-F(y',s)_{\pmb{+}}|}{|y-y'|} 
	\leq 
	\sup_{y\neq y'\geq 0} \frac{|F(y,s)-F(y',s)|}{|y-y'|}.
\end{align*}
With $ F(y,s) \leq F(y,s)_{\pmb{+}} $,
by Lemma~\ref{lem:comp} we have $ Y^\delta(\Cdot) \leq_{[t',S^\delta]} Y'(\Cdot) $,
thereby concluding $ \bbP(\lim_{r\to\infty} T^r_1=\infty) = 1 $.
\end{proof}

We next consider the equation \eqref{eq:DBMfc} as follows,
which is a finite-dimensional version of \eqref{eq:DBMg} with external forces.
For $\calA\subset\bbR$, we let
$\f^{\calA}_a(y,\bfz)$ and $\fS^\calA_a(\bfy)$ denote the restriction of 
$\f_a(y,\bfz)$ and $\fS_a(\bfy)$ onto $[0,\infty)\times(0,\infty)^{\calA\cap\hZ}$,
\begin{align}
	&
	\label{eq:psiA}
	\f^{\calA}_a(y,\bfz)
	:=
	\frac12 \sum_{i\in\calA, |i-a|>1}
	\frac{ y }{ z_{(a,i)} (y+z_{(a,i)}) },
\\
	&
	\label{eq:fSA}
	\eta^\calA_a(\bfy):=\tfrac{1}{y_a} -\f_a^\calA(y_a,\bfy),
\end{align}
which indeed satisfy the following analog of \eqref{eq:psi:mono:y}--\eqref{eq:psi:mono:z},
\begin{align}
	&
	\label{eq:psiA:mono:y}
	\f^{\calA}_a(y,\bfz) \leq \f^{\calA}_a(y',\bfz), \text{ for } y \leq y',
\\
	&
	\label{eq:psiA:mono:z}
	\f^{\calA}_a(y,\bfz) \geq \f^{\calA}_a(y,\bfz'), \text{ for } \bfz \leq \bfz'.
\end{align}
By abuse of notation,
we let $u$ and $\tilu$, defined as in \eqref{eq:u}--\eqref{eq:tilu},
act on the space $\Weyl^{[i_1,i_2]}$,
whereby $ \tilu: \Weyl^{[i_1,i_2]} \to (0,\infty)\times (0,\infty)^{(i_1,i_2)\cap\hZ} $
is also a bijection.

\begin{lemma}\label{lem:DBMfc}
Let $i_1\leq i_2\in\bbZ$, $\calI:=(i_1,i_2)\cap\hZ$, $t'\geq 0$,
$\bfZ^*\in C_+([t',\infty))^\calI$ be $ \Wfil $-adapted.
For any $\Wfil_{t'}$-measurable $\bfY^\ic\in(0,\infty)^\calI$,
the equation
\begin{align}\label{eq:DBMfc}
\begin{split}
	&
	Y_a(t) = Y^\ic_a(t') + (W_a(t)-W_a(t'))
\\
	&
	\quad
	+ 
	\beta \int_{t'}^t 
	\BK{ \fS^{\calI}_a(\bfY(s)) + Y_a(s) Z^*_a(s) } ds,
	\quad
	t\geq t', \
	a \in \calI
\end{split}
\end{align}
has a $ C_+([t',\infty))^\calI$-valued, 
$ \Wfil $-adapted solution starting from $\bfY^\ic$,
which is the unique $ C_+([t',\infty))^\calI$-valued solution in the pathwise sense.
\end{lemma}

\begin{proof}
The uniqueness follows by the same argument as in the proof of Lemma~\ref{lem:oneD}.
As for the existence,
following the proof of Lemma~\ref{lem:oneD},
we construct the solution $ \bfY^\delta $ up to the first hitting 
time $ S^\delta:= \inf\{t\geq 0: Y^\delta_a(t)<\delta,$ for some $ a\in\calI \} $,
and then using the method of Lyapunov function to show $ S^\delta\to \infty $.
Recall that $ y_{(i,j)}:=\sum_{a\in(i,j)} y_a $.
With $ \xi(\bfy) := \sum_{(i,j)\subset(i_1,i_2)} \log y_{(i,j)} $ being the Lyapunov function,
applying Ito's formula to the semimartingale $ \xi(\bfY^\delta(\,\Cdot\wedge S^{\delta})) $,
we obtain
\begin{align*}
	\xi(\bfY^\delta(t\wedge S^{\delta}))
	=
	\xi(\bfy^\ic) + \text{MG}
	+ 
	\int_0^{t\wedge S^{\delta}}
	\BK{ f_1(\bfY^\delta(s)) + f_2(\bfY^\delta(s),\bfZ^*(s)) } ds,
\end{align*}
where MG is a martingale with zero mean, and
\begin{align*}
	&
	f_1(\bfy) 
	:= 
	\sum_{(i,j)\subset(i_1,i_2)}
	\Big( 
		\frac{1}{y_{(i,j)}}
		\sum_{a\in(i,j)} 
		\beta\fS^\calI_a(\bfy)
		- \frac{1}{(y_{(i,j)})^{2}} 
	\Big),
\\
	&
	f_2(\bfy,\bfz) :=
	\sum_{(i,j)\subset(i_1,i_2)} \frac{\beta}{y_{(i,j)}} \sum_{a\in(i,j)} y_a z_a.
\end{align*}
Following \cite[p 252]{anderson10},
one obtains $f_1(\bfy) = (\beta-1)2^{-1} \sum_{i,j: i\neq j }(x_j-x_i)^{-2} >0$,
where $ \bfx := \tilu^{-1} (0,\bfy) $.
Let
$ T^r_1 := \inf \{ t\geq 0 : Y_a^\delta(t) > r, $ for some $ a\in\calI \} $,
and
$
	T^r_2 := \inf \{ t\geq 0 : f_2(\bfY^\delta(t),\bfZ^*(t)) > r \}.
$
With $ f_1(\bfy) > 0 $,
similar to the proof of Lemma~\ref{lem:oneD},
it now suffices to show 
$ \bbP( \lim_{r\to\infty} T^r_1=\infty)=1 $ 
and 
$ \bbP( \lim_{r\to\infty} T^r_2=\infty)=1 $. 
The former, similar to the proof of Lemma~\ref{lem:oneD}, 
is proven by comparing $ \bfY^\delta $ to the process $ \bfY' $,
defined as the unique solution (given by Lemma~\ref{lem:oneD}) of
\begin{align*}
	Y'_a(t) = Y'_a(t') + (W_a(t)-W_a(t'))
	+ 
	\beta \int_{t'}^t 
	\sqBK{ \BK{Y'_a(s)}^{-1} + Y'_a(s) Z^*_a(s) } ds,
	\quad
	a\in\calI.
\end{align*}
As for the latter, with $ \tfrac{1}{y_{(i,j)}} y_a \leq 1 $, $ \forall a\in(i,j) $,
we have $ |f_2(\bfy,\bfz)| \leq \beta |\calI|^3 \sum_{a\in(i,j)} |z_a| $.
From this, $ \bbP( \lim_{r\to\infty} T^r_2=\infty)=1 $ follows
since $ \bfZ^*\in C([t',\infty))^\calI $.
\end{proof}

\begin{remark}\label{rmk:DBMuniqe}
The preceding proof of pathwise uniqueness depends only on Gronwall's inequality,
so the \emph{uniqueness} in fact holds more generally for \emph{random} $ \calI=\fkI $,
where $ \fkI $ is not necessarily independent of $ \bfW $.
\end{remark}

Next,
we establish a comparison principle for the equation \eqref{eq:DBMfc}.
To this end, for $\calI\subset\hZ$, 
we let
\begin{align*}
	&
	\bfy \leq^{\calI} \bfy' \text{ if and only if } y_a \leq y'_a, \ \forall a\in\calI,
\\
	&
	\bfy \leq^{\calI}_{[t',t'']} \bfy' \text{ if and only if } y_a(t) \leq y'_a(t), \ \forall a\in\calI, \ t\in[t',t'']
\end{align*}
denote the restriction of \eqref{eq:order}--\eqref{eq:orderT} onto $ \calI $ and $ [t',t''] $.

\begin{lemma}\label{lem:DBMcomp}
Fixing $t'<t''\in[0,\infty)$, 
$I_1<I_2\in\bbZ$ (possibly random),
we let $\fkI:=(I_1,I_2)\cap\hZ$,
$\bfZup$ and $\bfZlw\in C([t',t''])^\fkI$,
and $\bfYup$, $\bfYlw$ be the $ C_+([t',t''])^\fkI$-valued solutions of \eqref{eq:DBMfc}
with the respective external forces $\bfZup$ and $\bfZlw$, i.e.\
\begin{align}
	\notag
	\Yup_a(t) = \Yup_a(t') &+ \BK{ W_a(t) - W_a(t') } 
\\
	\label{eq:DBMcomp:Y}
	&+
	\beta \int_{t'}^t \BK{ \fS^{\fkI}_a(\bfYup(s)) +\Yup_a(s) \Zup_a(s) } ds,
	\quad
	t \in [t',t''], \ a\in\fkI,
\\
	\notag
	\Ylw_a(t) = \Ylw_a(t') &+ \BK{ W_a(t) - W_a(t') } 
\\
	\label{eq:DBMcomp:Y'}
	&+
	\beta \int_{t'}^t \BK{ \fS^{\fkI}_a(\bfYlw(s)) + \Ylw_a(s) \Zlw_a(s) } ds,
	\quad
	t \in [t',t''], \ a\in\fkI.
\end{align}
If $\bfZlw \leq^\fkI_{[t',t'']} \bfZup $,
and $\bfYlw(t') \leq^\fkI \bfYup(t')$,
then
\begin{align}\label{eq:DBMcomp}
	\bfYlw \leq^\fkI_{[t',t'']} \bfYup ,
	\quad
	\text{  almost surely.}
\end{align}
\end{lemma}

\begin{remark}\label{rmk:detComp}
Note that here  
we do \emph{not} assume $W_a(\Cdot)$ conditioned on $(I_1,I_2)$
is a Brownian motion or even a martingale.
\end{remark}

\begin{proof}
For each finite, deterministic interval $\calI:=(i_1,i_2)\cap\hZ$,
we consider the iteration sequence $ \{\bfY^{(n),\calI}\}_{\bbZ_{\geq 0}} \subset C_+([t',\infty))^\calI $ as follows,
which is the analog of \eqref{eq:iter:} for \eqref{eq:DBMcomp:Y}:
\begin{subequations}\label{eq:iterf:}
\begin{align}
	&
	\label{eq:iterf0}
	Y^{{(0)},\calI}_a(t) 
	= \Yup_a(t') + (W_a(t)-W_a(t')) + \beta \int_{t'}^t \frac{1}{Y^{{(0)},\calI}_a(s)}  ds,
	\quad
	a\in\calI,
\\
	&
	\label{eq:iterf}
	Y^{{(n)},\calI}_a(t) 
	= \Yup_a(t') + (W_a(t)-W_a(t'))
\\
	\notag
	&
	\quad+ 
	\beta \int_{t'}^t
	\Bigg(
		\frac{1}{Y^{{(n)},\calI}_a(s)} 
		- \f^{\calI}_a(Y^{{(n)},\calI}_a(s), \bfY^{(n-1),\calI}(s)) 
		+ Y^{{(n)},\calI}_a(s) \Zup_a(s) 
	\Bigg)ds,
	\quad a\in\calI.
\end{align}
\end{subequations}
Such a sequence is constructed inductively
by applying Lemma~\ref{lem:oneD}
for $F(y,s)=0$ (when $ n=0 $),
and for $ F(y,s)=\psi^\calI_a(y,\bfY^{(n-1),\calI}(s))+y \Zup_a(s) $ (when $ n>0 $).
In particular, such $ F(y,s) $ indeed satisfies $ F(0,s)=0 $,
and by the same calculation as in \eqref{eq:psiLip},
$y\mapsto F(y,\bfy'(s))$ is uniform Lipschtiz continuity for $ \bfy'\in C_+([t',\infty))^\calI $.

With $ \bfY^{(n),\calI} $ solving \eqref{eq:iterf:}
and $ \bfYlw $ solving \eqref{eq:DBMcomp:Y'},
following the comparison argument as in the proof of Proposition~\ref{prop:iterDecr},
we obtain that
\begin{align}
	\label{eq:DBMcomp:YnY'}
	\bfY^{{(0)},\fkI} (\Cdot) 
	\geq^\fkI_{[t',t'']} 
	\bfY^{{(1)},\fkI} (\Cdot) 
	\geq^\fkI_{[t',t'']} 
	\ldots
	\geq^\fkI_{[t',t'']} 
	\bfYlw(\Cdot).
\end{align}
With this, defining the limiting process $ \bfY^{*} (t) := \lim_{n\to\infty } \bfY^{{(n)},\fkI} (t) $,
we have
\begin{align}\label{eq:Y*Ylw}
	\bfY^{*} \geq^\fkI_{[t',t'']} \bfYlw.
\end{align}
In \eqref{eq:iterf}, letting $ n\to\infty $,
by the dominated convergence theorem we obtain
\begin{align*}
	&
	\int_{t'}^t
	\Bigg(
		\frac{1}{Y^{{(n)},\calI}_a(s)} 
		- \f^{\calI}_a(Y^{{(n)},\calI}_a(s), \bfY^{(n-1),\calI}(s)) 
		+ Y^{{(n)},\calI}_a(s) \Zup_a(s) 
	\Bigg)ds
\\
	&
	\quad
	\longrightarrow
	\int_{t'}^t
	\BK{
		\fS^{\calI}_a(\bfY^{*}(s)) 
		- Y^{*}_a(s) \Zup_a(s) 
	}
	ds,
	\quad 
	\forall a\in\calI.
\end{align*}
Hence $ \bfY^* $ solves \eqref{eq:DBMcomp:Y}.
This automatically implies that $ \bfY^* $ is $ C([t',\infty)) $-valued,
and with \eqref{eq:Y*Ylw}, we actually have $ \bfY^*\in C_+([t',\infty)) $.
As the $ C_+([t',\infty)) $-valued solution of \eqref{eq:DBMcomp:Y} 
is unique (by Lemma~\ref{lem:DBMfc} and Remark~\ref{rmk:DBMuniqe}),
we must have $ \bfY^* = \bfYup $.
Combining this with \eqref{eq:Y*Ylw}, we conclude \eqref{eq:DBMcomp}.
\end{proof}

\section{Uniqueness, Proof of Proposition~\ref{prop:unique2}}
\label{sect:unique}

Fix $ \bfy^\ic\in\Ysp((\alpha,\rho)\cap\Rsp(p)) $
and $ \bfYlw(\Cdot) \leq \bfYup(\Cdot) \in(\YTsp(\alpha,\rho)\cap\RTsp(p)) $ 
as in Proposition~\ref{prop:unique2}.
Recall that $ E_{\calI}(s) := \sum_{a\in\calI} (\Yup_a(s)-\Ylw_a(s)) $
and that $ L^\pm_i(t) $ is defined as in \eqref{eq:L}.
We begin by proving \eqref{eq:uniEq:est}.
\begin{proof}[Proof of \eqref{eq:uniEq:est}]
By abuse of notation, we let $ \phi_i(\bfy) := \phi_i(\tilu^{-1}(0,\bfy)) $.
Summing \eqref{eq:DBMg} over $ a\in(i_1,i_2) $, and using \eqref{eq:phifS},
we obtain the equation
\begin{align}\label{eq:sumY:Eq}
	\sum_{a\in(i_1,i_2)} Y_a(s) \Big|^{s=t}_{s=t'}
	=
	( B_{i_2}(s) - B_{i_2}(s) ) \Big|^{s=t}_{s=t'}
	+
	\beta \int_{t'}^{t} \BK{ \phi_{i_2}\BK{\bfY(s)} - \phi_{i_1}\BK{\bfY(s)} } ds,
\end{align}
for a generic $ \YTsp(\alpha,\rho) $-valued solution $ \bfY $ of \eqref{eq:DBMg}.
Now, substitute $ \bfY $ for $ \bfYup $ and for $ \bfYlw $ in \eqref{eq:sumY:Eq},
and take the difference of the results.
With
$
	\phi_{i}\BK{\bfYup(s)} - \phi_i(\bfYlw(s))
	=
	L^+_{i}(s) - L^-_{i}(s),
$
we conclude \eqref{eq:uniEq:est}.
\end{proof}

Fixing $ m\in\bbZ_{>0} $ (which will be sent to $ \infty $ later),
for $ i\in [\pm m, \pm 2m] $ we decompose $ L^\pm_i(s) $ into the long-range interaction
\begin{align}\label{eq:L2}
	\tilL^{\pm}_{i,m}(s) :=
	\frac12
	\sum_{j \in(\pm 3m,\pm\infty)}
	\frac{ \Yup_{(i,j)}(s)-\Ylw_{(i,j)}(s) }{\Yup_{(i,j)}(s)\Ylw_{(i,j)}(s)},
\end{align}
and the short-range interaction
\begin{align}\label{eq:L1}
	L^{\pm}_{i,m}(s) :=
	\frac12
	\sum_{j\in(i,\pm 3m]}
	\frac{ \Yup_{(i,j)}(s)-\Ylw_{(i,j)}(s) }{\Yup_{(i,j)}(s)\Ylw_{(i,j)}(s)}.
\end{align}
The main step of the proving Proposition~\ref{prop:unique2} is the following estimates.
Recall the definition of $ q(t,1) $ from \eqref{eq:QLLN}.
\begin{lemma}\label{lem:Lbd}
\begin{enumerate}[label=(\alph*)]
	\item[]
	\item \label{enu:Llong} 
		For any $ t\geq 0 $,
		we have
		\begin{align}\label{eq:Lest:L2}
			\tilL^{\pm}_m := \sup_{s\in[0,t]} \sup_{i \in [\pm m, \pm 2m] } \curBK{ \tilL^{\pm}_{i,m}(s) } \to 0,
			\quad
			\text{ almost surely.}
		\end{align}
	\item \label{enu:Lshort} 
		For any $ t'<t''\in[0,\infty) $ such that $ q(t''-t',1)<\tfrac{\rho}{2} $,
		we have
		\begin{align}\label{eq:Lest:i}
			\lim_{m\to\infty} \inf_{i\in[\pm m, \pm 2m]} \int_{t'}^{t''} L^{\pm}_{i,m}(s) ds  = 0,
			\quad
			\text{ almost surely.}			
		\end{align}
\end{enumerate}
\end{lemma}
\begin{proof}[Proof of Part\ref{enu:Llong}]
Fix arbitrary $ i\in [\pm m,\pm 2m] $ and $ s\in[0,t] $.
With $\bfYup$, $\bfYlw\in\YTsp(\alpha,\rho)$,
we have
\begin{align}\label{eq:Lest:U1}
	\sup_{s\in[0,t], j\in\bbZ } 
	\big\{ \av_{(0,j)}(\bfYup(s)-\bfYlw(s)) |j|^{\alpha} \big\}
	=: N
	<
	\infty,
\end{align}
and $ \av_{(0,j)} (\bfYlw(s)) \to \rho $,
uniformly in $ s\in[t',t''] $, as $ |j|\to\infty $.
Combining the latter with $ \Ylw_{(\pm 2m,j)}(s) = \Ylw_{(0,j)}(s) - \Ylw_{(0,\pm 2m)}(s) $,
we obtain
\begin{align*}	
	\liminf_{m\to\infty} \inf_{j\in(\pm 3m,\pm\infty)} \big\{ |j|^{-1}\undYlw_{(\pm 2m,j)}(0,t) \big\}
	\geq
	\tfrac{\rho}{3}.
\end{align*}
As $ \Ylw_a\in C_+([0,\infty)) $, $ \forall a\in\bbZ $,
from this we further deduce
\begin{align*}	
	\inf_{m\in\bbZ_{>0}} \inf_{j\in(\pm 3m,\pm\infty)} 
	\big\{ |j|^{-1}\undYlw_{(\pm 2m,j)}(0,t) \big\}
	=:
	D>0.
\end{align*}
Inserting this and \eqref{eq:Lest:U1} into \eqref{eq:L2},
we conclude
$
	\tilL^{\pm}_{i,m}(s) \leq \frac{N}{D^2} \sum_{j>3m} j^{-1-\alpha},
$
from which \eqref{eq:Lest:L2} follows.
\end{proof}

We proceed to proving Part\ref{enu:Lshort}.
The preceding argument yields a bound on $ \tilL^{\pm}_{i,m}(s) $ 
which is uniform over $ i\in[\pm m, \pm 2m] $.
Such a uniform bound cannot be achieved for the short-range interaction $ L^\pm_{i,m}(s) $,
because, for example, 
$ \bfYlw \in \YTspl(\alpha,\rho) $ does \emph{not} imply $ \bfYlw_{(i,j)}(s) > |j-i|/c $ for small $ |j-i| $.
Instead, we proceed by constructing certain `good' index set $ \fkG^\pm_{m,k}\neq\emptyset $,
such that $ L^\pm_{i,m}(s) $ is controlled for $ i\in\fkG^\pm_{m,k} $.

To construct $ \fkG^\pm_{m,k} $,
letting $p'\in(1,\infty)$ denote the H\"{o}lder conjugates of $p$,
i.e.\ $1/p+1/p'=1$,
for fixed $ s\in[0,\infty) $ and $ m\in\bbZ_{>0} $,
we consider the set
\begin{align}\label{eq:Lest:A}
	\fkA_{m}(s):= \curBK{a\in\hZ: |\Yup_a(s)-\Ylw_a(s)| \geq |m|^{-\alpha/(3p') } }
\end{align}
of `bad' indices, 
where the corresponding terms in the numerator of \eqref{eq:L1} may be large at time $ s $.
For $\calA\subset\hZ$, $i,i'\in\bbZ$, let
\begin{align*}
	g_{(i,i')}(\calA) := \sup_{j\in(i,i']} 
	\frac{ |(i,j)\cap\calA| }{ |j-i| }
\end{align*}
denote the maximal cumulative occurrence frequency of $\calA$
when searching to the right (when $ i'>i $) or left (when $ i'<i $) 
over the interval $ (i,i') $, starting from $ i $.
Consider the set 
\begin{align}
	\label{eq:Lest:B}
	&
	\fkI^\pm_{m}(s) := \curBK{ 
		i\in  [\pm m,\pm 3m]\cap\bbZ
		:
		g_{(i,\pm 3m)} (\fkA_{m}(s)) > m^{-\alpha/3}  
		}
\end{align}
of `bad' indices, 
where the occurrence of $ \fkA_{m}(s) $ may be large over the interval $ (\pm m,\pm 3m) $.
The sets $ \fkA_m(s) $ and $ \fkI^\pm_m(s) $ are constructed for a fixed $ s $.
We now fix $ t'<t'' $ as in Lemma~\ref{lem:Lbd},
let $ \calT_k := \{ t'+\frac{(t''-t')\ell}{k} \}_{\ell=1}^{k} $,
and consider the set
\begin{align}\label{eq:Lest:C}
	\fkN^\pm_{m,k}
	:= 
	\curBK{ i \in \bbZ :\frac{1}{k} \sum_{s\in\calT_k} \ind\curBK{i\in\fkI^\pm_m(s)} \leq m^{-\alpha/(3p')} },
\end{align}
consisting of `good' indices $ i $ such that $ \{ \fkI^\pm_m(s) \ni i \} $
occurs rarely alone the discrete samples $ s\in\calT_k $ of time.
The set $ \fkN^\pm_{m,k} $ is constructed for bounding the numerator in the expression \eqref{eq:L1}.
As for the denominator, we consider
\begin{align}\label{eq:h}
	&
	h_{(i,j)}(\bfy) := 
	\inf_{i'\in(i,j]\cap\bbZ} \av_{(i,i')} (\bfy),
\end{align}
and define
\begin{align}\label{eq:Lest:G}
	\fkG^\pm_{m,k}
	:= 
	\{ i\in [\pm m,\pm 2m]\cap\bbZ 
	: 
	i\in\fkN^\pm_{m,k}, h_{(i,\pm 3m)}(\bfundYlw(t',t'')) \geq \tfrac{\rho}{3} \}.
\end{align}

Let 
$
	L^{\pm,k}_{i,m} := \frac{t''-t'}{k} \sum_{s\in\calT_k} L^{\pm}_{i,m}(s)
$
denote the $ k $-th discrete approximation of $ \int_{t'}^{t''} L^{\pm}_{i,m}(s) ds $.
Having constructed $ \fkG^\pm_{m,k} $,
we proceed to establishing a bound on $ L^{\pm,k}_{i,m} $ for $ i\in\fkG^\pm_{m,k} $.
Let $P :=\sup_{m\in\bbZ} \av^p_{(0,m)} ( \bfbarY^\text{up}{(t',t'')} )$, 
which is almost surely finite as $\bfYup\in\RTsp(p)$.
\begin{lemma}\label{lemma:Lest:L1}
For all $ m,k\in\bbZ_{>0} $, there exists $ c=c(t''-t',\rho,p)<\infty $ such that
\begin{align}\label{eq:Gbd}
	 L^{\pm,k}_{i_*,m}
	\leq
	(1+P^{1/p}) \frac{c \log m }{m^{\alpha/(3p')}},
	\quad
	\forall i_*\in\fkG^\pm_{m,k}.
\end{align}
\end{lemma}
\begin{proof}
Fixing $ k,m\in\bbZ_{>0} $ and $ i_*\in\fkN^\pm_{m,k} $,
we let $ c<\infty $ denote a generic constant depending only on $ t''-t',\rho,p $.
We begin by bounding the expression $ L^{\pm}_{i_*,m}(s) $, for $ s\in\calT_k $,
to which end we consider separately the two cases 
\begin{enumerate*}[label=\itshape\roman*\upshape)]
\item $ \{ \fkI^\pm_m(s) \not\ni i_* \} $; 
and \item $ \{ \fkI^\pm_m(s) \ni i_* \} $.
\end{enumerate*}

\textit{i})
In \eqref{eq:L1},
using $ \bfYlw(s)\leq\bfYup(s) $ and $ h_{(i_*,\pm 3m)}(\bfundYlw(t',t'')) \geq \tfrac{\rho}{3} $,
we bound the denominator from below by $ (|j-i_*|\rho/3)^{2} $.
As for the numerator, 
we divide $ \Yup_{(i_*,j)}(s)-\Ylw_{(i_*,j)}(s)=\sum_{a\in (i,j) } (\Yup_{a}(s)-\Ylw_{a}(s)) $
into two sums subject to the constraints 
$ \{a\notin\fkA_m(s)\} $ and $ \{a\in\fkA_m(s)\} $.
The former sum, by \eqref{eq:Lest:A}, is bounded by 
$ m^{-\alpha/(3p')} |j-i_*|. $ 
As for the latter, we apply the H\"{o}lder inequality to obtain
\begin{align*}
	&
	\sum_{a\in(i_*,j)} 
	\BK{ |\Yup_{a}(s)-\Ylw_{a}(s)| } 
	\BK{ \ind\curBK{a\in\fkA_m(s)} }
\\
	&
	\quad \leq
	\Big( \sum_{a\in(i_*,j)} \Yup_{a}(s)^p \Big)^{1/p}
	\Big( \sum_{a\in(i_*,j)} \ind\curBK{a\in\fkA_m(s)} \Big)^{1/p'}
\\
	&
	\quad \leq
	\BK{ |j-i_*|P }^{1/p}
	\BK{ g_{(i_*,\pm 3m)}(\fkA_m(s)) \ |j-i_*|}^{1/p'}.
\end{align*}
With $ i_*\notin\fkI^\pm_m(s) $, we have $ g_{(i_*,\pm 3m)}(\fkA_m(s)) \leq m^{-\alpha/3} $,
so the last expression is further bounded by $ cP^{1/p} m^{-\alpha/(3p')}|j-i_*|$.
Combining the preceding bounds yields
\begin{align}\label{eq:L1bd:1}
	L^{\pm}_{i_*,m}(s) 
	\leq
	c (1+P^{1/p}) m^{-\alpha/(3p')}
	\sum_{j\in(i_*,\pm 3m]} \frac{ |j-i_*| }{ |j-i_*|^2 }
	\leq	
	c (1+P^{1/p}) m^{-\alpha/(3p')} \log m.
\end{align}

\textit{ii})
Using $ \bfYup(s) \geq \bfYlw(s) $ in \eqref{eq:L1},
we bound the $ j $-th term by $ 1/\Ylw_{(i_*,j)}(s) $.
This, with $ h_{(i_*,\pm 3m)}(\bfundYlw(t',t'')) \geq \tfrac{\rho}{3} $,
is further bounded by $ (|j-i_*|\rho/3)^{-1} $.
Consequently,
\begin{align}\label{eq:L1bd:2}
	L^{\pm}_{i_*,m}(s) 
	\leq
	c \log(m+1).
\end{align}

Although the bound \eqref{eq:L1bd:2} is undesired ($ \to\infty $ as $ m\to\infty $),
the corresponding case $ \{s\in\calT_k : \fkI^\pm_m(s) \ni i_* \} $ occurs at low frequency $ \leq m^{-\alpha/(3p')} $.
Hence
\begin{align}\label{eq:L1bd:22}
	\frac{t''-t'}{k} \sum_{s\in\calT_k} \ind\curBK{ \fkI^\pm_m(s)\ni i_* } L^{\pm}_{i_*,m}(s) 
	\leq
	c \log(m+1) m^{-\alpha/(3p')}.
\end{align}
Averaging \eqref{eq:L1bd:1} over $ s\in\calT_k $ for $ \{s\in\calT_k :\fkI^\pm_m(s) \not\ni i_* \} $,
and combining the result with \eqref{eq:L1bd:22},
we conclude \eqref{eq:Gbd}.
\end{proof}

Next, we show that $ \fkG_{m,k}^\pm $ is nonempty for all large enough $ m $.
\begin{lemma}\label{lem:fkG:exists}
We have 
$
	\liminf\limits_{ m\to \infty} \big( \inf\limits_{k\in\bbZ_{>0}} |\fkG_{m,k}^\pm| \big) \geq 1,
$
almost surely.
\end{lemma}
\noindent
With $ \fkG^\pm_{m,k} $ defined as in \eqref{eq:Lest:G},
proving $ |\fkG_{m,k}^\pm| \geq 1 $ requires finding $ i\in[\pm m, \pm 2m) $ 
such that $ h_{(i,\pm 3m)}(\bfy) \geq \frac{\rho}{3} $ for $ \bfy=\bfundYlw(t',t'') $.
This is conveniently reduced to estimating $ \av_{[\pm m,j)}(\bfy) $, $ j\in [\pm 2m,\pm 3m] $, by the following lemma.
\begin{lemma}\label{lem:goodset}
Let $\bfy\in [0,\infty]^\hZ$,
$i^+_1<i^+_2\leq i^+_3$ and $i^-_3\leq i^-_2 < i^-_1$,
where $i^\pm_1,i^\pm_2\in\bbZ$ and $i^\pm_3\in\bbZ\cup\{\pm\infty\}$.
If, for some $\gamma\in(0,\infty)$,
\begin{align}\label{eq:lem:goodset:assum+}
 	\av_{(i^\pm_1,i)} (\bfy) > \gamma , \quad\forall \ i\in[i^\pm_2,i^\pm_3]\cap\bbZ,
\end{align} 
then there exists $i^\pm_*\in[i^\pm_1,i^\pm_2)\cap\hZ$ 
such that $h_{(i^\pm_*,i^\pm_3)}(\bfy)\geq\gamma$.
\end{lemma}
\begin{proof} 
Without lost of generality we consider only the $ + $ case.
Let $ f $ be the counting function
$ f: \bbZ \to \bbR $, $i\mapsto \sum_{a\in(i^+_1,i)} y_a$,
and let $ \calL:= \{(x,\gamma(x-i^+_1)): x\in\bbR\} $ 
denote the straight line of slope $ \gamma $ passing through $ (i^+_1,0) $.
By \eqref{eq:lem:goodset:assum+},
the graph of $ f $ is above $ \calL $ for all $ i\in[i^+_2,i^+_3]\cap\bbZ $.
Hence, letting
\begin{align*}
	i^+_*:= \sup \curBK{ i\in[i^+_1,i^+_2] \cap\bbZ : (i,f(i)) \text{ is not above } \calL }
	\geq i^+_1,
\end{align*}
we clearly have $ \frac{f(i)-f(i^+_*)}{i-i^+_*} \geq \gamma $,
for all $ i\in[i^+_2,i^+_3]\cap\bbZ $,
which is equivalent to $h_{(i^+_*,i^+_3)}(\bfy)\geq\gamma$.
\end{proof}

\begin{proof}[Proof of Lemma~\ref{lem:fkG:exists}]
Fixing $ m,k\in\bbZ_{>0} $,
to simply notations, 
we omit the dependence on $ m,k $ of the index sets
(e.g.\ $ \fkN^\pm := \fkN^\pm_{m,k} $) and let
$ 
	\tilY^{\pm}_a := \undYlw_a(t',t'') \ind\{a\in\fkN^\pm\pm\frac12\}.
$
We show
\begin{align}\label{eq:fkG:goal}
	\av_{(\pm m,\pm j)}( \bftilY^\pm ) > \tfrac{\rho}{3}, \ \forall j\in [\pm 2m,\pm 3m],
	\quad
	\forall
	\text{ large enough } m.
\end{align}
This, by Lemma~\ref{lem:goodset} for $ (i^\pm_1,i^\pm_2,i^\pm_3) = (\pm m,\pm 2m, \pm 3m) $,
implies the existence of $ I^\pm\in[\pm m,\pm 2m)\cap\bbZ $ such that
$ h_{(I^\pm,\pm 3m)}(\bftilY^\pm) \geq \tfrac{\rho}{3} $.
For such $ I^\pm $, we have $ h_{(I^\pm,\pm 3m)}(\bfundYlw(t',t'')) \geq \frac{\rho}{3} $
and $ \tilY_{(I^\pm,I^\pm\pm 1)} \geq \frac{\rho}{3} >0 $.
The later implies $ I^\pm\in\fkN^\pm $, and therefore $ I^\pm\in\fkG^\pm $.
Hence, it suffices to prove \eqref{eq:fkG:goal}.

To the end of showing \eqref{eq:fkG:goal},
with $ \bftilY^\pm $ defined as in the preceding,
we begin by estimating $ |(\fkN^\pm)^c| $.
To this end, as $ \fkN^\pm $ is defined in terms of $ \fkA(s) $ and $ \fkI^\pm(s) $,
we first establish bounds on $ | \fkA(s) \cap (\pm m, \pm 3m) | $ and $ |\fkI^\pm(s)| $.
Fixing $ s\in[t',t''] $,
with $ N $ as in \eqref{eq:Lest:U1} and $ \fkA(s) $ as in \eqref{eq:Lest:A},
we have
\begin{align}\label{eq:Lest:Abd}
	| \fkA(s) \cap (\pm m, \pm 3m) | \leq | \fkA(s) \cap (0, \pm 3m) |
	\leq
	\tfrac{(3m)^{1-\alpha} N }{m^{-\alpha/(3p')}}
	\leq
	(3m)^{1-\frac{2\alpha}{3}} N.
\end{align}
Proceeding to bounding $ |\fkI^\pm_m(s)| $,
we require the following inequality:
for any finite $\calA\subset\hZ$, $n\in\bbZ_{>0}$, we have
\begin{align}\label{eq:topple}
	|\calI^\pm_n| \leq n|\calA|,
	\quad
	\text{ where }
	\calI^\pm_n
	:=
	\curBK{ i \in \bbZ : g_{(i,\pm\infty)}(\calA) > n^{-1} } \subset \hZ.	
\end{align}
To prove this inequality,
we image a pile of $ n $ particles at each site of $ \calA $,
and topple the particles to the left (for $ + $) or right (for $ - $) in any order,
so that each sites of $ \hZ $ contains at most one particle.
Letting $ \calA^{\pm}_n\subset\hZ $ denote the resulting set of particles,
we clearly have $ \calI^\pm_n \subset (\calA^{\pm}_n\mp\frac12) $ and $ |\calA^{\pm}_n| = n |\calA| $,
thereby concluding \eqref{eq:topple}.
Now, with $ \fkI^\pm(s) $ as in \eqref{eq:Lest:B},
combining \eqref{eq:Lest:Abd} and 
\eqref{eq:topple} for $ \calA= \fkA(s)\cap(\pm m,\pm 3m) $ and $n=\ceilBK{m^{\alpha/3}}$,
we arrive at
\begin{align*} 
	|\fkI^\pm(s)|
	\leq
	\ceilBK{ m^{\alpha/3} } |\fkA(s)\cap(\pm m,\pm 3m)|
	\leq
	6N m^{1-\frac{\alpha}{3}}.
\end{align*}
Now, with $ \fkN^\pm $ as in \eqref{eq:Lest:C}, we have
$
	\ind\{ i\notin\fkN^\pm \}
	\leq
	m^{\frac{\alpha}{3p'}} \tfrac{1}{k} \sum_{s\in\calT_k} \ind\{ i\in\fkI^\pm(s) \}.
$
Summing both sides over $ i\in\bbZ $, we arrive at
\begin{align}\label{eq:Lest:Cbd}
	|(\fkN^\pm)^c|
	\leq
	\frac{1}{k} \sum_{s\in\calT_k} |\fkI^\pm(s)| m^{\frac{\alpha}{3p'}}
	\leq
	6N m^{1-\alpha'},
\end{align}
where $ \alpha':= \frac{\alpha}{3}(1-\frac{1}{p'})>0 $.

We proceed to proving \eqref{eq:fkG:goal}.
Fix $ j\in[\pm 2m,\pm 3m] $.
With $ \bftilY^\pm $ defined as in the proceeding,
we have
\begin{align*}
	\av_{(\pm m, j)}(\bftilY^\pm)
	=
	\av_{(\pm m, j)}(\bfundYlw(t',t''))
	-
	\frac{1}{|j\mp m|} \sum_{(\pm m,j)} \undYlw_{a}(t',t'') \ind\curBK{ a\in(\fkN^\pm)^c\pm\tfrac12 }.
\end{align*}
For the last term, with $ \frac{1}{|j\mp m|}|(\fkN^\pm)^c| \leq 6N m^{-\alpha'}$ (by \eqref{eq:Lest:Cbd})
and $ \bfYlw\in\RTsp(p) $, we have
\begin{align*}
	\frac{1}{|j\pm m|} \sum_{(\pm m,j)} \undYlw_{a}(t',t'') \ind\curBK{ a\in(\fkN^\pm)^c\pm\tfrac12 } 
	\xrightarrow{m\to\infty} 0,
	\
	\text{ uniformly in } j\in[\pm 2m,\pm 3m].
\end{align*}
Consequently, to prove \eqref{eq:fkG:goal}, we may and shall replace $ \bftilY^\pm $ with $ \bfundYlw(t',t'') $.
Applying the continuity estimate \eqref{eq:BesComp} for $ Y^*=\Ylw_a $, 
we have
\begin{align*}
	\av_{(\pm m,j)} ( \bfundYlw(t',t'') )
	&\geq	
	\av_{(\pm m,j)} ( \bfYlw(t'') ) - \av_{(\pm m,j)} (\bfQ^{t'}(t'')).
\end{align*}
With $ \bfYlw\in\YTsp(\alpha,\rho) $, 
the first term on the r.h.s.\ converges to $ \rho $ as $ m\to\infty $,
uniformly in $ j\in[\pm 2m,\pm 3m] $.
With $ q(t''-t',1)\leq\frac{\rho}{2} $, by \eqref{eq:QLLN},
the last term contributes $ \geq - \frac{\rho}{2} $ as $ m\to\infty $.
Combining the preceding we conclude \eqref{eq:fkG:goal}.
\end{proof}

Based on Lemma~\ref{lemma:Lest:L1}--\ref{lem:fkG:exists},
we now prove Lemma~\ref{lem:Lbd}\ref{enu:Lshort}.
\begin{proof}[Proof of Lemma~\ref{lem:Lbd}\ref{enu:Lshort}]
By Lemma~\ref{lemma:Lest:L1}--\ref{lem:fkG:exists},
we have that
\begin{align*}
	\inf_{i\in[\pm m, \pm 2m]} 
	L^{\pm,k}_{i,m} \leq (1+P^{1/p})c m^{-\alpha/(3p')} \log (1+m).
\end{align*}
Since the constant $ c $ does not depend on $ k $,
upon letting $k\to\infty$, by the continuity of $\Yup_a(\Cdot)$ and $\Ylw_a(\Cdot)$,
the l.h.s.\ tends to $ (\inf_{i\in[\pm m, \pm 2m]} \int_{t'}^{t''} L^{\pm}_{m,i}(s) ds) $.
Consequently, further letting $ m\to\infty $, we complete the proof.
\end{proof}

\begin{proof}[Proof of Proposition~\ref{prop:unique2}]
Fixing arbitrary $ t>0 $,
we partition $ [0,t] $ into $ j_* $ equally spaced subintervals 
$ [t_{j-1},t_{j}] $, $ j=1,\ldots,j_* $, so that $ q(t/j_*,1)<\tfrac{\rho}{2} $
(for $ q(t,1) $ as in \eqref{eq:QLLN}).
By \eqref{eq:uniEq:est}, we have
\begin{align*}
	\barE{(i_1,i_2)}(t_{j-1},t_{j}) 
	\leq
	E_{(i_1,i_2)}(t_{j-1})
	+
	\beta \int_{t_{j-1}}^{t_{j}} \BK{ L^+_{i_2}(s) + L^-_{i_1}(s) } ds,
\end{align*}
where $ \barE{(i_1,i_2)}(t_{j-1},t_{j}) := \sup_{s\in[t_{j-1},t_j]} E_{(i_1,i_2)}(s) $ 
by our convention.
Letting $ (i_1,i_2)\to(-\infty,\infty) $ 
and combining the result for $ j=1,\ldots,j_* $,
we obtain
\begin{align}\label{eq:unique:pre}
	\barE{(-\infty,\infty)}(0,t) 
	\leq
	\beta 
	\sum_{j=1}^{j_*}
	\liminf_{i\to\infty}
	\int_{t_{j-1}}^{t_{j}} \BK{ L^+_i(s) + L^-_{-i}(s) } ds.
\end{align}
Now, with $ L^\pm_i(s) = L^\pm_{i,m}(s) + \tilL^{\pm}_{i,m}(s) $,
we have
\begin{align*}
	\inf_{i\in[\pm m, \pm 2m]} \int_{t_{j-1}}^{t_{j}} L^\pm_i(s) ds
	\leq
	\inf_{i\in[\pm m, \pm 2m]} \int_{t_{j-1}}^{t_{j}} L^\pm_{i,m}(s) ds
	+
	\sup_{i\in[\pm m, \pm 2m]} \int_{t_{j-1}}^{t_{j}}\tilL^\pm_{i,m}(s) ds.	
\end{align*}
Applying Lemma~\ref{lem:Lbd} to bound the r.h.s., letting $ m\to\infty $,
and plugging the result in \eqref{eq:unique:pre},
we thus conclude $ \barE{(-\infty,\infty)}(0,t) =0 $.
\end{proof}

\section{Existence, Proof of Proposition~\ref{prop:existg}}
\label{sect:exist:special}

Fix $ \gamma>0 $ and $ \bfy^\ic\in[\gamma,\infty)^\hZ $ as in Proposition~\ref{prop:existg}.
We consider first the special case of \emph{equally spaced} initial condition,
$ \bfz^\ic := \bfgamma = (\ldots,\gamma,\gamma,\ldots) $,
and construct the corresponding iteration sequence 
$ \{\bfZ^{(n)}\}_{n\in\bbZ_{\geq 0}} $.
For $ n=0 $, $ \bfZ^{(0)} $ is the $ \Wfil $-adapted
Bessel process (as in \eqref{eq:iter0}) starting at $ \bfgamma $.
Recalling $ \Yspl(\gamma) $ and $ \YTspl(\gamma) $
are defined as in \eqref{eq:Yspl}--\eqref{eq:YTspl},
we check that $ \bfZ^{(0)} $ is $ \YTspl(\gamma) $-valued.
\begin{lemma}\label{lem:iter0YT}
We have $ \bfZ^{(0)} \in \YTspl(\gamma) $.
\end{lemma}
\begin{proof}
Fix arbitrary $ t\geq 0 $.
With $ Z^{(0)}_a $ satisfying \eqref{eq:iter0} and $ Z^{(0)}_a(0)=\gamma $,
averaging \eqref{eq:iter0} over $a\in(0,m)$ using $W_a(t)=B_{a+1/2}(t)-B_{a-1/2}(t)$, we obtain
\begin{align*}
	\inf_{s\in[0,t]} \curBK{ \av_{(0,m)} \BK{ \bfZ^{(0)}(s) } } - \gamma 
	\geq
	-\sup_{s\in[0,t]} |m|^{-1} \absBK{ B_m(s) - B_0(s) }.
\end{align*}
Upon letting $|m|\to\infty$, the r.h.s.\ tends to zero,
whereby $ \bfZ^{(0)} \in \YTspl(\gamma) $ follows.
\end{proof}
\noindent
For $ n>0 $, we construct the $ \Wfil $-adapted, $ \YTspl(\gamma) $-valued process $ \bfZ^{(n)} $ 
by induction on $ n $, using Lemma~\ref{lem:oneD}.
That is, fixing $ n>0 $, for each $ a\in\hZ $,
we let $ Z^{(n)}_a $ be the unique solution of \eqref{eq:oneD'}
for $ F(y,s) = - \f_a(y,\bfZ^{(n-1)}(s)) $,
assuming $ \bfZ^{(n-1)} $ is the $ \Wfil $-adapted, $ \YTspl(\gamma) $-valued process
satisfying \eqref{eq:iter:}.
For Lemma~\ref{lem:oneD} to apply,
we indeed have that $ F(0,s)=0 $, that $ F(y,s) $ is $ \Wfil $-adapted (since $ \bfZ^{(n-1)}(s) $ is),
and that $ F(\Cdot,s) $ is uniformly Lipschitz, by \eqref{eq:psiLip}.
This yields the unique $ \Wfil $-adapted, $ C_+([0,\infty))^\hZ $-valued process $ \bfZ^{(n)} $.

To complete the construction, we show that $ \bfZ^{(n)} $ is also $ \YTspl(\gamma) $-valued.
To this end, we first establish the \emph{shift-invariance} of $ \bfZ^{(n)} $.
We say $ \bfZ: [0,\infty) \to [0,\infty)^\hZ $ is shift-invariant if 
$
	\bfZ(\Cdot) \stackrel{\text{distr}}{=} (Z_{a+i}(\Cdot))_{a\in\hZ}
	:=\theta_i\BK{\bfZ(\Cdot)},
$
$  \forall i\in\bbZ $.
\begin{lemma}\label{lem:shiftInv}
The processes $ \bfZ^{(0)},\ldots, \bfZ^{(n)} $, constructed in the preceding,
are shift-invariant.
\end{lemma}
\begin{proof}
We prove by induction on $j$ the stronger statement
\begin{align}\label{eq:tranInv:induc}
	\BK{ \bfZ^{(j)}(\Cdot), \bfW(\Cdot) }
	\stackrel{\text{distr}}{=}
	\BK{ \theta_i\bfZ^{(j)}(\Cdot), \theta_i\bfW(\Cdot) },
	\quad
	\forall i\in\bbZ.
\end{align}
This is clear for $j=0$. 
For $ j>0 $,
with $ \f_a(y,\bfz) $ as in \eqref{eq:psi}, we have $\f_{a+i}(y,\bfz)=\f_{a}(y,\theta_i\BK{\bfz(s)})$.
Using this in \eqref{eq:iter}, we obtain
\begin{align*}
	Z^{(j)}_{a+i}(t) = \gamma + W_{a+i}(t) + 
	\beta \int_0^t \BK{ \tfrac{1}{Z^{(j)}_{a+i}(s)} -  
	\f_{a} \BK{ Z^{(j)}_{a+i}(s), \theta_i\BK{\bfZ^{(j-1)}(s)} } } ds.
\end{align*}
Combining this with the induction hypothesis,
we then deduce that $ \bfZ := \theta_i(\bfZ^{(j)}) $ solves
\begin{align}\label{eq:tranInv:uniq}
	Z_a(t) = \gamma + W'_a(t) 
	+ \beta \int_0^t \Big( \tfrac{1}{Z_a(s)} -& \f_a(Z_a(s),\bfZ'(s)) \Big) ds, 
	\quad
	a\in\hZ,
\\
	&
	\notag
	\text{ for } (\bfZ',\bfW') \stackrel{\text{distr}}{=} (\bfZ^{(i-1)},\bfW).
\end{align}
Indeed, $ \bfZ^{(j)} $ also solves \eqref{eq:tranInv:uniq}.
By Lemma~\ref{lem:oneD},
for each $ a\in\hZ $, the solution of \eqref{eq:tranInv:uniq} is unique in the \emph{pathwise} sense,
so the \emph{system} of \ac{SDE} \eqref{eq:tranInv:uniq} must also enjoy pathwise uniqueness.
Since pathwise uniqueness implies uniqueness in law
in finite dimensions (e.g.\ \cite[Proposition 5.3.20]{karatzas91}),
by first considering $ a\in \calI:=(i_1,i_2)\in\hZ $
and letting $ (i_1,i_2)\to(-\infty,\infty) $,
we obtain the uniqueness in law of \eqref{eq:tranInv:uniq}.
This completes the induction.
\end{proof}

Equipped with Lemma~\ref{lem:shiftInv},
we proceed to showing $ \bfZ^{(n)} \in \YTspl(\gamma) $.
To this end,
letting $ \fS_{(i_1,i_2)}(\bfy):=\sum_{a\in(i_1,i_2)}\fS_a(\bfy) $
(where $ \fS_a(\bfy) $ is defined as in \eqref{eq:fS}),
we will use the following readily verified identity (c.f.\ \eqref{eq:phifS})
in the proof of Lemma~\ref{lem:ZYsp}:
\begin{align}\label{eq:resum1}
	\eta_{(i_1,i_2)}(\bfy) 
	=
	\barf_{(i_1,i_2)}(\bfy) - \undf_{(i_1,i_2)}(\bfy),
\end{align}
where $ i_- := (i_1\wedge i_2) < i_+:=(i_1\vee i_2) $ and
\begin{align}
	\label{eq:barf}
	\barf_{(i_1,i_2)}(\bfy) 
	&:=
	\sum_{i\in(i_1,i_2]} \frac{1}{2y_{(i_1,i)}}
	+
	\sum_{i\in(i_2,i_1]} \frac{1}{2y_{(i_2,i)}},	
\\
	\label{eq:undf}
	\undf_{(i_1,i_2)}(\bfy) 
	&:=
	\undffp_{(i_1,i_2)}(y_{(i_1,i_2)},\bfy) + \undffm_{(i_1,i_2)}(y_{(i_1,i_2)},\bfy),
\\	
	\label{eq:undffpm}
	&
	\undffpm_{(i_1,i_2)}(z,\bfy) 
	:= \sum_{i'\in(i_\pm,\pm\infty)} \frac{ z }{ 2(z+y_{(i,i')}) y_{(i,i')} }.	
\end{align}
Note that the expressions
$ \eta_{(i_1,i_2)}(\bfy) $,
$ \barf_{(i_1,i_2)}(\bfy) $
and
$ \undf_{(i_1,i_2)}(\bfy) $
are well-defined for all $ \bfy\in\Yspl(\gamma) $.

\begin{lemma}\label{lem:ZYsp}
Let $ \bfZ^{(0)},\ldots, \bfZ^{(n)} $, 
with $ \{\bfZ^{(i)}\}_{i=0}^{n-1}\subset \YTspl(\gamma) $ and $ \bfZ^{(n)}\in C_+([0,\infty))^\hZ $, be as in the proceeding,
we have $ \bfZ^{(n)} \in \YTspl(\gamma) $.
\end{lemma}
\begin{proof}
Let $ V^n_{(i_1,i_2)}(s) := \av_{(i_1,i_2)} (\bfZ^{(n)}(s)) $.
With $ \bfZ^{(n)}\in C_+([0,\infty))^\hZ $,
fixing $t\geq 0$, it suffices to show $ (\liminf_{|m|\to\infty} \underline{V^{n}_{(0,m)}}(0,t)) \geq \gamma $.
We achieve this in two steps by showing
\begin{align*}
	& i)&
	&\lim_{|m|\to\infty} V^{n}_{(0,m)}(s') \geq \gamma \ \text{ almost surely, for each fixed } s'\in [0,t];
	&
\\
	& ii)&
	& \liminf_{|m|\to\infty} \underline{V^{n}_{(0,m)}}(t) \geq \gamma \ \text{ almost surely.}
\end{align*}

\textit{i})
Fixing $ s'\in[0,t] $,
we begin by deriving a lower bound on $ V^{n}_{(0,m)}(s') $.
With $\bfZ^{(n-1)}(\Cdot)\geq\bfZ^{(n)}(\Cdot)$ (by Proposition~\ref{prop:iterDecr}),
by \eqref{eq:psi:mono:y} we have
\begin{align*}
	\tfrac{1}{Z^{(n)}_a(s)} - \f_a(Z^{(n)}_a(s),\bfZ^{(n-1)}_a(s))
	\geq
	\tfrac{1}{Z^{(n-1)}_a(s)} - \f_a(Z^{(n-1)}_a(s),\bfZ^{(n-1)}_a(s))
	=
	\fS_a(\bfZ^{(n-1)}(s)).
\end{align*}
Inserting this into \eqref{eq:iter}, summing the result over $ a\in(0,m) $,
and dividing both sides by $ |m| $, 
with $ \sum_{a\in(0,m)} W_a(s') = B_{m}(s') - B_0(s') $, $ Z^{(n)}_a(0)=\gamma $ and \eqref{eq:resum1},
we have
\begin{align}
	\label{eq:YZsp:eq}
	V^{n}_{(0,m)}(s')
	\geq
	\gamma + 
	|m|^{-1}(B_{m}(s') - B_0(s')) 
	- \frac{\beta}{|m|} \int_0^{s'} \undf_{(0,m)}(\bfZ^{(n-1)}(s)) ds.
\end{align}
As $ \lim_{|m|\to\infty} (|m|^{-1}(B_{m}(s') - B_0(s')))= 0 $ almost surely,
it clearly suffices to show
\begin{align}\label{eq:ZYsp:fpm:}
	\int_0^{s'} |m|^{-1} \undf_{(0,m)}(\bfZ^{(n-1)}(s)) ds \longrightarrow 0~
	\text{ almost surely, as } |m|\to\infty.
\end{align}
With $\{Z^{(n)}_a(s')\}_{a\in\hZ}$ being shift-invariant (by Lemma~\ref{lem:shiftInv})
and having a finite first moment (since $ \bfZ^{(n)}(s') \leq \bfZ^{(0)}(s') $),
by the Birkhoff--Khinchin ergodic theorem,
we have that $ V^{n}_{(0,m)}(s') $ converges almost surely 
(to a possibly random limit) as $ |m| \to \infty $.
Using this, we further reduce showing \eqref{eq:ZYsp:fpm:} to showing
\begin{align}\label{eq:ZYsp:fpm}
	\int_0^{s'} |m|^{-1} \undf_{(0,m)}(\bfZ^{(n-1)}(s)) ds \Longrightarrow 0,
	\text{ as } |m|\to\infty,
\end{align}
where $ \Rightarrow $ denotes convergence in law.

We proceed to showing \eqref{eq:ZYsp:fpm}.
This, with \eqref{eq:undf},
amounts to estimating $ \undfpm_\calI(\bfy) := \undffpm_{\calI}( y_{\calI}, \bfy ) $,
for $ \calI := (0,m) $ and $ \bfy=\bfZ^{(n-1)}(s) $.
With $ Z^{(n-1)}_a $ satisfying \eqref{eq:iter},
by \eqref{eq:BesComp:sup} we have that 
$ 
	Z^{(n-1)}_a(s') 
	\leq 
		Z^{(n-1)}_a(0) + Q^{0,s'}_a 
	=
		\gamma + Q^{0,s'}_a.
$
Combining this with \eqref{eq:QLLN},
we have
\begin{align*}
	N := \sup \curBK{  \overline{ V^{n-1}_{(0,m)}} (0,s'): m\in\bbZ } < \infty.
\end{align*}
With $\bfZ^{(n-1)}\in\YTspl(\gamma)$, we have
$
	D := 
	\inf \Big\{
		\underline{ V^{n-1}_{(i,0)} } \big(0,s'\big), i\neq 0
	\Big\}
	>0.
$
With
\begin{align*}
	\undfm_{(0,|m|)}(\bfy)
	=
	\sum_{i=1}^\infty \frac{y_{(0,|m|)}}{y_{(-i,|m|)}y_{(-i,0)}},
	\quad
	\undfp_{(0,-|m|)}(\bfy)
	=
	\sum_{i=1}^\infty \frac{y_{(-|m|,0)}}{y_{(0,i)}y_{(-|m|,i)}},	
\end{align*}
so by the preceding bounds we then have
\begin{align}\label{eq:ZYsp:psi-}
	\int_{0}^{s'} |m|^{-1} \undfmp_{(0,\pm|m|)} \BK{\bfZ^{(n-1)}(s)} ds
	\leq  
	s' \sum_{i=1}^\infty \frac{N}{2(iD+|m|N)iD}
	\xrightarrow{\text{a.s.}}
	0,
	\text{ as }|m|\to \infty.
\end{align}
Next, using the shift-invariance of $\bfZ^{(n-1)}$, we have
\begin{align*}
	\undfmp_{(0,\pm|m|)} \BK{ \bfZ^{(n-1)}(s) }
	\stackrel{\text{distr}}{=}
	\undfmp_{(0,\pm|m|)} \BK{ \theta_{\mp|m|}(\bfZ^{(n-1)}(s)) }
	=
	\undfmp_{(0,\mp|m|)} \BK{ \bfZ^{(n-1)}(s) }.
\end{align*}
Combining this with \eqref{eq:ZYsp:psi-} yields
\begin{align*} 
	\int_{0}^{s'} |m|^{-1} \undfmp_{(0,\mp|m|)} \BK{ \bfZ^{(n-1)}(s) } ds
	\Longrightarrow
	0,
	\quad \text{ as } |m|\to\infty.
\end{align*}
From this and \eqref{eq:ZYsp:psi-} we conclude \eqref{eq:ZYsp:fpm},
thereby completing the proof of (\textit{i}).

\textit{ii})
With (\textit{i}),
this is achieved by a continuity estimate based on \eqref{eq:BesComp}.
To this end, partition $[0,t]$ into $j_*$ 
equally spaced subintervals $0=t_0<\ldots<t_{j_*}=t$.
For each $a\in\hZ$, with $ Z^{(n)}_a $ satisfying \eqref{eq:iter},
we apply \eqref{eq:BesComp} for $ Y^*=Z^{(n)}_a $.
Averaging the result over $a\in(0,m)$, we obtain 
\begin{align}\label{eq:Zlw:tj}
	\underline{ V^{n}_{(0,m)} }(t_{j-1},t_j)
\geq
	V^{n}_{(0,m)}(t_j) 
	- \av_{(0,m)} \BK{ \bfQ^{t_{j}}(t_{j-1}) }.
\end{align}
Letting $|m|\to\infty$, by (\textit{i}) and \eqref{eq:QLLN},
we have
\begin{align*}
	\liminf_{|m|\to\infty} \underline{ V^{n}_{(0,m)} }(t_{j-1},t_j)
\geq
	\gamma - q(t/j_*,1).
\end{align*}
Combining this for $ j=1,\ldots,j_* $,
using the readily verified inequality
\begin{align*}
	\liminf_{|m|\to\infty} \underline{f_m}(0,t)
	\geq
	\min_{j=1}^{j_*}
	\Big\{
		\liminf_{|m|\to\infty} \underline{f_m}(t_{j-1},t_j)
	\Big\},
	\quad
	f_m(\Cdot) : [0,\infty) \to \bbR,
\end{align*}
we thus conclude
$
	(\liminf_{|m|\to\infty} \underline{ V^{n}_{(0,m)} }(0,t)) 
	\geq 
	\gamma - q(t/j^*,1),
$
almost surely. 
With $j^*$ being arbitrary, 
the proof is completed upon letting $ j^*\to\infty $, (whence $q(t/j^*,1)\to 0$).
\end{proof}

Having constructed the iteration sequence $ \{\bfZ^{(n)}\}_{n} $ for $ \bfz^\ic=\bfgamma $,
with $ \bfZ^{(n)}(\Cdot) \geq \bfZ^{(n+1)}(\Cdot) $ (by Proposition~\ref{prop:iterDecr}),
we let $ Z^{(\infty)}_a(t) := \lim_{n\to\infty} Z^{(n)}_a(t) \geq 0 $ denote the limiting process.
We next establish a lower bound on the average spacing of $ \bfZ^{(\infty)} $.
\begin{lemma}\label{lem:YZlw}
We have $ \bfZ^{(\infty)}\in\YTspll(\gamma) $ almost surely,
where
\begin{align}\label{eq:YTspll}
	\YTspll(\gamma) :=
	\Big\{  
		\bfy(\Cdot): [0,\infty) \to [0,\infty)^\hZ:
		\liminf_{|m|\to\infty} \inf_{s\in[0,t]} \av_{(0,m)} (\bfy(s)) \geq \gamma, \forall t \geq 0
	\Big\}.
\end{align}
\end{lemma}

\begin{proof}
Fixing $ t\geq 0 $,
we let $ V^\infty_\calI(s) := \av_{\calI} (\bfZ^{(\infty)}(s)) $,
and recall that $ V^n_{\calI}(s) := \av_{\calI} (\bfZ^{(n)}(s)) $.
As already mentioned in the proof of Lemma~\ref{lem:ZYsp},
since $ \bfZ^{(n)} $ is shift-invariant for $ n\in\bbZ_{>0} $ 
and (hence) for $ n=\infty $,
and since each $ Z^{(n)}_a $ as a finite mean for $ n\in\bbZ_{>0}\cup\{\infty\} $
(because $\bfZ^{(\infty)}(s) \leq \bfZ^{(0)}(s) $),
by the Birkhoff--Khinchin ergodic theorem, the limits 
\begin{align*}
	V^n(s) :=\lim_{|m|\to\infty} V^{n}_{(0,m)}(s),
	\quad
	V^\infty(s) :=\lim_{|m|\to\infty} V^{\infty}_{(0,m)}(s),
\end{align*}
exists almost surely.

As in the proof of Lemma~\ref{lem:ZYsp},
we proceed by first proving
$ V^\infty(s) \geq \gamma $ almost surely,
for any fixed $ s\in[0,t] $.
With
$
	Z^{(\infty)}_a(s) \leq Z^{(n)}_a(s) \leq Z^{(0)}_a(s)
	\leq
	\gamma + Q^{0,s}_a,
$
we have that
$
	\{ V^{n}_{(0,m)}(s) \}_{m\in\bbZ}
$
is uniformly integrable, for $ n\in\bbZ_{\geq 0}\cup\{\infty\} $.
Consequently,
we have
\begin{align*}
	\bbE(V^n(s)) 
	= 
	\lim_{|m|\to\infty}  
	\bbE\Big( V^n_{(0,m)}(s) \Big)	
	= 
	\lim_{|m|\to\infty}  
	\bbE\Big( \av_{(0,m)} (\bfZ^{(n)}(s)) \Big)
	=
	\bbE( Z^{(n)}_{1/2}(s) ),
	\
	\forall  n\in\bbZ_{\geq 0}\cup\{\infty\}.
\end{align*}
With $ Z^{(n)}_{1/2}(s) \searrow Z^{(\infty)}_{1/2}(s) $,
we thus conclude $ \bbE(V^n(s) ) \to \bbE(V^\infty(s)) $.
Combining this with $ V^n(s) \geq V^\infty(s) \geq 0 $ 
(as $ \bfZ^{(n)}(s) \geq \bfZ^{\infty}(s) \geq \mathbf{0} $),
we further obtain that $ V^n(s) \to V^\infty(s) $ almost surely.
By Lemma~\ref{lem:ZYsp}, $ V^\infty(s) \geq \gamma $ almost surely,
so $ V^\infty(s) \geq \gamma $ almost surely.

Now, letting $n\to\infty$ in \eqref{eq:Zlw:tj}, 
we obtain
\begin{align*}
	\underline{ V^\infty_{(0,m)} }(t_{j-1},t_j)
\geq
	\underline{ V^\infty_{(0,m)} }(t_j)
	- \av_{(0,m)} ( \bfQ^{t_{i}}(t_{j-1}) ).
\end{align*}
With this and $ V^\infty(t_j) \geq \gamma $,
the proof is completed by following the same continuity 
argument as in the proof of Lemma~\ref{lem:ZYsp}(\textit{ii}).
\end{proof}

Now, we turn to the initial condition $ \bfy^\ic\in [\gamma,\infty)^\hZ $
and construct the corresponding iteration sequence and limiting process.
\begin{lemma}\label{lem:Yiter}
Let $ \bfy^\ic\in [\gamma,\infty)^\hZ $ be as in the preceding.
There exists a $ \YTspl(\gamma) $-valued, $ \Wfil $-adapted,
decreasing sequence $ \{\bfY^{n}\}_{n\in\bbZ_{\geq 0}} $ satisfying \eqref{eq:iter:}.
Further, with $ Y^{(\infty)}_a(t) := \lim_{n\to\infty} Y^{(n)}_a(t) \geq 0 $
denoting the limiting process,
we have $ \bfY^{(\infty)} \in \YTspl'(\gamma) $.
\end{lemma}
\begin{proof}
To construct such a sequence $ \{\bfY^{n}\}_{n} $,
as seen from the proceeding construction of $ \{\bfZ^{n}\}_n $,
it suffices to show $ \bfY^{(n)}\in\YTspl(\gamma) $.
This follows directly by induction on $ n $ using Proposition~\ref{prop:iterDecr},
which assures $ \bfY^{(n)}(\Cdot) \geq \bfZ^{(n)}(\Cdot) $.
Letting $ n\to\infty $ in the previous inequality,
we obtain $ \bfY^{(\infty)}(\Cdot) \geq \bfZ^{(\infty)}(\Cdot) $,
thereby concluding $ \bfY^{(\infty)} \in \YTspl'(\gamma) $.
\end{proof}

With $ \bfY^{(\infty)} $ constructed as in the preceding,
we proceed to showing that $ \bfY^{(\infty)} $ is in fact a $ \YTspl(\gamma) $-valued solution.
This is done in a slightly more general context as follows.

\begin{proposition}\label{prop:bd>0}
Let $ \{\bfY^{[n]}\}_{n\in\bbZ_{>0}} $ 
be a nonnegative decreasing sequence such that either 
\begin{enumerate}[label=\itshape\roman*\upshape)]
\item \eqref{eq:iter} holds; 
or \item  $ \bfY^{[n]} $ solves \eqref{eq:DBMg} for all $ n $.
\end{enumerate}
Let $ Y^{[\infty]}_a(t) := \lim_{n\to\infty} Y^{[n]}_a(t) $ denote the limiting process.
If $ \bfY^{[\infty]}\in\YTspll(\gamma) $, $ \gamma>0 $, and $ \bfY^{[\infty]}(0) \in (0,\infty)^\hZ$,
then $\bfY^{[\infty]}$ is a $\YTspl(\gamma)$-valued solution of \eqref{eq:DBMg}.
\end{proposition}
\begin{proof}
Fixing $ t \geq 0 $ and $ a_*\in\hZ $, 
we begin by showing
\begin{align}\label{eq:bd>0}
	\undY^{[\infty]}_{a_*}(0,t) >0,
	\quad
	\text{ almost surely}.
\end{align}
This is achieved by first showing that
there exists $ J^\pm\in (a_*,\pm\infty)\cap\bbZ $ such that
\begin{align}\label{eq:Y>0:h}
	h_{(J^\pm,\pm\infty)} ( \bfundY^{[\infty]}(0,t)  ) \geq \tfrac{\gamma}{2},
\end{align}
(where $ h_\calI(\bfy) $ is as in \eqref{eq:h}),
and then, using \eqref{eq:Y>0:h}
to reduce the problem to finite dimensions,
whereby showing that $\undY^{[\infty]}_a(0,t)  >0 $, $ \forall a\in (J^-,J^+) $.
Without lost of generality,
we assume $ t $ is small enough such that $q(t,1)<\gamma/2$,
since the general case follows by partition $ [0,t] $ into small enough subintervals.
With $Y^{[n]}_a$ solving an equation of the type \eqref{eq:BesComp:Y},
applying \eqref{eq:BesComp} for $Y^*=Y^{[n]}_a$, we obtain
\begin{align*}
	\av_{(a_*,m)} ( \bfundY^{[n]}(0,t) ) 
	\geq 
	\av_{(a_*,m)} (\bfY^{[n]}(t))  - \av_{(a_*,m)} (\bfQ^{0}(t) ).
\end{align*}
Sending $n\to\infty$ and $ |m|\to\infty $ in order,
with $ \bfY^{[\infty]}(t) \in \YTspll(\gamma) $ and \eqref{eq:QLLN},
we obtain
\begin{align*}
	\liminf_{|m|\to\infty} \curBK{ \av_{(a_*,m)} ( \bfundY^{[\infty]}(0,t) ) }
	\geq 
	\gamma - q(t,1) > \tfrac{\gamma}{2}.
\end{align*}
From this we obtain some random $I^\pm \in (a_*,\pm\infty) \cap\bbZ$ such that
$
	\av_{[a_*,i)} (\bfundY^{[\infty]}(0,t) )
	> 
	\tfrac{\gamma}{2},
$
$\forall i\in (-\infty,I^-]\cup[I^+,\infty) $.
Combining this with Lemma~\ref{lem:goodset} for $(i^\pm_1,i^\pm_2,i^\pm_3)=(a_*\pm\tfrac12,I^\pm,\pm\infty)$
we obtain the desired $J^\pm\in ( a_*,\pm\infty) \cap\bbZ$ satisfying \eqref{eq:Y>0:h}.

Equipped with \eqref{eq:Y>0:h},
we proceed to truncating the equation of $ \bfY^{[n]} $, \eqref{eq:iter} or \eqref{eq:DBMg}, 
at the finite window $\fkJ:=(J^-,J^+)$.
To this end, we express \eqref{eq:iter} and \eqref{eq:DBMg} 
as a system of finite-dimensional equations with external forces
(i.e.\ \eqref{eq:DBMfc}), as
\begin{align}\label{eq:lem:Y>0:YnEq}
\begin{split}
	&
	Y^{[n]}_a(t) = Y^{[n]}_a(0) + W_a(t)
\\
	&\quad 
	+ 
	\beta
	\int_{0}^t 
	\big(
		\fS_a^\fkI (Y^{[n]}_a(s),\bfY^{[n]}(s))
		+ Y^{[n]}_a(s) Z^{**}_a(s) 
	\big) 
	ds,
	\quad \forall \ a\in\fkJ,
\end{split}
\end{align}
where the external force $ Z^{**}_a(s) := z^{**,\fkI}_a(\bfY^{[n]}(s),\bfY^{[n-1]}_a(s)) $ 
takes the form
\begin{align*}
	z^{**,\calA}_a(\bfy,\bfy') := 
	\left\{\begin{array}{l@{,}l}
			z^{1,\calA}_a(\bfy,\bfy') 
			:= \tfrac{1}{y} \BK{ \f_a^\calA(y,\bfy) -\f_a^\calA(y,\bfy') - \f_a^{\calA^c}(y,\bfy')} 
		 	&
		 	\text{ for } \eqref{eq:iter},
		 \\
			 z^{2,\calA}_a(\bfy)
			 :=
		 	- \tfrac{1}{y} \f_a^{\fkI^c}(y,\bfy)
		 	&
		 	\text{ for } \eqref{eq:DBMg}.
	\end{array}\right.
\end{align*}
With $\{\bfY^{[n]}\}$ being decreasing,
by \eqref{eq:psiA:mono:z} we have
\begin{align*}
	&
	\f_a^\calA(y,\bfY^{[n]}(s)) - \f_a^\calA(y,\bfY^{[n-1]}(s)) \geq 0, &
	&
	\f_a^{\calA^c}(y,\bfY^{[n-1]}(s)) \leq \f_a^{\calA^c}(y,\bfY^{[n]}(s)),
\end{align*} 
so
$ 
	z^{1,\calA}_a(\bfY^{[n]}_a(s),\bfY^{[n-1]}(s)) 
	\geq
	z^{2,\calA}_a(\bfY^{[n]}(s)).
$
Further,
with $ \f_a^{\fkI^c}(y,\bfy) $ as in \eqref{eq:psiA},
we have 
$ 
	z^{2,\fkI}_a(\bfy) 
	\geq 	
	-\tfrac12
	\sum_{\sigma=\pm} \sum_{i=1}^\infty 
	(y_{(J^\sigma,J^\sigma+\sigma i)})^{-2}.
$
Using \eqref{eq:Y>0:h}, we thus conclude
\begin{align*}
	&
	Z^{**}_a(s) 
	\geq	
	z^{2,\fkI}_a(\bfY^{[n]}(s))
	\geq
	- \sum_{i=1}^\infty (i\gamma/2)^{-2} =: c^* > -\infty.
\end{align*}
With this,
letting $ \bfY^{(i_1,i_2)} $ be the $ C_+([0,\infty))^{(i_1,i_2)\cap\hZ} $-valued
solution of \eqref{eq:DBMfc} for $ Z^{*}_a(s) = c^* $,
by Lemma~\ref{lem:DBMcomp} we have
$ \bfY^{[n]} \geq^\fkI_{[0,t]} \bfY^{\fkI} \in C_+([0,\infty))^\fkI $.
As $ a_*\in\fkI $, letting $ n\to\infty $, we conclude \eqref{eq:bd>0}.

By \eqref{eq:bd>0}, we now have
\begin{align}\label{eq:Yinf:YTspll}
	\bfY^{[\infty]} \in \YTspll\cap \curBK{y:[0,\infty)\to\bbR \ : \undy(0,t) >0, \forall t>0 }^{\hZ}.
\end{align}
With this, we now let $ n\to\infty $
in the equation for $ \bfY^{[n]} $, \eqref{eq:iter} or \eqref{eq:DBMg}.
By the dominated convergence theorem, we have
\begin{align*}
	&
	\lim_{n\to\infty} \int_0^ t \frac{1}{Y_a^{[n]}(s)} ds
	=
	\int_0^ t \frac{1}{Y_a^{[\infty]}(s)} ds < \infty,
\\
	&
	\lim_{n\to\infty} \int_0^ t \f_a(Y^{[n]}_a,\bfY^{[n-1]}(s)) ds
	=
	\int_0^t \f_a(Y^{[\infty]}_a,\bfY^{[\infty]}(s)) ds < \infty,
\\
	&
	\lim_{n\to\infty} \int_0^ t \f_a(Y^{[n]}_a,\bfY^{[n]}(s)) ds
	=
	\int_0^t \f_a(Y^{[\infty]}_a,\bfY^{[\infty]}(s)) ds < \infty,
\end{align*}
so $ \bfY^{[\infty]} $ solves \eqref{eq:DBMg}.
This automatically implies that $ \bfY^{[\infty]} \in C([0,\infty))^\hZ $.
With \eqref{eq:Yinf:YTspll},
we thus conclude that $ \bfY^{[\infty]} \in \YTspl(\gamma) $.
\end{proof}

\begin{proof}[Proof of Proposition~\ref{prop:existg}]
Combining Lemma~\ref{lem:Yiter} and Proposition~\ref{prop:bd>0},
we conclude that $ \bfY^{(\infty)} $ is a $ \YTspl(\gamma) $-valued solution of \eqref{eq:DBMg}.
To show that it is the greatest solution,
we consider a generic a $ \YTspl $-valued solution 
$ \bfY' $ with $ \bfY'(0) \leq \bfy^\ic $.
With $ \bfY' $ and $ \bfY^{(n)} $ solving the respective equations,
\eqref{eq:DBMg} and \eqref{eq:iter:},
following the comparison argument as in the proof of Proposition~\ref{prop:iterDecr},
we obtain that $ \bfY'(\Cdot) \leq \bfY^{(n)}(\Cdot) $, $ \forall n\in\bbZ_{\geq 0} $.
Upon letting $ n\to\infty $ we obtain $ \bfY'(\Cdot) \leq \bfY^{(\infty)}(\Cdot) $.
Next, assuming $ \bfy^\ic\in\RTsp(p) $,
with $ Y^{(\infty)}_a $ solving \eqref{eq:DBMg},
applying \eqref{eq:BesComp:sup} for $ Y^*=Y^{(\infty)}_a $,
we have
$	
	(Y^{(\infty)}_a(t))^p \leq (y^\ic_a + Q^{0,t}_a)^p
	\leq 2^p ((y^\ic_a)^p + (Q^{0,t}_a)^p).
$
From this and \eqref{eq:QLLN}, we conclude $ \bfY^{(\infty)} \in \RTsp(p) $.
\end{proof}

\section{Existence: Proof of Proposition~\ref{prop:existGener}}
\label{sect:exist}
Fixing $ \bfy^\ic\in\Ysp(\alpha,\rho) $
and a sequence $ 1\geq \gamma_1 \geq \gamma_2 \geq \ldots \to 0 $,
we let $ \bfYgn $ be as in Proposition~\ref{prop:existGener}.
Let $ m_i:= \floorBK{ i^{1/\alpha} } $, for $ i\geq 0 $, and $ m_{i}:=-m_{|i|} $ for $ i<0 $.
For any fixed $ k\in\bbZ_{>0} $, 
we construct a partition $ \{ \AL_{b,k} \}_{b\in\hZ} $ of $ \hZ $ by letting
$ \tilm^k_i := m_{ki} $,
\begin{align}\label{eq:A}
	\A_{b,k} := (\tilm^k_{b-1/2},\tilm^k_{b+1/2}),
	\quad
	\AL_{b,k} := \A_{b,k} \cap \hZ.
\end{align}
This partition is constructed so that $ |\AL_{k,b}| \sim k(\tilm^k_{|b|+1/2})^{1-\alpha} $.
More precisely, with
$	
	|\AL_{b,k}|
	=
	\tilm^k_{|b|+1/2} - \tilm^{k}_{|b|-1/2}
$
and 
$
	\floorBK{y-x} \leq \floorBK{y} - \floorBK{x} \leq \ceilBK{ y-x },
$
$ \forall x\leq y\in[0,\infty) $,
we have
\begin{align} 
	\label{eq:ALwbd}
	|\AL_{b,k}|
	&\geq
	\floorBK{ \tfrac{k^{\frac{1}{\alpha}}}{2\alpha} |b|^{\frac{1-\alpha}{\alpha}} }  
	\geq
	\floorBK{ \tfrac{k}{\alpha 2^{1/\alpha}} (m_{k(|b|+1/2)})^{1-\alpha} }, 
\\
	\label{eq:AUpbd}
	|\AL_{b,k}|
	&\leq
	\ceilBK{ \tfrac{k^{\frac{1}{\alpha}}}{\alpha}  (|b|+\tfrac12)^{\frac{1-\alpha}{\alpha}} }  
	\leq
	\ceilBK{ \tfrac{k}{\alpha 2^{1/\alpha}} (m_{k(|b|+1/2)}+1)^{1-\alpha} }.
\end{align}
With \eqref{eq:ALwbd}--\eqref{eq:AUpbd} and
$
	\av_{\A_{b,k}}(\bfy) 
	= 
	\tfrac{1}{|\AL_{b,k}|} 
	\big| \sum_{a\in(0,\tilm^k_{b+1/2})}(\bfy) - \sum_{a\in(0,\tilm^k_{b-1/2})}(\bfy) \big|,
$
we have
\begin{align}\label{eq:avbd}
	\big| \av_{(0,m)}(\bfy) - \rho \big| 
	\leq 
	\tfrac{1}{|\AL_{b,k}|} \norm{y} 
	\BK{ |\tilm^k_{b+1/2}|^{1-\alpha} + |\tilm^k_{b-1/2}|^{1-\alpha} }
	\leq
	\tfrac{c}{k} \norm{y}, 
\end{align}
where $ \norm{\bfy} $ is defined as in \eqref{eq:Ynorm}.
Hereafter, we assume $ k\in\bbZ_{>0} $ is large enough so that 
$ \{\A_{b,k}\}_b $ is nondegenerated: 
i.e.\ $ \A_{b,k}\neq\emptyset $, $ \forall b\in\hZ $.

Recall that $ \bfY(t):=\lim_{n\to\infty} \bfYgn(t) $.
The main step of the proof is to establish lower bound on $ \av_{\A_{b,k}}(\bfY(s)) $, uniform in $ s\in[0,t] $.
To this end, we will repeatedly use the following inequalities \eqref{eq:dens:lw}--\eqref{eq:dens:up}.
Recall $ \barf_\calI(\bfy) $ and $ \undf_\calI(\bfy) $ are defined as in \eqref{eq:barf}--\eqref{eq:undf}.
\begin{lemma}\label{lem:denseq}
Let $ \bfY^* $ be a $ \YTspl $-valued solution of \eqref{eq:DBMg},
$ \calK \subset \calI \subset \calK' \subset\hZ $ be nested intervals,
and $ s'<s'' \in [0,t] $.
We have that
\begin{align}
	\label{eq:dens:lw}
	&
	\av_{\calK} ( \bfY^*(s'') )
	\geq
	\av_{\calK} ( \bfY^*(s') )
	- \frac{\beta}{|\calK\cap\hZ|} 
	\int^{s''}_{s'} \undf_{ \calI }(\bfY^*(s))  ds - \hatB^{\calK'}_{\calK}(t)
	- 
	\bdy^{\calK'}_{\calK}(\bfQ^{s',s''}),
\\
	\label{eq:dens:up}
	&
	\av_{\calK} ( \bfY^*(s'') )
	\leq
	\av_{\calK} ( \bfY^*(s') )
	+ \frac{\beta}{|\calK\cap\hZ|} 
	\int^{s''}_{s'} \barf_{ \calI }(\bfY^*(s))  ds + \hatB^{\calK'}_{\calK}(t)
	+ 
	\bdy^{\calK'}_{\calK} ( \bfY^*(s') ),
\end{align}
where
\begin{align*}
	&
	\hatB^{\calK'}_{\calK}(t) 
	:= 
	\frac{4}{|\calK\cap\bbZ|}
	\sup_{ j \in \overline{\calK'}\setminus\calK } \overline{|B_j|}(0,t),&
	&
	\bdy_{\calK}^{\calK'}(\bfy) 
	:= 
	\frac{1}{|\calK\cap\bbZ|} \sum_{a\in\calK'\setminus\calK} y_a,
\end{align*}
and $ \overline{\calK'} $ denotes the closure of $ \calK' $.
\end{lemma}
\begin{proof}
Let $ \calI:=(i_1,i_2)\cap\hZ $.
With $ \bfY^* $ satisfying \eqref{eq:DBMg}, we have
\begin{align}\label{eq:dens:eqq}
	\sum_{a\in\calI} Y^*_a(s) \Big|^{s=s''}_{s=s'}
	=
	\beta
	\int^{s''}_{s'} \fS_{ \calI }(\bfY^*(s))  ds 
	+
	(B_{i_2}(s) - B_{i_1}(s)) \Big|^{s=s''}_{s=s'}.
\end{align}
Expressing the l.h.s.\ as
\begin{align*}
	\sum_{a\in\calK} Y^*_a(s'') - \sum_{a\in\calK} Y^*_a(s')
	+
	\sum_{a\in\calI\setminus\calK} Y^*_a(s) \Big|^{s=s''}_{s=s'}.	
\end{align*}
By \eqref{eq:BesComp:sup} we have
$ 
	\sum_{a\in\calI\setminus\calK} Y^*_a(s) |^{s=s''}_{s=s'}
	\leq
	\sum_{a\in\calK'\setminus\calK} Q^{s',s''}_a,	
$
and with $ Y^*_a(s'') > 0 $,
we clearly have
$ 
	\sum_{a\in\calI\setminus\calK} Y^*_a(s) |^{s=s''}_{s=s'}
	\geq
	-\sum_{a\in\calK'\setminus\calK} Y^*_a(s').	
$
Combining these with \eqref{eq:dens:eqq},
and dividing both sides by $ |\calK\cap\bbZ| $,
with \eqref{eq:resum1}, we conclude \eqref{eq:dens:lw}--\eqref{eq:dens:up}.
\end{proof}

Recalling the definition of $ q(t,1) $ from \eqref{eq:QLLN},
we begin by establishing the following preliminary estimate.
\begin{proposition}\label{prop:denss}
Fix $ t<\infty $ and let $ \tau<\infty $ be such that $ q(\tau,1) = \tfrac{\rho}{400} $.
For any $ t_*\in[0,t-\tau] $, if there exists $ K_*\in\bbZ_{>0} $ such that
\begin{align}\label{eq:denss:assum}
	\av_{\A_{b,k}}(\bfY(t_*)) > \tfrac{3\rho}{4},
	\quad
	\forall b\in\hZ,
	\ k \geq K_*,
\end{align}
then there exists some $ K\in\bbZ_{>0} $, satisfying the tail bound
\begin{align}\label{eq:Ktail}
	\bbP( K \geq k ) \leq \exp\BK{ -k^{1/\alpha}/c },
	\text{ where }
	c=c(\bfy^\ic,\alpha,\rho,t,\beta)<\infty,
\end{align}
such that
\begin{align}\label{eq:denss}
	\av_{\A_{b,k}} \BK{ \bfY(s) } \geq \tfrac{\rho}{2},
	\quad
	\forall s\in[t_*,t_*+\tau],\ b\in\hZ,
	\ k \geq K_*\vee K.
\end{align}
\end{proposition}
\begin{proof}
Throughout this proof we let 
$c<\infty $ denote a generic finite constant 
depending only on $ \bfy^\ic,\alpha,\rho,t,\beta $.

Fixing arbitrary $ n\in\bbZ_{>0} $,
we let $ S_{b,k} := \inf\{s \geq t_*: \av_{\A_{b,k}}(\bfYgn(s))<\tfrac{\rho}{2}\} $
and $ T_k := (t_*+\tau)\wedge (\inf_{b\in\hZ} S_{b,k}) $.
With $ \bfY(t):= \lim_{n\to\infty} \bfYgn(t) $,
proving \eqref{eq:denss} amount to proving $ T_{k} = t_*+\tau $,
for all $ k\geq K $, where $ K\in\bbZ_{>0} $ satisfies \eqref{eq:Ktail}.
However, as $ T_{k} $ involves infinitely many $ \A_{b,k} $, $ b\in\hZ $,
it is not even clear, a-priori, whether $ T_k>t_* $.
We circumvent this problem by truncating $ T_k $ as follows.
Consider the $ \YTspl(\gamma_n) $-valued solution $ \bfZ $ 
of \eqref{eq:DBMg} starting from $(\ldots,\gamma_n,\gamma_n,\ldots) $,
given by Proposition~\ref{prop:existg}.
With $ \bfZ $ being shift-invariant (by Lemma~\ref{lem:shiftInv}),
fixing arbitrary $ \ell\in\bbZ_{>0} $,
by the Birkhoff--Khinchin ergodic theorem,
we have
$
	\lim_{m\to\infty}
	\av_{(\pm \ell,\pm m)}\BK{ \bfundZ(0,t) }
	=
	Z > 0,
$
so $ \av_{(\pm \ell ,\pm m)}\BK{ \bfundZ(0,t) } > \tfrac{Z}{2}:=Z' $
for all $ |m| $ large enough, $ |m| \geq M_0 $.
With this, applying Lemma~\ref{lem:goodset} 
for $ (i^\pm_1,i^\pm_2,i^\pm_3) = (\pm \ell, \pm M_0, \pm\infty) $,
we obtain $ M_{\pm} \in[\pm\ell,\pm\infty)\cap\bbZ $ such that
\begin{align}\label{eq:dens:Z'}
	h_{(M_{\pm},\pm\infty)}(\bfundY^{\vee\gamma_n}(0,t) ) 
	\geq
	h_{(M_{\pm},\pm\infty)}(\bfundZ(0,t))
	\geq Z' > 0.
\end{align}
%
Having constructed $ M_{\pm} $,
we set 
\begin{align*}
	\{b\in\hZ: \A_{b,k}\subset (M_{-},M_{+}) \} =: (J_{-},J_{+}) \cap\hZ,
	\quad
	J_{-} \leq J_{+} \in\bbZ,
\end{align*}
and define the truncation of $ T_k $ as
$
	\tilT_{k} := (t_*+\tau) \wedge ( \inf_{ b\in (J_{-},J_{+}) }  S_{b,k} ).
$
Instead of proving $ T_k = t_*+\tau $,
we prove $ \tilT_{k} = t_*+\tau $ for all large enough $ \ell $
(which yields $ T_k = t_*+\tau $ upon letting $ \ell\to\infty $),
or equivalently
\begin{align}\label{eq:denss:goal}
	\av_{\A_{b,k}} (\bfYgn(\tilT_{k})) \geq \tfrac{\rho}{2},
	\quad
	\forall b \in (J_{-},J_{+}),
	\
	k \geq K,
	\
	\text{ large enough } \ell,
\end{align}
where $ K $ satisfies the tail bound \eqref{eq:Ktail}.

To the end of proving \eqref{eq:denss:goal}, 
fixing arbitrary $ b \in (J_-,J_+) $,
we let $ \tilA := (\A_{b-1,k}\cup\A_{b,k}\cup\A_{b+1,k}) $ denote the union of three consecutive intervals,
and apply \eqref{eq:dens:lw}, 
for arbitrary fixed $ \calI:=(i_1,i_2)\cap\hZ $ satisfying $ \A_{b,k} \subset (i_1,i_2) \subset \tilA $, 
to obtain
\begin{align}
	\label{eq:dens:avbd}
	\av_{\A_{b,k}} ( \bfYgn(\tilT_{k}) )
	>
	\frac{3\rho}{4}
	- \beta 
		\int^{\tilT_{k}}_{t_*} \frac{\undf_{\calI}(\bfYgn(s))}{|\AL_{b,k}|}  ds
	- \hatB^{\tilA}_{\A_{b,k}}(t)
	- \bdy^{\tilA}_{\A_{b,k}}(\bfQ^{t_*,t_*+\tau}).
\end{align}
With $ \hatB^{\tilA}_{\A_{b,k}}(t) $ defined as in Lemma~\ref{lem:denseq},
letting 
\begin{align}\label{eq:tilB}
	\tilB^k_b(t) 
	:= 
	\sup \curBK{ 
		\overline{ |B_j| }(0,t):
		j\in[\tilm^k_{b-1/2},\tilm^k_{b+1/2}]
		},
\end{align}
we clearly have 
\begin{align}\label{eq:hatBbd}
	\hatB^{\tilA}_{\A_{b,k}}(t) \leq \tfrac{2}{|\AL_{b,k}|} (\tilB^k_{b-1}(t)+\tilB^k_{b+1}(t)).
\end{align}
Further, by \eqref{eq:ALwbd}--\eqref{eq:AUpbd}, we have
\begin{align}\label{eq:Anbr}
	|\AL_{b\pm 1,k}|/ |\AL_{b,k}| \leq 16,
	\quad
	\forall b\in\hZ, \ k\in\bbZ_{>0},
\end{align}
so
\begin{align}\label{eq:tilQbd}
	\bdy^{\tilA}_{\A_{b,k}}(\bfQ^{t_*,t_*+\tau})
	\leq
	16 ( \av_{\A_{b+1,k}}(\bfQ^{t_*,t_*+\tau}) + \av_{\A_{b-1,k}}(\bfQ^{t_*,t_*+\tau}) ).
\end{align}
As $ \{B_i(\Cdot)\}_i $ is i.i.d.,
following standard arguments we show that, there exists $ K_0 \in \bbZ_{>0} $, 
satisfying the tail bound \eqref{eq:Ktail}, such that
\begin{align}
	\label{eq:denss:claim1}
	&\tilB^k_{b}(t) \leq |\AL_{b,k}|^{\frac12},
	\quad\quad
	\av_{\A_{b,k}} \BK{ \bfQ^{t_*,t_*+\tau} } \leq 2 q(\tau,1) = \tfrac{\rho}{200},
\\
	\label{eq:denss:claim2}
	&
	\av_{\A_{b,k}}(\bfbarY^{\vee\gamma_n}(0,t)) \leq 2+\rho+2q(t,1) = c(\rho,t),
	\quad\quad
	\forall b\in\hZ, \ k\geq K_0.	
\end{align}

Deferring the proof of \eqref{eq:denss:claim1}--\eqref{eq:denss:claim2} until after this proof,
we proceed to bounding the interaction term $ \tfrac{1}{|\AL_{k,b}|} \undf_{\calI}(\bfYgn(s)) $ for the suitable $ \calI $.
The endpoints of such $ \calI $ will be chosen from certain `seeds' $ I^{\pm}_{b} $,
which we now construct.
Fix $ k\geq K_*\vee K_0 $.
By the continuity estimate \eqref{eq:BesComp}
and the bound \eqref{eq:denss:claim1},
we have
$	\av_{\A_{b,k}}(\bfundY^{\vee\gamma_n}(t_*,\tilT_{k}))
	\geq
	\av_{\A_{b,k}}(\bfYgn(\tilT_{k})) - \tfrac{\rho}{200}.
$
Further, as $ s \mapsto \av_{\A_{b,k}}(\bfYgn(s)) $ is continuous,
by \eqref{eq:denss:assum} we have $ \av_{\A_{b,k}}(\bfYgn(\tilT_{k})) \geq \tfrac{\rho}{2} $, so
$	
	\av_{\A_{b,k}}(\bfundY^{\vee\gamma_n}(t_*,\tilT_{k}))
	\geq
	\tfrac{\rho}{2} - \tfrac{\rho}{200} > \tfrac{\rho}{3}.
$
With this, applying Lemma~\ref{lem:goodset}
for $ (i^\pm_1,i^\pm_2,i^\pm_3)=(\tilm^k_{b\mp 1/2},\tilm^k_{b\pm 1/2},\tilm^k_{J_\pm}) $,
we obtain $ I^{\pm}_{b}\in [\tilm^k_{b\mp 1/2},\tilm^k_{b\pm 1/2})\cap\bbZ $
such that
\begin{align}\label{eq:denss:h}
	 h_{ (I^{\pm}_{b},\tilm^k_{J_\pm}) }(\bfundY^{\vee\gamma_n}(t_*,\tilT_{k})) \geq \tfrac{\rho}{3}.
\end{align}
Having constructed $ I^{\pm}_b $,
we set $ \calI = \fkI := (I_-,I_+) $, where
\begin{align}\label{eq:denss:J}
	I_+ := \left\{\begin{array}{c@{,}l}
		I^{+}_{b+1}		&\text{ if } b+2 <J_+,	\\
		M_{+}			&\text{ otherwise},
	\end{array}\right.
\quad
	I_- := \left\{\begin{array}{c@{,}l}
		I^{-}_{b-1}		&\text{ if } b-2>J_-,	\\
		M_{-}			&\text{ otherwise},
	\end{array}\right.
\end{align}
and proceed to bounding $ \tfrac{1}{|\AL_{k,b}|} \undf_{\fkI}(\bfYgn(s)) $.

Letting $ \bfundY^* := \bfundY^{\vee\gamma_n}(t_*,\tilT_{k}) $ and $ \bfbarY^* := \bfbarY^{\vee\gamma_n}(t_*,\tilT_{k}) $,
with $ \undf_{\calI}(\bfy) $ and $ \undffpm_{\calI}(z,\bfy) $ defined as in \eqref{eq:undf}--\eqref{eq:undffpm},
we have
$
	\undf_{\fkI}(\bfYgn(s)) 
	\leq 
	\sum_{\sigma=\pm} \undffs_{ \fkI } ( \barY^{*}_\fkI, \bfundY^* ).
$
Further, by \eqref{eq:denss:claim2} and \eqref{eq:Anbr} we have
$
	\barY^*_\fkI \leq \barY^*_{\tilA} \leq c |\tilA\cap\hZ| \leq c |\AL_{b,k}|.
$
Using this to bound $ \undffs_{ \fkI } ( \barY^{*}_\fkI, \bfundY^* ) $, we obtain
\begin{align}
	\label{eq:denss:fS:}
	\frac{ \undf_{\fkI}(\bfYgn(s)) }{ |\AL_{b,k}| } 
	\leq
	\sum_{\sigma=\pm} \frac{  \undffs_{ \fkI } \BK{ c |\AL_{b,k}| , \bfundY^* } }{ |\AL_{b,k}| }	
	= 
	\sum_{\sigma=\pm} \sum_{i\in(I_\sigma,\sigma \infty)} \frac{ c }{ (|\AL_{b,k}|+\undY^*_{(I_\sigma,i)}) \undY^*_{(I_\sigma,i)} }.
\end{align}
With \eqref{eq:dens:Z'} and \eqref{eq:denss:h}, we have
\begin{align}\label{eq:denss:H}
	\undY^*_{(I_\pm,i)} 
	&\geq
	\tfrac{\rho}{3} \Big( (i-I_\pm)_{{\pmb{\pm}}} \wedge (\tilm^k_{J_\pm}-I_\pm)_{{\pmb{\pm}}} \Big) 
	+ (i-M_\pm)_{{\pmb{\pm}}} Z' ,
\end{align}
where $ (\ldots)_{{\pmb{\pm}}} $ denote the partitive/negative part.
By \eqref{eq:denss:J},
$ \AL_{J_\pm \mp 1/2,k} \subset (I_\pm,\tilm^k_{J_\pm}) $ if $ I_\pm\neq M_\pm $,
and	by \eqref{eq:Anbr} we have $ |\AL_{J_\pm \pm 1/2,k}| \leq c |\AL_{J_\pm\mp1/2,k}| $,
so
\begin{align*}
	\Big( (i-I_\pm)_{{\pmb{\pm}}} \wedge (\tilm^k_{J_\pm}-I_\pm)_{{\pmb{\pm}}} \Big)
	&\geq
	\tfrac{1}{c} \Big( (i-I_\pm)_{{\pmb{\pm}}} \wedge (\tilm^k_{J_\pm\pm 1}-I_\pm)_{{\pmb{\pm}}} \Big)
\\
	&\geq
	\tfrac{1}{c}  \Big( (i-I_\pm)_{{\pmb{\pm}}} \wedge (M_\pm-I_\pm)_{{\pmb{\pm}}} \Big).
\end{align*}
Combining this with \eqref{eq:denss:H}, and inserting the result into \eqref{eq:denss:fS:},
using the readily verified inequality
\begin{align*}
	\sum_{i=1}^\infty \frac{1}{(x+(i\wedge\mu)+(i-\mu)_{\pmb{+}}z)((i\wedge\mu)+(i-\mu)_{\pmb{+}}z)}
	\leq
	\frac{c}{x} \log (x+1) + \frac{c}{xz} \log \Big( \frac{1+\frac{\mu+x}{z}}{1+\frac{\mu}{z}} \Big),
\end{align*}
for $ x=c|\AL_{b,k}| >0 $, $ \mu=|M_\pm-I_\pm| \geq 0 $ and $ z=cZ'>0 $,
we arrive at
\begin{align*}
	\tfrac{ 1 }{ |\AL_{b,k}| } \undf_{\fkJ}(\bfYgn(s))
	\leq
	\tfrac{c}{|\AL_{b,k}|} \log(|\AL_{b,k}|) + cR^+_{b} + cR^-_{b},
\end{align*}
where $ R^\pm_{b} := (|\AL_{b,k}| Z')^{-1} \log ( U^{\pm}_{b} ) $ and
$
	U^{\pm}_{b} :=
	\frac{1+(|\AL_{b,k}|+\frac{1}{c}|M_{\pm}-I_{\pm}|)/Z'}{1+\frac{1}{c}|M_{\pm}-I_{\pm}|/Z'}.
$
With $ U^{\pm}_{b} \leq  1+ c\frac{|\AL_{b,k}|}{|M_{\pm}-I_{\pm}|} $
and $ U^{\pm}_{b} \leq  1+ |\AL_{b,k}|/Z' $,
we have 
\begin{align}\label{eq:dness:R}
	R^\pm_{b} \leq 
	c \
	\min
	\Big\{
		\frac{1}{|M_{\pm}-I_\pm| Z' },
		\frac{ \log(1+|\AL_{b,k}|/Z') }{ |\AL_{b,k}|Z' } 
	\Big\}.
\end{align}
Further, by \eqref{eq:ALwbd} and \eqref{eq:denss:J},
we have that $ |\AL_{b,k}| \geq \tfrac{1}{c} |\AL_{b\pm 2,k}| \geq \tfrac{1}{c} k (|I_\pm|+1)^{1-\alpha} $.
Using this to replace $ |\AL_{b,k}| $ with $ (|I_\pm|+1)^{1-\alpha} $ in \eqref{eq:dness:R},
and letting $ \ell\to\infty $ (whereby $ |M_{\pm}|\to\infty $),
we find that
$
	\sup\{ R^\pm_{b}: b\in(J_-,J_+), k\in\bbZ_{>0} \} \to 0,
$
as $ \ell \to \infty $.
Consequently, 
\begin{align}\label{eq:denss:fS}
	\beta \int^{\tilT_{k}}_{t_*} \tfrac{1}{|\AL_{b,k}|} \undf_{\fkI}(\bfYgn(s))  ds	
	\leq 
	\tfrac{c}{|\AL_{b,k}|} \log(|\AL_{b,k}|)
	+\frac{\rho}{100},
\end{align}
for all large enough $ \ell $ (i.e.\ $ \ell\geq L$, $L=L(Z')<\infty $).

Now, inserting the preceding bounds, 
\eqref{eq:hatBbd}, \eqref{eq:tilQbd}--\eqref{eq:denss:claim2} and \eqref{eq:denss:fS},
into the estimate \eqref{eq:dens:avbd},
we arrive at
\begin{align*}
	\av_{\A_{b,k}} ( \bfYgn(\tilT_{k}) )
	>
	\frac{3\rho}{4}
	- c \frac{ \log(|\AL_{b,k}|) }{|\AL_{b,k}|} 
	- \frac{\rho}{100}
	- \frac{ 2(|\AL_{b-1,k}|^{1/2} +|\AL_{b+1,k}|^{1/2}) }{|\AL_{b,k}|}
	- \frac{2\rho}{ 25 },
\end{align*}
for all $ k \geq K_*\vee K_0 $ and large enough $ \ell $.
By \eqref{eq:Anbr}, we have 
$  
	\frac{ |\AL_{b-1,k}|^{1/2} +|\AL_{b+1,k}|^{1/2} }{|\AL_{b,k}|}
	\leq 
	8|\AL_{b,k}|^{-1/2}
$.
Hence,
with $ |\AL_{b,k}| \geq  |\AL_{1/2,k}| \geq \tfrac{1}{c} k^{1/\alpha} $,
the r.h.s.\ is indeed
$
	\geq \tfrac{\rho}{2}
$
for all large enough $ k $, i.e.\ $ k \geq k_0$, where $k_0=k_0(\rho)<\infty $.
From this we conclude \eqref{eq:denss:goal} for $ K:= K_0\vee k_0 $.
\end{proof}
\begin{proof}[Proof of \eqref{eq:denss:claim1}--\eqref{eq:denss:claim2}]
We let $ c=c(\bfy^\ic,\alpha,\rho,t,\beta)<\infty $ denote a generic finite constant and
\begin{align*}
	&
	K_B := \inf \curBK{ 
		k\in\bbZ_{>0}: 
		\tilB^{k'}_b \leq |\AL_{b,k'}|^{\frac12},
		\
		\forall b\in\hZ, 
		\
		k'\geq k
		},
\\
	&
	K_Q := \inf \curBK{ 
		k\in\bbZ_{>0}: 
		\av_{\A^{k'}_b}(\bfQ^{t_*,t_*+\tau}) \leq 2q(\tau,1),
		\
		\forall b\in\hZ, 
		\
		k'\geq k
		}.
\end{align*}
Indeed, $ \bbP( \overline{|B_j|}(0,t) \geq |\AL_{b,k'}|^{1/2} ) \leq \exp(-\tfrac{1}{c} |\AL_{b,k'}| ) $.
Summing this inequality over 
\begin{align*}
	\{ (j,b,k'):
	j\in [ \tilm^{k'}_{b-1/2}, \tilm^{k'}_{b+1/2} ],
	b\in\hZ,
	k'\geq k
	\}
\end{align*}
using \eqref{eq:ALwbd}, 
we obtain $ \bbP( K_B \geq k) \leq c\exp(-\tfrac{1}{c} k^{1/\alpha}) $.
As for $ K_Q $, 
with \eqref{eq:Qiid} and $ \bbE(Q^{t_*,t_*+\tau}_{1/2}) = q(\tau,1) $,
we have the large deviation upper bound
$
	\bbP( \av_{\A_{b,k'}}(\bfQ^{t_*,t_*+\tau}) \geq 2 q(\tau,1) ) \leq \exp(\tfrac{1}{c}|\AL_{b,k'}|).
$
Summing this inequality over all $ b\in\hZ $ and $ k'\geq k $
as in the preceding,
we conclude $ \bbP(K_Q \geq k ) \leq c \exp(\tfrac{1}{c} k^{1/\alpha}) $.

Turning to establishing \eqref{eq:denss:claim2},
following the preceding argument we have 
$ \bbP(K'_Q \geq k ) \leq c \exp(\tfrac{1}{c} k^{1/\alpha}) $,
where
$
	K'_Q := \inf \{ 
		k\in\bbZ_{>0}: 
		\av_{\A_{b,k}}(\bfQ^{0,t}) \leq 2q(t,1),
		\
		\forall b\in\hZ, 
		\
		k'\geq k
		\}.
$
Combining this with \eqref{eq:BesComp:sup},
we obtain 
$ 
	\av_{\A_{b,k}}(\bfbarY^{\vee\gamma_n}(0,t)) \leq \av_{\A_{b,k}}(\bfY^{\vee\gamma_n}(0)) + 2q(t,1), 
$
$\forall b\in\hZ$, $k\geq K'_Q $.
Further, with $ \bfYgn(0)\leq \bfgamma_n +\bfy^\ic $, by \eqref{eq:avbd} we have
$
	\av_{\A_{b,k}}(\bfY^{\vee\gamma_n}(0))
	\leq
	\gamma_n + \rho + ck^{-1}
	\leq
	1 + \rho + ck^{-1}.
$
From this, we obtain $ k_\ic\in\bbZ_{>0} $ such that
$ 
	\av_{\A_{b,k}}(\bfbarY^{\vee\gamma_n}(0,t)) \leq 2 + \rho + 2q(t,1), 
$
$
	\forall k \geq K'_Q\vee k_\ic.
$
This completes the proof of \eqref{eq:denss:claim1}--\eqref{eq:denss:claim2}
for $ K_0 := K_B \vee K_Q \vee K'_Q \vee k_\ic $.
\end{proof}

Equipped with Proposition~\ref{prop:denss},
we proceed to proving the following uniform density estimate.
\begin{proposition}\label{prop:dens}
For any $ t\geq 0 $, there exists some $ K\in\bbZ_{>0} $,
satisfying the tail bound \eqref{eq:Ktail},
such that 
\begin{align}\label{eq:dens}
	\av_{\A_{b,k}} \BK{ \bfY(s) } \geq \tfrac{\rho}{2},
	\quad
	\forall s\in[0,t], \ b\in\hZ, n\in\bbZ_{>0}, \ k \geq K.
\end{align}
\end{proposition}
\begin{proof}
By \eqref{eq:avbd} we have
$
	\av_{\A_{b,k}} ( \bfY(0) )
	\geq
	\rho - c k^{-1}.
$
Hence for all large enough $ k $: $ k\geq k_0=k_0(\rho,\bfy^\ic) $,
we have $ \av_{\A_{b,k}} ( \bfY(0) ) > \tfrac{3\rho}{4} $,
$ \forall b\in\hZ $.
With this and $ \tau $ as in Proposition~\ref{prop:denss}, 
applying Proposition~\ref{prop:denss} for $ t_*=0 $ and $ K_*=k_0 $,
we conclude \eqref{eq:dens} if $ t\leq \tau $.
To progress to $ t>\tau $,
we show that, actually,
$ \av_{\A_{b,k\ell}} ( \bfY(\tau) ) > \tfrac{3\rho}{4} $,
for $ k $ further chosen large enough.
This is achieved by improving the estimation following \eqref{eq:dens:avbd}.
In this estimation,
the contribution the interaction term and $ \tilB^k_b $ 
are made arbitrarily small by choosing large enough $ k $,
but the term $ \tilQ^k_b $ stays bounded away from zero.
This problem is resolved by changing $ k\mapsto k \ell $,
which 
corresponds to grouping $ \ell $ consecutive intervals
of $ \{\A_{b,k}\}_b $ to form a new, coarser, partition $ \{\A_{b, k\ell}\}_b $.
Fixing arbitrary $ \A_{b,k\ell} $,
we let $ \A_\pm := \A_{\ell(b\pm 1/2)\pm 1/2,k}  $ denote the neighboring `small' intervals,
and form the spliced interval $ \tilA' := \A_- \cup \A_{b,k\ell} \cup \A_+ $.
Let $ \calI' $ be such that 
$
	\A_{b,k\ell} \subset \calI' \subset
	\tilA'.
$
With such interval $ \calI' $ replacing $ \calI $ as in \eqref{eq:dens:avbd},
we obtain
\begin{align}\label{eq:dens:Q}
	\bdy^{\tilA'}_{\A_{b,k\ell}}(\bfQ^{0,\tau})
	\leq	
	\tfrac{|\AL_{+}|}{|\AL_{b,k\ell}|} \av_{\A_{+}} \BK{ \bfQ^{0,\tau} }
	+
	\tfrac{|\AL_-|}{|\AL_{b,k\ell}|} \av_{\A_{-}} \BK{ \bfQ^{0,\tau} },
\end{align}
where $ \AL_\pm:=\A_\pm\cap\hZ $.
By \eqref{eq:ALwbd}--\eqref{eq:AUpbd},
we have that $ |\AL_{\pm}|/|\AL_{b,k\ell}| \to 0 $ as $ \ell\to\infty $,
uniformly in $ b\in\hZ $,
so, the term \eqref{eq:dens:Q} is made arbitrarily small by choosing $ \ell $ large enough.
With this improvement of the estimation of \eqref{eq:dens:avbd},
we obtain that $ \av_{\A_{b,k\ell_1}}(\bfY(\tau)) > \tfrac{3\rho}{4} $,
for all $ n\in\bbZ_{>0} $, $ k \geq k_0\vee K $ and some $ \ell_1=\ell_1(\rho) $,
which then allows us to
apply Proposition~\ref{prop:denss} for $ K_*= \ell_1(k_0\vee K) $ and $ t_*=\tau $.
Iterating the preceding procedure $ i_*:= \ceilBK{ t/\tau } $ times, we conclude \eqref{eq:dens}.
\end{proof}

Recall the definition of $ \YTspll(\gamma) $ from \eqref{eq:YTspll}.
Indeed, the uniform lower bound \eqref{eq:dens} implies that $ \bfY\in \YTspl'(\rho/2) $.
Combining this with Proposition~\ref{prop:bd>0} for $ \bfY^{[n]}=\bfYgn $,
we obtain that $ \bfY $ is a $ \YTspl(\rho/2) $-valued solution of \eqref{eq:DBMg}.
We next show that the bound \eqref{eq:dens} actually implies $ \bfY\in\YTsp(\alpha,\rho) $.
\begin{lemma}\label{lem:den2YTsp}
Let $ \bfY^* $ be a $ \YTspl(\gamma) $-valued solution of \eqref{eq:DBMg}, $ \gamma>0 $,
starting from $ \bfy^\ic\in\YTsp(\alpha,\rho) $.
If, for each $ t\geq 0 $, there exists $ K'\in\bbZ_{>0} $
such that \eqref{eq:dens} holds for $ \bfY^* $,
then in fact $ \bfY^*\in\YTsp(\alpha,\rho) $.
\end{lemma}
\begin{proof}
Fixing arbitrary $ t\geq 0 $, 
as $ \bfy^\ic \in\YTsp(\alpha,\rho) $,
proving $ \bfY^*\in\YTsp(\alpha,\rho) $ amounts to proving
\begin{align}\label{eq:Ynorm:bd}
	\sup_{s\in[0,t]} \Big\{ \sup_{m\in\bbZ} \big| \av_{(0,m)}\BK{\bfY^*(s)-\bfy^\ic} \big| \ |m|^{\alpha} \Big\}
	<\infty.
\end{align}
To this end, we assume without lost of generality $ |m| \geq \tilm^K_1 $.
Let $ K_0 $ be as in \eqref{eq:denss:claim1}--\eqref{eq:denss:claim2} and let $ K:=K'\vee K_0 $.
Set $ \calK:=(0,m) $, let $ \cup_{b=b^-}^{b^+} \A_{b,K} $, $ b^-<b^+ $,
be the smallest such interval that contains $ \calK $
and let $ \calK':= (\tilm^K_{b^--1/2},\tilm^K_{b^++1/2}) $.
Partition $ [0,t] $ into $ j $ equally spaced subintervals $ [t_{j-1},t_{j}] $,
$ j=1,\ldots,j_* $, each with length $ t/j_*\leq \tau $,
where $ \tau $ is as in Proposition~\ref{prop:denss}.
Applying \eqref{eq:dens:lw}--\eqref{eq:dens:up}
for $ [s',s'']=[t_{j-1},t_j] $ and $ \calI_j, \calI'_j\subset\bbR $ 
such that $ \calK \subset \calI_j,\calI'_j \subset \calK' $, we obtain
\begin{align}
	&
	\label{eq:denlw:(0,m)}
	\inf_{s\in[t_{j-1},t_j]}
	\av_{(0,m)} ( \bfY^*(s) )
	\geq
	\av_{(0,m)} ( \bfY^*(t_{j-1}) )
	- \frac{\beta}{|m|} 
	\int_{t_{j-1}}^{t_j} \undf_{ \calI_j }(\bfY^*(s))  ds - R_m,
\\
	&
	\label{eq:denup:(0,m)}
	\sup_{s\in[t_{j-1},t_j]}
	\av_{(0,m)} ( \bfY^*(s) )
	\leq
	\av_{(0,m)} ( \bfY^*(t_{j-1}) )
	+ \frac{\beta}{|m|} 
	\int_{t_{j-1}}^{t_j} \barf_{ \calI'_j }(\bfY^*(s))  ds + R_m,
\end{align}
where
$
	R_m := 
	\frac{1}{|m|}
	(
		\hatB^{\calK'}_{\calK}(t)
		+ \bdy^{\calK'}_{\calK}(\bfQ^{t_{j-1},t_j})
		+ \bdy^{\calK'}_{\calK}(\bfY^*(t_{j-1}))
	).
$
With $ \hatB^{\calK'}_{\calK}(t), \bdy^{\calK'}_{\calK}(\bfy) $
defined as in Lemma~\ref{lem:denseq}, 
$ \tilB^K_b(t) $ as in \eqref{eq:tilB},
and $ \calK,\calK' $ as in the preceding,
we clearly have that 
\begin{align*}
	R_m 
	\leq 
	\frac{1}{|m|} \sum_{b:|b-b_\pm|\leq 1 } 
	\BK{	
		\tilB^{K}_{b}(t)
		+ |\AL_{b,k}|\av_{\A_{b,k}} (\bfQ^{t_{j-1},t_{j}})
		+ |\AL_{b,k}|\av_{\A_{b,k}} (\bfY^*(t_{j-1}))
		}.
\end{align*}
Further applying \eqref{eq:denss:claim1}--\eqref{eq:denss:claim2},
we have that 
$ 
	R_m 
	\leq 
	\frac{c}{|m|} (|\tilm^K_{b_++3/2}|^{1-\alpha} + |\tilm^K_{b_--3/2}|^{1-\alpha}) 
	\leq 
	c |m|^{-\alpha} 
$.
Plugging this in \eqref{eq:denlw:(0,m)}--\eqref{eq:denup:(0,m)},
and combining the result for $ j=1,\ldots,j_* $,
we arrive at
\begin{align}
	&
	\label{eq:denlw:(0,m):}
	\inf_{s\in[0,t]}
	\av_{(0,m)} ( \bfY^*(s) )
	\geq
	\av_{(0,m)} ( \bfy^\ic )
	-  
	\sum_{j=1}^{j_*}
	\beta  
	\int_{t_{j-1}}^{t_j} \frac{1}{|m|}\undf_{ \calI_j }(\bfY^*(s))  ds 
	- c(t) |m|^{-\alpha},
\\
	&
	\label{eq:denup:(0,m):}
	\sup_{s\in[0,t]}
	\av_{(0,m)} ( \bfY^*(s))
	\leq
	\av_{(0,m)} ( \bfy^\ic )
	+ 
	\sum_{j=1}^{j_*} 
	\beta
	\int_{t_{j-1}}^{t_j}  \frac{1}{|m|} \barf_{ \calI'_j }(\bfY^*(s))  ds 
	+
	 c(t) |m|^{-\alpha}.
\end{align}

Proceeding to bounding the interaction terms, we fix $ j\in\{1,\ldots,j_*\} $.
With $ \av_{\A^K_b}(\bfY^*(t_j)) \geq \tfrac{\rho}{2} $, $ \forall b\in\hZ $,
and $ q(t_{j}-t_{j-1},1) \leq q(\tau,1) $,
following the same procedure of obtaining \eqref{eq:denss:h},
we obtain 
\begin{align}\label{eq:Ipmbj}
	I^{\pm}_{b,j}\in [\tilm^K_{b\mp 1/2},\tilm^K_{b\pm 1/2})\cap\bbZ
	\quad
	\text{such that }
	h_{(I^{\pm}_{b,j},\pm\infty)}(\bfundY^*(t_{j-1},t_j)) \geq \tfrac{\rho}{3},
	\ \forall b\in\hZ.
\end{align}
Now, set $ \calI_j= \fkI:= (I^{-}_{b^--1,j},I^{+}_{b^++1,j}) $ 
and $ \calI'_j= \fkI':= (I^{+}_{b^--1,j},I^{-}_{b^++1,j}) $.
Using \eqref{eq:Ipmbj} to bound $ \barf_{\fkI'}(\bfY^*(s)) $ (as in \eqref{eq:barf}),
we obtain
\begin{align}\label{eq:barfbd}
	\tfrac{ 1 }{ |m| } \barf_{\fkI'}(\bfY^*(s))	
	\leq 
	\tfrac{c}{|m|} \log(|\tilm^K_{b^--3/2}|+|\tilm^K_{b^++3/2}|)
	\leq
	\tfrac{c}{|m|} \log |m|,
	\quad
	\forall s\in[t_{j-1},t_j].
\end{align}
Next, by \eqref{eq:BesComp:sup} we have 
$ \barY^*_{\fkI}(0,t) \leq y^\ic_{\fkI} + Q^{0,t}_{\fkI} $.
With $ \undf_\calI(\bfy) $ and $ \undffpm_\calI(z,\bfy) $ as in \eqref{eq:undf}--\eqref{eq:undffpm},
we have 
\begin{align}\label{eq:undffbd}
	\undf_{\fkJ}(\bfY^*(s)) 
	\leq 
	\undffp_{ \fkI }( y^\ic_{\fkI} + Q^{0,t}_{\fkI} , \bfY^*(s))
	+\undffm_{ \fkI }( y^\ic_{\fkI} + Q^{0,t}_{\fkI} , \bfY^*(s)),
	\quad
	\forall s\in[t_{j-1},t_j].		
\end{align}
Further, by $ \fkI_j\subset\calK' $, $ \bfy^\ic\in\Ysp(\alpha,\rho) $ and \eqref{eq:denss:claim2},
we have $ y^\ic_{\fkI} + Q^{0,t}_{\fkI} \leq c(\rho+q(t,1)) |\calK'\cap\hZ| \leq c(\rho+q(t,1)) |m| $.
Using this and \eqref{eq:Ipmbj} to bound $ \undfpm_{ \fkI }( y^\ic_{\fkI} + Q^{0,t}_{\fkI} , \bfY^*(s)) $,
we obtain
\begin{align}\label{eq:undffbd:}
	&
	\tfrac{1}{|m|} \undffpm_{ \fkI }( y^\ic_{\fkI} + Q^{0,t}_{\fkI} , \bfY^*(s))	
	\leq 
	\tfrac{c}{|m|} \log |m|,
	\quad
	\forall s\in[t_{j-1},t_j].
\end{align}
Combining the preceding bounds \eqref{eq:barfbd}, \eqref{eq:undffbd} and \eqref{eq:undffbd:}
on the interaction terms,
and inserting the result into \eqref{eq:denlw:(0,m):}--\eqref{eq:denup:(0,m):}, 
we conclude \eqref{eq:Ynorm:bd}.
\end{proof}
\begin{proof}[Proof of Proposition~\ref{prop:existGener}]
As stated in the preceding, 
$ \bfY $ is a $ \YTspl(\rho/2) $-valued solution of \eqref{eq:DBMg}.
With this, combining Proposition~\ref{prop:dens} and Lemma~\ref{lem:den2YTsp}
we obtain that $ \bfY\in\YTsp(\alpha,\rho) $.
To show that $ \bfY $ is the greatest solution,
let $ \bfY' $ be a $ \YTsp(\alpha,\rho) $-valued solution with $ \bfY'(0) \leq \bfy^\ic $.
By Proposition~\ref{prop:existg}, $ \bfYgn $ the greatest $ \YTsp $-valued solution,
so, with $ \bfY'(0) \leq (\bfy^\ic\vee\bfgamma_n) = \bfYgn(0) $, 
we must have $ \bfY'(\Cdot) \leq \bfYgn(\Cdot) $.
Letting $ n\to\infty $ we conclude $ \bfY'(\Cdot) \leq \bfY(\Cdot) $.
\end{proof}

\section{Proof of Theorem~\ref{thm:conv}}
\label{sect:conv}
\subsection{Proof of Part\ref{enu:asCnvg}}
Fixing $\bfx^\ic\in\XRsp(\alpha,\rho,p)$,
we let $i^\pm_n$, $\I_n$, $\bfX\in\XRTsp(\alpha,\rho,p) $ and $ \bfX^n\in C([0,\infty))^{\I_n\cap\bbZ} $
be as in Theorem~\ref{thm:conv}.
We consider the corresponding gap processes:
$ \bfy^\ic := u(\bfx^\ic) \in (\Ysp(\alpha,\rho)\cap\Rsp(p)) $,
$ \bfY := u(\bfX) \in (\Ysp(\alpha,\rho)\cap\Rsp(p)) $,
$ \bfY^n := u(\bfX^n) \in C_+([0,\infty))^{\I_n\cap\hZ} $
and let
\begin{align*}
	\bfYR{n}(t) := \BK{\YR{n}{a}(t) }_{a\in\hZ} \in [0,\infty]^\hZ,
	\quad
	\YR{n}{a}(t) 
	:= 
	Y^n_a(t) \ind\curBK{a\in\I_n} 
	+
	\infty \ind\curBK{a\notin\I_n}.
\end{align*}
That is, $ \bfYR{n} $ is constructed from $ \bfY^n $
by declaring all gaps to be $ \infty $ outsides of $ \I_n $.
Hereafter we adapt the convention $ \tfrac{1}{\infty}:=0 $,
so in particular
\begin{align}
	\label{eq:DBMf:}
	&
	X^n_i(t) = x^\ic_i + B_i(t) + \beta \int_0^t \phi_i(\bfYR{n}(s)) ds,
	\quad
	\forall i \in \I_n,
\end{align}
where, recall that (by abuse of notation) 
$ \phi_i(\bfy) := \phi_i( \tilu^{-1}(0,\bfy) ). $
We begin by showing that $ \bfYR{n} $ is decreasing.
\begin{lemma}\label{lem:YrDecr}
We have that, almost surely,
$ \bfYR{1}(\Cdot) \geq \bfYR{2}(\Cdot) \geq \ldots \geq \bfY(\Cdot) $.
\end{lemma}

\begin{proof}
For the gap process $ \bfY^n $, by \eqref{eq:DBMf:}, we have 
\begin{align}\label{eq:DBMgf:}
	Y^n_a(t) = y^\ic_i + W_a(t) + \beta \int_0^t \fS^{\I_n}_a(\bfY^n(s)) ds,
	\quad
	\forall a \in \I_n,
\end{align}
where $ \fS^{\I_n}_a(\bfy) $ is defined as in \eqref{eq:fSA}.
With this, applying Lemma~\ref{lem:DBMcomp} for $ \fkI=\calL_n\cap\hZ $,
$ \bfYup=\bfY^n $, $ \bfZup = \mathbf{0} $,
$ \bfYlw=\bfY^{n+1} $, $ \Zlw_a(s) = -(Y^{n+1}_a(s))^{-1} \f^{ \I_{n+1}\setminus \I_{n} }_a(Y^{n+1}_a,\bfY^{n+1}) $,
we conclude $ \bfY^n(\Cdot) \geq^{\I_n} \bfY^{n+1}(\Cdot) $,
whereby $ \bfYR{n}(\Cdot) \geq \bfYR{n+1}(\Cdot) $.
Similarly, $ \bfYR{n}(\Cdot) \geq \bfY(\Cdot) $ follows by applying Lemma~\ref{lem:DBMcomp},
for $ \fkI=\I_{n}\cap\hZ $.
\end{proof}
An immediate consequence of Lemma~\ref{lem:YrDecr}
is the following convergence of the gap process.
\begin{lemma}\label{lem:Yrconv}
For any fixed $ t\geq 0 $, $ a\in\hZ $ and $ p'\geq 1 $,
we have 
\begin{align}
	&
	\label{eq:YRasCnvg}
	\sup_{s\in[0,t]} |\YR{n}{a}(s) - Y_a(s) | \xrightarrow{\text{a.s.}} 0,
\\
	&
	\label{eq:YRLpCnvg}
	\bbE \Big( \sup_{s\in[0,t]} |\YR{n}{a}(s) - Y_a(s) | \Big)^{p'} \to 0,
\end{align}
as $ n\to\infty $.
\end{lemma}
\begin{proof}
By Lemma~\ref{lem:YrDecr}, the limiting process 
$ \bfY^{\infty}(t) := \bfYR{n}(t) $ exists,
and satisfies $ \bfY^{\infty}(\Cdot) \geq \bfY(\Cdot) \in \YTspl(\rho) $.
With this, letting $ n\to\infty $ in \eqref{eq:DBMgf:} for any fixed $ a\in\hZ $,
using the dominated convergence theorem,
one easily sees that the terms on the r.h.s.\ converges
to the corresponding terms (for $ \bfY^\infty $),
whereby concluding that $ \bfY^\infty $ is in fact a $ \YTspl(\rho) $-valued solution of \eqref{eq:DBMg},
(c.f.\ Proof of Proposition~\ref{prop:bd>0}).
However, by Proposition~\ref{prop:existg},
$ \bfY $ is the greatest such solution,
so we must have $ \bfY^{\infty} = \bfY $.
For any fixed $ a\in\hZ $,
this implies $ \bbP( \lim_{n\to\infty} \YR{n}{a}(s) = Y_a(s))=1 $, $ \forall s\geq 0 $.
As $ \{\YR{n}{a}\}_{n} \subset C([0,\infty)) $ is decreasing, 
by Dini's theorem, this further implies the desired uniform convergence of \eqref{eq:YRasCnvg}.
Fixing arbitrary large enough $ n'\in\bbZ_{>0} $ such that $ \I_{n'}\ni a $,
we have $ \barYR{n}{a}(0,t) \leq \barYR{n'}{a}(0,t) $, $ \forall n\geq n' $.
With this and \eqref{eq:YRasCnvg},
by the dominated convergence theorem we conclude \eqref{eq:YRLpCnvg}.
\end{proof}

By \eqref{eq:YRasCnvg}, it suffices to prove \eqref{eq:asCnvg} for the special case $ i=0 $.
To this end,
we rewrite \eqref{eq:DBM0} and \eqref{eq:DBMf:} for $ i=0 $ as
\begin{align}
	\label{eq:DBM0::}
	X_0(s') - X_0(s)
	&= 
	 B_0(s') - B_0(s) + \beta \int_{s}^{s'} \phi^{\calI}_0(\bfY(s)) ds
	+
	\beta \int_{s}^{s'} \phi^{\calI^c}_0(\bfY(s)) ds,
\\
	\label{eq:DBM0:}
	X^n_0(s') - X^n_0(s) 
	&= 
	B_0(s') - B_0(s) + \beta \int_{s}^{s'} \phi^{\calI}_0(\bfYR{n}(s)) ds
	+
	\beta \int_{s}^{s'} \phi^{\calI^c}_0(\bfYR{n}(s)) ds,
\end{align}
for $ s<s' $ and generic $ \calI := [i^-,i^+]\ni 0 $,
where $ \phi^{\calI}_0(\bfy) := \sum_{i\in\calI\setminus\{0\}} \frac{-\sign(i)}{2y_{(0,i)}} $
and $ \phi^{\calI^c}_0(\bfy) := \phi_0(\bfy) - \phi^{\calI}_0(\bfy) $
denote the interaction within and outside of $ \calI $, respectively.

Now, fixing $ t>0 $, 
$ K\in\bbZ_{>0} $ be as in Proposition~\ref{prop:dens},
and $ I^\pm_{b,j} $ be as in \eqref{eq:Ipmbj} for $ \bfY^*=\bfY $,
where $ j=1,\ldots,j_* $ indexes the equally spaced partition
$ 0=t_0<\ldots<t_{j_*}=t $ with $ t/j_{*} \leq \tau $ 
and $ \tau $ is as in Proposition~\ref{prop:denss}.
Taking the difference of \eqref{eq:DBM0::}--\eqref{eq:DBM0:} 
for $ (s,s')=(t_{j-1},t_j) $ and $ \calI=\fkI_{b,j}:=[I^-_{b,j},I^+_{b,j}] $,
and combing the result for $ j=1,\ldots,j_* $,
we arrive at
\begin{align}
\begin{split}
	&
	\sup_{s\in[0,t]} |X_0(s)-X^n_0(s)|
	\leq
	\beta 
	\sum_{j=1}^{j_*}
	\int_{t_{j-1}}^{t_{j}} \big| \phi^{\fkI_{b,j}}_0(\bfY(s)) - \phi^{\fkI_{b,j}}_0(\bfYR{n}(s)) \big| ds
\\
	&
	\label{eq:X0diff}
	\quad
	+
	\beta \sum_{j=1}^{j_*}
	\int_{t_{j-1}}^{t_{j}} 
		\Big( 
			\big| \phi^{(\fkI_{b,j})^c}_0(\bfY(s))\big| 
			+ \big|\phi^{(\fkI_{b,j})^c}_0(\bfYR{n}(s)) | 
		\Big)ds.
\end{split}
\end{align}
By \eqref{eq:YRasCnvg},
the first term on the r.h.s.\ of \eqref{eq:X0diff}
tends to $ 0 $ as $ n\to\infty $, for each fixed $ b<\infty $.
With this, it suffices to estimate the last term in \eqref{eq:YRasCnvg}.
\begin{lemma}\label{lem:phiOut}
Let $ t>0 $, $ K\in\bbZ_{>0} $, $ I^\pm_{b,j} $ and $ \fkI_{b,j}:=[I^-_{b,j},J^+_{b,j}] $ 
be as in the preceding.
There exists $ c=c(t,\bfy^\ic,\rho) $ such that
\begin{align}
	\label{eq:phiOut:bd}
	&
	\sup_{ s\in[t_{j-1},t_j] }  
	\big| \phi^{(\fkI_{b,j})^c}_0(\bfY(s)) \big|
	\leq
	c (\tilm^K_{b-1/2})^{-\alpha} + c (b-1/2)^{-\alpha},	
\\
	\label{eq:phiOut:bd:}
	&
	\sup_{ s\in[t_{j-1},t_j] }  
	\big| \phi^{(\fkI_{b,j})^c}_0(\bfYR{n}(s)) \big|
	\leq
	c (\tilm^K_{b-1/2})^{-\alpha} + c (b-1/2)^{-\alpha},
\end{align}
for all $ b\in\hZ\cap(2,\infty) $, $ j=1,\ldots,j_* $, $ n\in\bbZ_{>0} $.
\end{lemma}
\noindent
Assuming the estimates \eqref{eq:phiOut:bd}--\eqref{eq:phiOut:bd:},
we proceed to completing the proof of \eqref{eq:asCnvg}.
Inserting \eqref{eq:phiOut:bd}--\eqref{eq:phiOut:bd:} into \eqref{eq:X0diff}
and letting $ n\to\infty $ 
yields
\begin{align*}
	\limsup_{n\to\infty} \Big\{ \sup_{s\in[0,t]} |X_0(s)-X^n_0(s)| \Big\}
	\leq
	c (\tilm^K_{b-1/2})^{-\alpha},
\end{align*}
for any fixed $ b\in(2,\infty)\cap\hZ $.
From this, further letting $ b\to\infty $,
we conclude \eqref{eq:asCnvg} for $ i=0 $ and hence for all $ i\in\bbZ $.

\begin{proof}[Proof of Lemma~\ref{lem:phiOut}]
Fixing such $ b,n,j $ and $ s\in[t_{j-1},t_j] $,
we let $ \fkI := \fkI_{b,j} $ and
$ c=c(t,\bfy^\ic,\rho)<\infty $ denote a generic finite constant.
To prove \eqref{eq:phiOut:bd},
we express $ \fkJ^c $ as $ (\fkI^\vee)^c \cup \fkB $,
where $ \fkI^\vee $ is the `symmetrized' interval 
and $ \fkB $ is the `boundary' interval, defined as
\begin{align*}
	\fkI^\vee := [-I^\vee, I^\vee],
	\text{ where }
	I^\vee := (I^+_{b,j} \vee |I^-_{-b,j}|),
	\quad
	\fkB := \fkI^\vee\setminus\fkI.
\end{align*}
With this, we similarly decompose $ \phi^{\fkI^c}(\bfY(s)) $ as
\begin{align}\label{eq:phiDecmp}
	\phi^{\fkI^c}_0(\bfY(s)) 
	=
	\phi^{(\fkI^\vee)^c}_0(\bfY(s)) + \phi^{\fkB}_0(\bfY(s)).
\end{align}
Proceeding to bounding the interactions on the r.h.s., we clearly have
\begin{align}\label{eq:temp1}
	\phi^{(\fkI^\vee)^c}_0(\bfY(s))
	=
	\sum_{ I^\vee \leq m } \frac{ Y_{(-m,0)}(s) - Y_{(0,m)}(s) }{ 2Y_{(0,m)}(s) Y_{(-m,0)}(s) },
\quad
	|\phi^{ \fkB }_0(\bfY(s))|
	\leq
	\sum_{ I^\wedge \leq |m| \leq I^\vee } \frac{1}{2Y_{(0,m)}(s)}.
\end{align}
By \eqref{eq:Ipmbj}, we have
\begin{align}\label{eq:temp3}
	Y_{(0,m)}(s) \geq \tfrac{\rho}{3} ( |m| - |I^\pm_{\pm1/2,j} | ) \geq \tfrac{1}{c} |m|,
	\quad
	\forall
	|m| \geq I^{\wedge},
\end{align}
where the last inequality follows since $ b>2 $.
With $ \bfY \in \YTsp(\alpha,\rho) $,
we have $ |Y_{(-m,0)}(s) - Y_{(0,m)}(s)| \leq c|m|^{1-\alpha} $.
Using this and \eqref{eq:temp3} in \eqref{eq:temp1}, we obtain
\begin{align}
	\notag
	|\phi^{(\fkI^\vee)^c}_0(\bfY(s))|
	&\leq
	c \sum_{ I^\vee \leq m } \frac{ m^{1-\alpha} }{ m^2 }
	\leq
	c (I^{\vee})^{-\alpha}
	\leq
	c (\tilm^K_{b-1/2})^{-\alpha},
\\
	\label{eq:temp2}
	|\phi^{ \fkB }_0(\bfy)|
	&\leq
	\sum_{ I^\wedge \leq |m| \leq I^\vee } \frac{c}{|m|}
	\leq
	c \log \big( I^\vee/I^\wedge \big).
\end{align}
With $ \tilm^K_i := \floorBK{ (Ki)^{1/\alpha} } $,
we have
\begin{align}\label{eq:Iratio} 
	1 
	\leq
	(I^\vee/I^\wedge) 
	\leq 
	(\tilm^K_{b+1/2}/\tilm^K_{b-1/2})
	\leq 
	1 + c (b-1/2)^{-\alpha}.
\end{align}
Inserting \eqref{eq:Iratio} into \eqref{eq:temp2}
then yields $ |\phi^{ \fkB }_0(\bfy)| \leq c (b-1/2)^{-\alpha} $,
from which \eqref{eq:phiOut:bd}.

Turning to proving \eqref{eq:phiOut:bd:},
recalling $ \I_n := [i^-_n,i^+_n] $,
we let $ \I^\wedge := (-i^\wedge, i^\wedge) $,
where $ i^\wedge := (i^+_{n} \wedge |i^-_{n}|) $.
With
$ 	
	\phi^{\fkI^c}_0(\bfYR{n}(s))
	=
	\phi^{\I_n \setminus \fkI }_0(\bfYR{n}(s)),
$
similar to \eqref{eq:phiDecmp},
we decompose $ \phi^{\fkI^c}_0(\bfYR{n}(s)) $ as
$
	\phi^{\fkI^c}_0(\bfYR{n}(s))
	=
	\phi^{\I^\wedge \setminus \fkI^\vee }_0(\bfYR{n}(s)) + \phi^{\fkB'}_0(\bfYR{n}(s)),
$
where 
$
	\fkB' := (\I_n \setminus \fkI)\setminus (\I^\wedge \setminus \fkI^\vee ),
$
is the boundary interval.
To further bound the interactions on the r.h.s., 
similar to \eqref{eq:temp1} we have
\begin{align*}
	\phi^{\I^\wedge \setminus \fkI^\vee }_0(\bfYR{n}(s))
	&=
	\sum_{ I^\vee \leq m \leq i^\wedge } 
	\frac{ \YR{n}{(-m,0)}(s) - \YR{n}{(0,m)}(s) }{ 2\YR{n}{(-m,0)}(s)\YR{n}{(0,m)}(s) },
\\
	|\phi^{ \fkB }_0(\bfYR{n}(s))|
	&\leq
	\sum_{ (i^\wedge\vee I^\wedge )\leq |m| \leq i^\vee } \frac{1}{ 2\YR{n}{(0,m)}(s) }
	+
	\sum_{  I^\wedge \leq |m| \leq I^\vee } \frac{1}{ 2\YR{n}{(0,m)}(s)  },
\end{align*}
where $ i^\vee_n := (i^+_{n} \vee |i^-_{n}|). $
With $ \bfYR{n}(s) \geq \bfY(s) $, from \eqref{eq:temp3} we further obtain
\begin{align}
	\label{eq:phiOut:N:}
	|\phi^{\I^\wedge \setminus \fkI^\vee }_0(\bfYR{n}(s))|
	&\leq
	c
	\sum_{ I^\vee \leq m \leq i^\wedge } 
	\frac{ 1 }{ m^2 } |\YR{n}{(-m,0)}(s) - \YR{n}{(0,m)}(s)|,
\\
	\notag
	|\phi^{ \fkB }_0(\bfYR{n}(s))|
	&\leq
	\sum_{ (i^\wedge\vee I^\wedge )\leq |m| \leq i^\vee } \frac{ c }{ |m| }
	+ 
	\sum_{  I^\wedge \leq |m| \leq I^\vee } \frac{ c }{ |m| }
\\
	\label{eq:phiOut:N}
	&
	\quad\quad
	\leq
	c\ind_\curBK{I^\wedge<i^\vee} \log \big( i^\vee/i^\wedge \big)
	+ c \log \big( I^\vee/I^\wedge \big).
\end{align}
With $ \bfx^\ic\in\XRsp(\alpha,\rho,p) $ and $ i^\pm_n $ as in \eqref{eq:I},
we have $ 1\leq i^\vee/i^\wedge \leq 1 + c(i^\vee)^{-\alpha} $.
Using this and \eqref{eq:Iratio} in \eqref{eq:phiOut:N},
we conclude 
$
	|\phi^{ \fkB }_0(\bfYR{n}(s))| 
	\leq 
	c (\tilm^K_{b-1/2})^{-\alpha} + c(b-1/2)^{-\alpha} 
$.

It remains only to bound the expression \eqref{eq:phiOut:N:}.
To this end, we fix $ I^\vee \leq m \leq i^\wedge $.
With $ \bfY^n $ solving \eqref{eq:DBMgf:},
similar to \eqref{eq:denlw:(0,m):}--\eqref{eq:denup:(0,m):}, we have
\begin{align*}
	\inf_{s\in[0,t]}
	\av_{(0,\pm m)} ( \bfY^n(s))
	\geq
	\av_{(0,\pm m)} ( \bfy^\ic )
	-
	\sum_{j=1}^{j_*} 
	\beta
	\int_{t_{j-1}}^{t_j} \frac{1}{m} \undf_{ \calI_j }(\bfYR{n}(s))  ds 
	-
	 c m^{-\alpha},
\\
	\sup_{s\in[0,t]}
	\av_{(0,\pm m)} ( \bfY^n(s))
	\leq
	\av_{(0,\pm m)} ( \bfy^\ic )
	+ 
	\sum_{j=1}^{j_*} 
	\beta
	\int_{t_{j-1}}^{t_j} \frac{1}{m} \barf_{ \calI'_j }(\bfYR{n}(s)) ds 
	+
	 c m^{-\alpha},
\end{align*}
for any $ \calI_j, \calI'_j $ such that $ (0,\pm m) \subset \calI_j,\calI'_j \subset \I_n $.
Hence, with $ \bfy^\ic\in\Ysp(\alpha,\rho) $,
\begin{align}
	\notag
	\sup_{s\in[0,t]} |Y^n_{(0,\pm m)}(s)- & m \rho|
	\leq
	c m^{1-\alpha}
	+
	\sum_{j=1}^{j_*} 
	\beta
	\int_{t_{j-1}}^{t_j} \undf_{ \calI_j }(\bfYR{n}(s))  ds 
	+
	\sum_{j=1}^{j_*} 
	\beta
	\int_{t_{j-1}}^{t_j} \barf_{ \calI'_j }(\bfYR{n}(s)) ds.
\end{align}
Proceeding to bounding the interaction terms,
by \eqref{eq:BesComp:sup} we have $ \barYR{n}{\calI_j}(0,t) \leq y^\ic_{\calI_j} + Q^{0,t}_{\calI_j} $,
so
\begin{align*}
	\undf_{ \calI_j }(\bfYR{n}(s)) 
	\leq
	\undffpm_{ \calI_j }(y^\ic_{\calI_j} + Q^{0,t}_{\calI_j}, \bfYR{n}(s)),
	\quad
	\forall s\in[t_{j-1},t_j],
\end{align*}
where $ \undffpm_{\calI_j}(z,\bfy) $ is defined as in \eqref{eq:undffpm}.
With $ \bfY(\Cdot) \leq \bfYR{n}(\Cdot) $,
we further obtain
\begin{align*}
	\undffpm_{ \calI_j }(y^\ic_{\calI_j} + Q^{0,t}_{\calI_j}, \bfYR{n}(s))
	\leq
	\undffpm_{ \calI_j }(y^\ic_{\calI_j} + Q^{0,t}_{\calI_j}, \bfY(s)),
	\quad
	\barf_{ \calI'_j }(\bfYR{n}(s))
	\leq
	\barf_{ \calI'_j }(\bfY(s)).	
\end{align*}
With this, applying the bounds \eqref{eq:barfbd} and \eqref{eq:undffbd:} for $ \bfY^*=\bfY $ 
on the interaction terms, we obtain
\begin{align}\label{eq:densYR}
	\sup_{s\in[0,t]} |Y^n_{(0,\pm m)}(s) - \rho m|
	\leq c m^{1-\alpha},
	\quad
	\forall m \geq \tilm^K_{1}.
\end{align}
Inserting this into \eqref{eq:phiOut:N:} yields
$ 
	|\phi^{\I^\wedge \setminus \fkI^\vee }_0(\bfYR{n}(s))| 
	\leq c/(I^\vee)^{\alpha} \leq c (\tilm^{K}_{b-1/2})^{-\alpha}  
$,
whereby completing the proof.
\end{proof}

\subsection{Proof of Part\ref{enu:LpCnvg}}
By \eqref{eq:YRLpCnvg}, it suffices to prove \eqref{eq:LpCnvg} for $ i=0 $.
This, by \eqref{eq:asCnvg} for $ i=0 $,
is further reduced to showing the boundedness of $ \{ \bbE (\overline{|X^n_0|}(0,t) )^{p'} \}_{n\in\bbZ_{>0}} $,
for all $ p'\geq 1 $.
To this end, fixing arbitrary $ p'\geq 1 $ and $ t\geq 0 $,
we combine \eqref{eq:DBM0:} and Lemma~\ref{lem:phiOut} for $ b=5/2 $
to get 
\begin{align}\label{eq:Xn0:bd}
	\overline{|X^n_0|}(0,t)
	\leq
	\overline{|B_0|}(0,t)
	+ c(1+(\tilm^K_{2})^{-\alpha})	
	+ \beta \sum_{j=1}^{j_*} \int_{t_{j-1}}^{t_j} \phi^{\fkL_j}_0(\bfYR{n}(s)) ds,
\end{align}
for $ \fkL_j := [I^-_{-5/2,j},I^+_{5/2,j}] $ and $ c=c(t,\bfy^\ic,\rho)<\infty $. 
We proceed to establishing a bound on the last term.
%
\begin{lemma}\label{lem:phiIn}
Let $ \fkL_j := [I^-_{-5/2,j},I^+_{5/2,j}] $.
There exists $ c=c(t)<\infty $ such that
\begin{align}\label{eq:phiIn:bd}
	\beta
	\int_{t_{j-1}}^{t_j} | \phi^{\fkL_j}_0(\bfYR{n}(s)) | ds
	\leq
	c \tilm^K_{3} \BK{ B^* + (1 + \rho)\tilm^K_{3}  },
	\quad
	\forall
	n\in\bbZ_{>0}, j=1,\ldots,j_*.
\end{align}
where $ B^* := \sum_{i\in[\tilm^K_{-3},\tilm^K_{3}]} \overline{|B_i|}(0,t) $.
\end{lemma}
\begin{proof}
To simply notations, fixing $ j=1,\ldots,j_* $,
we let $ I^\pm:=I^{\pm}_{\pm5/2,j} $ and $ \fkL:=[I^-,I^+] $.
Letting $ \phi^{\fkL}_i(\bfy) := \sum_{j\in\fkL\setminus\{i\}} \frac{\sign(i-j)}{2y_{(i,j)}} $
denote the restriction of $ \phi_i(\bfy) $ onto $ \fkL $,
we define 
\begin{align}\label{eq:phiIn:id}
	\tilphi^{\fkL}_{i_*}(\bfy)
	:=
	\sum_{ \ell \in [I^-,i_*]}
	\sum_{ \ell' \in (i_*,I^+] } \frac{1}{2 y_{(\ell,\ell')}}
	=
	-
	\sum_{i\in[I^-,i_*]} \phi^{\fkL}_{i_*}(\bfy),
\end{align}
where the last equality is easily verified by substituting in 
the preceding definition of $ \phi^{\fkL}_i(\bfy) $.

With $ \tilphi^{\fkL}_{i_*}(\bfy) $ defined as in \eqref{eq:phiIn:id},
we indeed have $ |\phi^{\fkL}_0(\bfy)| \leq \tilphi^{\fkL}_0(\bfy) + \tilphi^{\fkL}_{-1}(\bfy)  $,
so, instead of proving \eqref{eq:phiIn:bd},
we establish the corresponding bound on $ \beta \int_{t_{j-1}}^{t_j} \tilphi^{\fkL}_{i_*}(\bfYR{n}(s)) ds $
for $ i_*=0,-1 $.
Fix $ i_*=0, -1 $ and $ i\in [I^-,i_*] $.
With $ \bfX^n $ satisfying \eqref{eq:DBMf:},
we have
\begin{align*}
	Y^n_{(i,I^+)}(s) \Big|_{s=t_{j-1}}^{s=t_j}
	=
	\BK{ B_{I^+}(s) - B_{i}(s) } \Big|_{s=t_{j-1}}^{s=t_j}
	+
	\beta \int_{t_{j-1}}^{t_j} \fS_{(i,I^+)}(\bfYR{n}(s)) ds,
\end{align*}
where, recall that
$
	\fS_{(i,I^+)}(\bfy) := \sum_{i\in(i,I^+)} \fS_{a}(\bfy).
$
With $ \eta_a(\bfy) = \phi_{a+1/2}(\bfy)-\phi_{a-1/2}(\bfy) $,
we have
$
	\fS_{(i,I^+)}(\bfy) = \phi_{I^+}(\bfy) - \phi_{i}(\bfy),
$
for all $ \bfy\in\Ysp(\alpha,\rho) $.
Further decompose the last expression as
$
	\phi^{\fkL}_{I^+}(\bfy) - \phi^{\fkL}_{i}(\bfy) + \fS'_{(i,I^+)}(\bfy),
$
where
\begin{align*}
	\fS'_{(i,I^+)}(\bfy) 
	:=
	-\sum_{\ell >I^+} 
	\frac{ y_{(i,I^+)} }{ y_{(I^+,\ell)} ( y_{(i,I^+)} + y_{(I^+,\ell)}) }
	+\sum_{\ell <I^-} 
	\frac{ y_{(i,I^+)} }{ y_{(\ell,I^-)} ( y_{(i,I^+)} + y_{(\ell,I^-)}) }.
\end{align*}
With $ \phi^{\fkL}_{I^+}(\bfy) >0 $,
we then obtain
\begin{align}\label{eq:phiIn:eq}
	2 \barY^n_{\fkL}(0,t) + 2 B^* + \beta \int_{t_{j-1}}^{t_j} |\fS'_{(i,I^+)}(\bfYR{n}(s))| ds
	\geq
	\beta \int_{t_{j-1}}^{t_j} \BK{ - \phi^{\fkL}_i(\bfYR{n}(s)) } ds.
\end{align}
For the integral term on the l.h.s., using
$
	h_{(I^\pm,\pm \infty)}(\bfYR{n}(s))
	\geq
	h_{(I^\pm,\pm \infty)}(\bfundY(t_{j-1},t_j)) \geq \tfrac{\rho}{3},
$
$ \forall s\in[t_{j-1},t_j], $
we obtain 
\begin{align*}
	|\fS'_{(i,I^+)}(\bfYR{n}(s))| \leq c\log( Y^n_{(i,I^+)}(s) +1 ).
\end{align*}
Plugging this in \eqref{eq:phiIn:eq} yields
\begin{align*}
	\beta \int_{t_{j-1}}^{t_j} 
	\BK{ - \phi^{\fkL}_i(\bfYR{n}(s)) } ds
	\leq
	2B^* + c(\barY^n_{\fkL}(t_{j-1},t_{j}) + 1)
	\leq
	c \BK{ B^* + (\tilm^K_{3})^{1-\alpha} + \tilm^K_{3}\rho },
\end{align*}
where the last inequality follows by \eqref{eq:densYR}.
Summing this over $ i=I^-,\ldots,i_* $,
using \eqref{eq:phiIn:id},
we conclude the desired bound 
$
	\beta \int_{t_{j-1}}^{t_j} 
	\tilphi^{\fkL}_{i_*}(\bfYR{n}(s)) ds
	\leq
	c \tilm^K_{3} ( \hatB + (\tilm^K_{3})^{1-\alpha} + \rho ).
$
\end{proof}

Now, inserting \eqref{eq:phiIn:bd} into \eqref{eq:Xn0:bd},
and taking the $ p' $-th moment of both sides,
we arrive at
$
	\bbE ( \overline{|X^n_0|}(0,t) )^{p'}
	\leq
	c + c \bbE (\tilm^K_{3})^{2p'} 
	+	c\bbE( (B^*)^{2p'}).
$ 
With $ \tilm^K_i := \floorBK{ (iK)^{\frac{1}{\alpha}} } $,
by \eqref{eq:Ktail}, we have $ \bbE (\tilm^K_{3})^{2p'} <\infty $.
As for $ \bbE( (B^*)^{p'} ) $,
applying 
\begin{align}
	\notag
	\bbE \Big| \sum_{j\in[\tilm^K_{-i},\tilm^K_{i}]} F_j \Big|
	&\leq
	\bbE \Big( \sum_{m\geq 0} \ind\curBK{\tilm^K_i =m} \sum_{j\in[-m,m]} |F_j| \Big)
\\
	\label{eq:unionBd}
	&\leq
	\sum_{m\geq 0} \Bigg( \bbP\BK{\tilm^K_i \geq m} \bbE \Big(\sum_{j\in[-m,m]} |F_i| \Big)^2 \Bigg)^{1/2}
\end{align}
for $ F_j = \overline{|B_j|}(0,t) $ and $ i=3 $,
we obtain $ \bbE( (B^*)^{p'} ) \leq	c\sum_{m\geq 0} (2m+1) \BK{ \bbP( \tilm^K_{3} \geq m ) }^{1/2} $,
which, by \eqref{eq:Ktail}, is finite.
From this, we conclude the desired bound
\begin{align}\label{eq:X0Lpbd}
	\bbE ( \overline{|X^n_0|}(0,t) )^{p'}
	\leq
	c(t,p'),
	\quad
	\forall \I_n\ni 0.
\end{align}

\subsection{Proof of Part\ref{enu:crlCnvg}}
Fixing $ s_1,\ldots, s_{j_*} \in [0,t] $ and open sets $ \calO_1,\ldots,\calO_{j_*}\subset [-r,r] $,
we let
$
	L_n := \sup \{ |i|: i\in\bbZ \text{ such that } X^n_i(s)\in[-r,r], \text{ for some } s\in[0,t] \}
$
denote the maximal relevant index
and similarly define
$
	L_\infty := \sup \{ |i|: i\in\bbZ \text{ such that } X_i(s)\in[-r,r], \text{ for some } s\in[0,t] \}.
$
We begin by establishing the following tail bound on $ L_n $ and $ L $.
\begin{lemma}\label{lem:Lnbd}
For any fixed $ p'\geq 1 $, there exists $ c=c(t,r,\rho,p')<\infty $ such that
\begin{align*}
	\bbP(L_n\geq\ell) \leq c \ell^{-p'},
\end{align*}
for all $ \ell\in\bbZ_{>0} $ and $ n\in\bbZ_{>0}\cup\{\infty\} $.
\end{lemma}
\begin{proof}
We prove only the tail bound for $ L_n $, $ n<\infty $, 
as the one for $ L_\infty $ is proven similarly.
Fixing $ p'\geq 1 $ and $ n,\ell \in\bbZ_{>0} $, we let $ c=c(t,r,\rho,p')<\infty $ denote a generic finite constant.
Indeed, $ L_n \geq \ell $ only if $ \underline{ |X^n_i| }(0,t) \leq r $ for $ i=\ell $ or $ -\ell $, whereby
\begin{align*}
	\bbP\BK{ L_n \geq \ell }
	\leq
	\bbP\BK{ \underline{|X^n_\ell|}(0,t) \leq r } + \bbP\BK{ \underline{|X^n_{-\ell}|}(0,t) \leq r }.
\end{align*}
As $ X^n_i(s) = X^n_0(s) + Y^n_{(0,i)}(s) $,
letting $ Y^*_i := \inf_{s\in[0,t]} Y^n_{(0,i)}(s) $,
we have
$
	\bbP( \underline{|X^n_{\pm\ell}|}(0,t) \leq r )
	\leq
	\bbP( Y^*_{\pm\ell} \leq r + \overline{|X^n_0|}(0,t) ).
$
Next, by \eqref{eq:densYR} we have
\begin{align*}
	Y^{*}_{\pm\ell} 
	\geq 
	(\ell\rho - c \ell^{1-\alpha}) \ind\curBK{ \tilm^K_1 \leq \ell }
	\geq
	\ell\rho - c \ell^{1-\alpha} - c \tilm^K_1.
\end{align*}
From this we conclude
\begin{align*}
	\bbP \BK{ \underline{|X^n_{\pm\ell}|}(0,t) \leq r }
	&\leq
	\bbP \BK{ \ell - c\ell^{1-\alpha} - r \leq  c\tilm^K_1 + \overline{|X^n_0|}(0,t) }
\\
	&
	\leq
	c \ell^{-p'}
	\BK{ \bbE (\tilm^K_1)^{p'} + \bbE(\overline{|X^n_0|}(0,t))^{p'} }.
\end{align*}
By \eqref{eq:X0Lpbd}, $ \bbE(\overline{|X^n_0|}(0,t))^{p'} < \infty $,
and, with $ \tilm^K_1 := \floorBK{K^{1/\alpha}} $ and \eqref{eq:unionBd},
we have $ \bbE (\tilm^K_1)^{p'} <\infty $,
thereby completing the proof.
\end{proof}

For any $ \ell\in\bbZ_{>0}\cup\{\infty\} $,
we let $ f_\ell(\bfx) := \prod_{j=1}^{j_*} | \calO_j \cap \{x_i\}_{|i|\leq \ell} | $.
Our goal is to show $ \bbE(f_\infty(\bfX^n)) \to \bbE(f_\infty(\bfX)) $ and $ \bbE(f_\infty(\bfX)) < \infty $.
The latter follows directly from Lemma~\ref{lem:Lnbd}:
\begin{align*}
	\bbE(f_\infty(\bfX^n)) \leq \bbE(2L_n+1) \leq c(t,r,\rho) < \infty.
\end{align*}
To prove the former, we write
\begin{align}\label{eq:ff}
	\bbE(f_\infty(\bfX^n)) - \bbE(f_\infty(\bfX))
	=
	f'_1 + \big( \bbE(f_\ell(\bfX^n)) - \bbE(f_\ell(\bfX)) \big) + f'_2,
\end{align}
where
$ f'_1 := \bbE(f_\infty(\bfX^n)) - \bbE(f_\ell(\bfX^n)) $
and $ f'_2 := \bbE(f_\ell(\bfX)) - \bbE(f_\infty(\bfX)) $.
For any fixed $ \ell<\infty $,
by \eqref{eq:asCnvg},
we clearly have $ \bbE(f_\ell(\bfX^n))-\bbE(f_\ell(\bfX)) \to 0 $, as $ n\to\infty $.
Further, by Lemma~\ref{lem:Lnbd} we have
\begin{align*}
	&
	\bbE|f_\infty(\bfX^n)-f_\ell(\bfX^n)| 
	\leq 
	\bbE\BK{ (2L_n+1)\ind\curBK{ L_n \geq \ell } }
	\leq
	\tfrac{1}{2\ell+1} \bbE(2L_n+1 )^{2}
	\leq
	c(t,r,\rho) \tfrac{1}{\ell},
\\
	&
	\bbE|f_\infty(\bfX)-f_\ell(\bfX)| 
	\leq 
	\bbE\BK{ (2L_\infty+1)\ind\curBK{ L \geq \ell } }
	\leq
	\tfrac{1}{2\ell+1} \bbE(2L_\infty+1 )^{2}
	\leq
	c(t,r,\rho) \tfrac{1}{\ell}.
\end{align*}
Hence letting $ n\to\infty $ and $ \ell\to\infty $ in \eqref{eq:ff} in order,
we conclude the desired result.

\section{Regularity of Near-Equilibrium Solutions}
\label{sect:nearEq}
We begin by defining near-equilibrium solutions.
Let $\calU$ denote the space of all simple point processes on $\bbR$.
That is, the space of all $\bbZ_{\geq 0}$-valued Radon measure $\chi$ such that $\chi(\{x\})\leq 1$ for all $x\in\bbR$.
Fixing $ \beta=1,2,4 $ hereafter, let $ N\in\calU $ denote the sine process (see, e.g.\ \cite[(2.17), (9.3)]{osada12} for the definition).
With $v: \Weyl\rightarrow \calU$, $(x_i)_{i\in\bbZ} \mapsto \sum_{i\in\bbZ} \delta_{x_i}$
denoting the map from labeled configurations to unlabeled configurations,
we say a weak solution $ \bfX $ of \eqref{eq:DBM} is near-equilibrium if
there exists $ \sine\subset\calU $ such that $\bbP(N\in\sine)=1$ and that $ \bbP(v(\bfX(t))\in\sine)=1 $, for all $ t\geq 0 $.
The motivation is to relate the solutions constructed in \cite{osada12} to that of this paper.
In \cite{osada12}, a near-equilibrium solution is constructed for each initial condition $ \bfx^\ic\in\sine $. 
\begin{remark}
In \cite{osada12}, the interaction $\phi_i(\bfx)$ is defined slightly different as\\ 
$	
	\phi'_i \BK{\bfx} :=
	\lim_{r\to\infty}\sum_{|x_j-x_i|<r,  i\neq j } 
	\frac{1}{2(x_i-x_j)},
$
which is clearly equivalent to \eqref{eq:phi} for al $\bfx\in \Xsp(\alpha,\rho)$.
\end{remark}

We first show that the sine process is $ v(\XRsp(\alpha,1,p)) $-valued.
\begin{lemma}\label{lem:sine}
We have $\bbP(N \in v(\XRsp(\alpha,1,p)))=1$,
for $ \alpha\in(0,1/2) $ and $ p>1 $.
\end{lemma}
\begin{remark}
Neither the determinantal or Pfaffian structure is directly used in the proof of Lemma~\ref{lem:sine}.
More precisely,
letting 
\begin{align*}
	G_x := \big( \inf\curBK{x': N([x',x])=0 }, \sup\curBK{x': N([x,x'])=0 } \big) \subset \bbR
\end{align*}
denoting the gap around $x$,
and $ |G_x| $ denoting the \emph{length} of $ G_x $,
in the following proof of Lemma~\ref{lem:sine},
we use \emph{only} the translation invariance and the following two properties of $N$:
\begin{align}
	&
	\label{eq:sine:var}
	\bbE\BK{ N([x_1,x_2]) - (x_2-x_1) }^2 \leq c\log(2+|x_2-x_1|), \quad \forall x_1\leq x_2\in\bbR, 
\\
	&
	\label{eq:sine:gap}
	\bbE{e^{\gamma |G_x|}} < \infty, \quad \forall \gamma\in\bbR,
\end{align}
which are proven in \cite[Section A.38]{mehta04} and \cite[Theorem 5]{valko09}, respectively.
\end{remark}
\begin{proof}
Throughout this proof, we let $ c=c(\alpha,p)<\infty $ denote a generic finite constant.
Our goal is to prove $ N\in v(\Xsp(\alpha,1)) $ and $ N\in v(\Rsp(p)) $ almost surely.
These conditions, by the duality relation
\begin{align*}
	\{\chi ([0,r]) < n \} = \{ x_{n+i_*} >r \},
	\quad
	\text{ where }
	\bfx \in \Weyl , \ x_{i_*}<0\leq x_{i_*+1}, \ \chi := v(\bfx),
\end{align*}
are equivalent to
\begin{align}
	&
	\label{eq:sine:Xsp}
	\sup_{r\in\bbR} \big| N([0,r]) - |r| \big| |r|^{\alpha-1}<\infty,
\\
	&
	\label{eq:sine:Rsp}
	\sup_{m\in\bbZ}
		\Big\{ \frac{1}{|\fkG_m|} \sum_{\calI\in\fkG_m} |\calI|^{p}  \Big\}
	< \infty,	
\end{align}
where $ \fkG_m:= \{ (\gamma,\gamma')=G_x: x\in\bbR, G_x\subset [0,m] \} $ 
denote the set of all gaps contained in $ [0,m] $.

We begin by proving \eqref{eq:sine:Xsp}.
Let $\calI_j:=[(j-1)^{1/\alpha},j^{1/\alpha})$ for $ j\in\bbZ_{>0} $
and $ \calI_{j} := -\calI_{-j} $ for $ j\in\bbZ_{<0} $.
Combining \eqref{eq:sine:var} and the Chebyshev's inequality, we obtain
\begin{align*}
	&
	\bbP\BK{ \absBK{ N(\calI_j) } \geq 2|\calI_j| } 
	\leq
	c |\calI_j|^{-2} \log|\calI_j| 
	\leq
	c |j|^{2(1-\alpha)/\alpha} \log|j|,
\\
	&
	\bbP \big( \big| N([0,j]) - |j| \big| \geq |j|^{1-\alpha} \big) \leq c |j|^{-2+2\alpha}\log|j|.
\end{align*}
With $\alpha<1/2$, the r.h.s.\ are finite when being summed over $ j\in\bbZ\setminus\{0\} $.
Consequently, by the first Borel-Cantelli lemma have
\begin{align}\label{eq:sine:BC}
	\sup_{j\in\bbZ\setminus\{0\} } N(\calI_j) \absBK{\calI_j}^{-1} <\infty,
	\quad
	\sup_{j\in\bbZ\setminus\{0\} } \Big| N([0,j]) - |j| \Big| |j|^{1-\alpha} <\infty,
\end{align}
almost surely.
Now, fixing arbitrary $r\in\bbR$,
we let $j_*\in\bbZ_{>0}$ be such that $r\in \calI_{j_*}$, 
and let $k_*\in\calI_{j_*}\cap\bbZ$ be arbitrary.
With $N([0,r]) \leq N([0,k_*]) + N(\calI_{j_*})$,
we obtain
\begin{align*}
	| N([0,r])-r | r^{\alpha-1}
	\leq
	\big| N([0,{k_*}])-k_* \big| r^{\alpha-1}
	+
	|k_*-r| r^{\alpha-1}
	+
	 N(\calI_{j_*}) r^{\alpha-1}.
\end{align*}
Further using $ |r| \geq (|j_*|-1)^{1/\alpha} $ and \eqref{eq:sine:BC},
we conclude \eqref{eq:sine:Xsp}.

Turning to \eqref{eq:sine:Rsp},
letting $\fkG'_m$ denote the set of all gaps in $ [0,m] $ with length greater than $ 1 $,
we have
\begin{align*}
	\sup_{m\in\bbZ}
		\Big\{ \frac{1}{|\fkG_m|} \sum_{\calI\in\fkG_m} |\calI|^{p}  \Big\}
	\leq
	1 + \frac{1}{|\fkG_m|} \sum_{\calI\in\fkG'_m} |\calI|^p.
\end{align*}
Further, for each $ \calI\in\fkG'_m $, we must have $ \calI\cap\bbZ\neq\emptyset $,
so
$ 
	\sum_{\calI\in\fkG'_m} (|\calI|^p) \leq \sum_{i=1}^m |G_i|.
$
With $N([0,m])-1 \leq |\fkG_m| \leq N([0,m]) $, by \eqref{eq:sine:Xsp}
we have $\lim_{|m|\to\infty} \frac{|m|}{|\fkG_m|} =1 $ almost surely.
Consequently, letting $ \bfG := (G_i)_{i\in\bbZ} $, we have
\begin{align}\label{eq:sine:Rsp:}
	\sup_{m\in\bbZ}
	\Big\{ \frac{1}{|\fkG_m|} \sum_{\calI\in\fkG_m} |\calI|^{p}  \Big\}
	\leq
	C  + C \sup_{m\in\bbZ} \{ \av^p_{[0,m]}(\bfG) \},
\end{align}
for some $ C<\infty $ almost surely.
As the sine process is translation invariant, $ \bfG $ is shift-invariant.
With this, by \eqref{eq:sine:gap} and the Birkhoff--Khinchin ergodic theorem,
we obtain that $ \lim_{|m|\to\infty} \{ \av^p_{[0,m]}(\bfG) \} = G<\infty $ almost surely,
so in particular the r.h.s.\ of \eqref{eq:sine:Rsp:} is finite almost surely.
\end{proof}

\begin{lemma}\label{lem:reg}
Any $\WeylT$-valued weak solution $\bfX$ of \eqref{eq:DBM}
such that $\bbP(\bfX(t)\in\XRsp(\alpha,\rho,p))=1$, for all $t\geq 0$,
actually takes value in $\XRsp(\alpha,\rho,p)$.
In particular, any near-equilibrium solution of \eqref{eq:DBM}
is actually the $ \XRTsp(\alpha,1,p) $-valued solution given by Theorem~\ref{thm:DBM}.
\end{lemma}
\begin{proof}
Let $ \bfY:=u(\bfX) $.
Fixing arbitrary $ t\geq 0 $, by \eqref{eq:BesComp:sup} we have
$ \bfbarY(0,t) \leq \bfY(0) + \bfQ^{0,t} $.
With this and $ \bfY(0)\in\RTsp(p) $, by \eqref{eq:QLLN} we obtain $ \bfY\in\RTsp(p) $.

It now suffices to prove $ \bfY\in\YTsp(\alpha,\rho) $.
This, by Lemma~\ref{lem:den2YTsp}, amounts to proving
the bound \eqref{eq:dens} and $ \bfY\in\YTspl(\rho) $.
The latter, with $ \bfY $ being a weak solution of \eqref{eq:DBMg} satisfying $ \bbP(\bfY(t)\in\Yspl(\rho)) $, $\forall t\geq 0 $,
is proven by the continuity argument in proof of Lemma~\ref{lem:ZYsp}(\textit{ii}).
Turning to proving the bound \eqref{eq:dens} (recall that $ \A_{b,k} $ is defined as in \eqref{eq:A}), 
we partition $ [0,t] $ into equally spaced subintervals $ [t_{j-1},t_j] $, $ j=1,\ldots,j_* $,
each with length $ t/j_* \leq \tau $, where $ \tau $ is as in Proposition~\ref{prop:denss}.
Similar to \eqref{eq:avbd}, here we have
\begin{align*}
	\av_{A^{k}_b}(\bfY(t_j)) 	\geq
		\rho - \tfrac{2}{|\AL_{b,k}|} \norm{\bfY(t_j)} |\tilm^k_{|b|+1/2}|^{1-\alpha}.
\end{align*}
With $ \bfX(t_j)\in\XRsp(\alpha,\rho,p) $, 
by \eqref{eq:ALwbd}, the last term tends to zero as $ k\to\infty $,
so there exists $ K\in\bbZ_{>0} $ such that
$ \av_{\A_{b,k}}(\bfY(t_j)) \geq \tfrac{3\rho}{4} $,
$ \forall  j=1,\ldots,j_* $, $ b\in\hZ $ and $ k\geq K $.
Combining this with 
$
	\bfundY(t_{j-1},t_{j}) \geq \bfY(t_j) - \bfQ^{t_{j-1},t_j}
$
(by \eqref{eq:BesComp}), and \eqref{eq:denss:claim1} (choosing $ K $ larger if necessary),
we obtain \eqref{eq:dens}.
\end{proof}

\bibliographystyle{abbrv}
\bibliography{DBM}

\begin{thebibliography}{10}

\bibitem{anderson10}
G.~W. Anderson, A.~Guionnet, and O.~Zeitouni.
\newblock {\em An introduction to random matrices}.
\newblock Number 118. Cambridge University Press, 2010.

\bibitem{cepa1995}
E.~C\'{e}pa.
\newblock Equations diff{\'e}rentielles stochastiques multivoques.
\newblock {\em Lecture Notes in Mathematics}, pages 86--107, 1995.

\bibitem{corwin12}
I.~Corwin.
\newblock The {K}ardar--{P}arisi--{Z}hang equation and universality class.
\newblock {\em Random Matrices: Theory Appl.}, 1(01), 2012.

\bibitem{corwin14}
I.~Corwin and A.~Hammond.
\newblock {B}rownian {G}ibbs property for {A}iry line ensembles.
\newblock {\em Invent. Math.}, 195(2):441--508, 2014.

\bibitem{corwin15}
I.~Corwin and A.~Hammond.
\newblock {KPZ} line ensemble.
\newblock {\em Probab. Theory Relat. Fields}, 2015.

\bibitem{dyson62}
F.~J. Dyson.
\newblock A {B}rownian-motion model for the eigenvalues of a random matrix.
\newblock {\em J. Math. Phys.}, 3(6):1191--1198, 1962.

\bibitem{fritz87}
J.~Fritz.
\newblock Gradient dynamics of infinite point systems.
\newblock {\em Ann. Probab.}, 15(2):478--514, 1987.

\bibitem{gorin14a}
V.~Gorin and M.~Shkolnikov.
\newblock Interacting particle systems at the edge of multilevel {D}yson
  {B}rownian motions.
\newblock {\em arXiv preprint arXiv:1409.2016}, 2014.

\bibitem{gorin14}
V.~Gorin and M.~Shkolnikov.
\newblock Multilevel {D}yson {B}rownian motions via {J}ack polynomials.
\newblock {\em Probability Theory and Related Fields}, pages 1--51, 2014.

\bibitem{karatzas91}
I.~Karatzas and S.~Shreve.
\newblock {\em {B}rownian Motion and Stochastic Calculus (Graduate Texts in
  Mathematics) (Volume 113)}.
\newblock Springer, 1991.

\bibitem{katori09}
M.~Katori and H.~Tanemura.
\newblock Zeros of {A}iry function and relaxation process.
\newblock {\em Journal of Statistical Physics}, 136(6):1177--1204, 2009.

\bibitem{katori10}
M.~Katori and H.~Tanemura.
\newblock Non-equilibrium dynamics of {D}yson's model with an infinite number
  of particles.
\newblock {\em Comm. Math. Phys.}, 293(2):469--497, 2010.

\bibitem{lang77}
R.~Lang.
\newblock Unendlich-dimensionale {W}ienerprozesse mit wechselwirkung. {I}.
\newblock {\em Zeitschrift f{\"u}r Wahrscheinlichkeitstheorie und Verwandte
  Gebiete}, 38(1):55--72, 1977.

\bibitem{lang77a}
R.~Lang.
\newblock Unendlich-dimensionale {W}ienerprozesse mit wechselwirkung. {II}.
\newblock {\em Zeitschrift f{\"u}r Wahrscheinlichkeitstheorie und Verwandte
  Gebiete}, 39(4):277--299, 1977.

\bibitem{mckean69}
H.~P. McKean.
\newblock {\em Stochastic integrals}, volume 353.
\newblock American Mathematical Soc., 1969.

\bibitem{mehta04}
M.~L. Mehta.
\newblock {\em Random matrices}, volume 142.
\newblock Academic press, 2004.

\bibitem{osada12}
H.~Osada.
\newblock Infinite-dimensional stochastic differential equations related to
  random matrices.
\newblock {\em Probab. Theory Related Fields}, 153(3-4):471--509, 2012.

\bibitem{osada13}
H.~Osada.
\newblock Interacting {B}rownian motions in infinite dimensions with
  logarithmic interaction potentials.
\newblock {\em Ann. Probab.}, 41(1):1--49, 2013.

\bibitem{osada13a}
H.~Osada.
\newblock Interacting {B}rownian motions in infinite dimensions with
  logarithmic interaction potentials {II}: {A}iry random point field.
\newblock {\em Stochastic Processes and their applications}, 123(3):813--838,
  2013.

\bibitem{osada14a}
H.~Osada and H.~Tanemura.
\newblock Infinite-dimensional stochastic differential equations and tail $
  \sigma$-fields.
\newblock {\em arXiv preprint arXiv:1412.8674}, 2014.

\bibitem{osada14}
H.~Osada and H.~Tanemura.
\newblock Infinite-dimensional stochastic differential equations arising from
  {A}iry random point fields.
\newblock {\em arXiv preprint arXiv:1408.0632}, 2014.

\bibitem{osada14b}
H.~Osada and H.~Tanemura.
\newblock Strong {M}arkov property of determinantal processes with extended
  kernels.
\newblock {\em arXiv preprint arXiv:1412.8678}, 2014.

\bibitem{spohn87}
H.~Spohn.
\newblock Interacting {B}rownian particles: a study of {D}yson's model.
\newblock In {\em Hydrodynamic behavior and interacting particle systems},
  pages 151--179. Springer, 1987.

\bibitem{valko09}
B.~Valk{\'o} and B.~Vir{\'a}g.
\newblock Continuum limits of random matrices and the {B}rownian carousel.
\newblock {\em Invent. Math.}, 177(3):463--508, 2009.

\end{thebibliography}

\end{document}